\documentclass[11pt]{article}
\usepackage[a4paper]{geometry}
\usepackage{amsmath,amsfonts,amssymb,amsthm,graphics,amscd}
\usepackage[color=green]{todonotes}
\usepackage{graphicx}
\usepackage{fontenc}

\def\F{{\mathbb F}}

\def\R{\mathcal{R}}

\def\N{\mathcal{N}}

\def\B{\mathcal{B}}

\def\F{\mathcal{F}}
\def\bN{\mathbb{N}}
\def\bR{\mathbb{R}}
\def\R{\mathbb{R}}

\usepackage{hyperref}
\usepackage{amsmath}
\newtheorem{theorem}{Theorem}[section]
\newtheorem{prop}[theorem]{Proposition}
\newtheorem{lemma}[theorem]{Lemma}
\newtheorem{corollary}[theorem]{Corollary}

\newtheorem{remark}[theorem]{Remark}
\theoremstyle{definition}
\newtheorem{definition}[theorem]{Definition}

\DeclareTextSymbol{\dj}{T1}{158}
\usepackage{gensymb}
\usepackage{dsfont}
\usepackage{tikz}
\usetikzlibrary{arrows.meta}
\makeatother
\usepackage{listings}
\usepackage{fancyhdr}
\usepackage[utf8]{inputenc} 
\usepackage[T1]{fontenc}
\numberwithin{equation}{section}
\parskip \medskipamount
\newcommand{\pr}[1]{\mathbb{P}\!\left(#1\right)}
\newcommand{\E}[1]{\mathbb{E}\!\left[#1\right]}
\newcommand{\estart}[2]{\mathbb{E}_{#2}\!\left[#1\right]}
\newcommand{\prstart}[2]{\mathbb{P}_{#2}\!\left(#1\right)}
\newcommand{\prcond}[3]{\mathbb{P}_{#3}\!\left(#1\;\middle\vert\;#2\right)}
\newcommand{\econd}[2]{\mathbb{E}\!\left[#1\;\middle\vert\;#2\right]}
\newcommand{\escond}[3]{\mathbb{E}_{#3}\!\left[#1\;\middle\vert\;#2\right]}
\usepackage{xspace}
\newcommand{\Poincare}{Poincar\'e\xspace}
\newcommand{\random}{random\xspace}

\begin{document}

\title{\bf Cutoff for random walk on random graphs with a community structure}
\author{Jonathan Hermon
\thanks{
                University of British Columbia, Vancouver, CA. E-mail: {jhermon@math.ubc.ca}. }
        \and An{\dj}ela \v{S}arkovi\'c
        \thanks{University of Cambridge, Cambridge, UK. E-mail: {as2572@cam.ac.uk}}
\and Perla Sousi
\thanks{
University of Cambridge, Cambridge, UK. E-mail: {p.sousi@statslab.cam.ac.uk}}
}
\date{}
\maketitle

\begin{abstract}
We consider a variant of the configuration model with an embedded community structure and study the mixing properties of a simple random walk on it. Every vertex has an internal $\deg^{\text{int}}\geq 3$ and an outgoing $\deg^{\text{out}}$ number of half-edges. Given a stochastic matrix $Q$, we pick a random perfect matching of the half-edges subject to the constraint that each vertex $v$ has $\deg^{\text{int}}(v)$ neighbours inside its community and the proportion of outgoing half-edges from community~$i$ matched to a half-edge from community $j$ is $Q(i,j)$. Assuming the number of communities is constant and they all have comparable sizes, we prove the following dichotomy: simple random walk on the resulting graph exhibits cutoff if and only if the product of the Cheeger constant of~$Q$ times $\log n$ (where $n$ is the number of vertices) diverges.  

In~\cite{AnnasPaper}, Ben-Hamou established a dichotomy for cutoff for a non-backtracking random walk on a similar random graph model with $2$ communities. We prove that the same characterisation of cutoff holds for simple random walk. 
MSC 2010 \emph{subject classifications.} Primary 60F05, 60G50.
\end{abstract}


\section{Introduction}\label{secIntroduction}
In this paper we study the mixing time of a simple random walk on two different random graph models which are  generalisations of the configuration model incorporating a community structure. 

We first define the two random graph models that we will consider. The first one was defined by Anna Ben-Hamou in~\cite{AnnasPaper}.
\begin{definition}(2 communities model)\label{def:2CommounitiesRandomGraph}
Let $V$ be a set of vertices which is a disjoint union of sets $V_1$ and $V_2$ representing the two communities. Let $d:V\to \mathbb{N}\setminus\{0,1\}$ be a degree sequence specified in advance, such that $$\sum_{v\in V_i}d(v)=N_i,\,\,\,\,\,\,\,\, \mathrm{for}\,\,i\in\{0,1\},$$ where $N_1$ and $N_2$ are both even. Let $N=N_1+N_2,$ and let $p$ be a fixed even integer between $2$ and $\min\{N_1,N_2\}.$  We construct the model by first, for all $v\in V$, assigning $d(v)$ half-edges to the vertex~$v$. We then choose uniformly at random $p$ half-edges in each community and label them \emph{outgoing}. We label the rest of the half-edges \emph{internal}. Finally, the random graph is obtained by matching the outgoing half-edges in community $1$ with the outgoing half-edges in community $2$ uniformly at random and for $i\in \{1,2\}$ taking a uniform random perfect matching of the internal half-edges in  community $i$ (half-edges at vertices $v$ and $u$ which are matched form an edge between vertices $v$ and $u$). We write for $i\in \{1,2\}$, $\alpha_i=\frac{p}{N_i}$ and $n_i=|V_i|$ and set $\alpha=\alpha_1+\alpha_2$ and $n=n_1+n_2=|V|.$ 
\end{definition}

\begin{definition}[$m$-communities model]\label{def:kCommounitiesRandomGraph}
 We define a random graph $G_n$ consisting of $m\in \mathbb{N}$ communities as follows.
For $i\in\{1,2\ldots, m\}$ let $n_i\in \mathbb{N}$ and let $V_i$ be a set of vertices belonging to community $i$ with $|V_i|=n_i$. Denote $n=\sum_{i=1}^mn_i$ and $V=\bigcup_{i=1}^mV_i$. 
We are given a symmetric $m\times m$ matrix $E$ with non-negative integer entries and with all diagonal elements being even. 

Let $\deg^{\rm int}, \deg^{\rm out}:V\to \mathbb{N}_0$ be two fixed sequences satisfying that $\sum_{j\neq i}E_{i,j} = \sum_{v\in V_i}\deg^{\rm out}(v)$ and $E_{i,i}=\sum_{v\in V_i}\deg^{\rm int}(v)$ for all $i$. For every $i\in \{1,\ldots, m\}$ to each $v\in V_i$, we assign $\deg^{\rm int}(v)$ internal half-edges and $\deg^{\rm out}(v)$ outgoing half-edges.
We write $d(v)=\deg(v)= \deg^{\rm int}(v)+\deg^{\rm out}(v)$.

For $i\leq m$ we connect the internal half edges coming from vertices in community $i$ uniformly at random to each other. For every community $i$, out of $\sum_{j\neq i}E_{i,j}$ outgoing half-edges we pick $E_{i,j}$ half edges to connect to community $j\neq i$ uniformly at random (without replacement). We call the chosen edges $j$-half-edges. For $i\ne j$ we match the $j$-half-edges coming from community $i$ to the $i$-half-edges coming from community $j$ uniformly at random.

\end{definition}

We now recall the definition of mixing time for a Markov chain with transition matrix $P$ and invariant distribution $\pi$. The $\varepsilon$ mixing time is defined to be 
\[
t_{\rm mix}(\varepsilon) = \min\{ t\geq 0: \max_x\|P^t(x,\cdot) -\pi \|_{\rm{TV}}\leq \varepsilon\},
\]
where $\|\mu -\nu \|_{\rm TV} = \frac{1}{2}\sum_x|\mu(x)-\nu(x)|$ for $\mu$ and $\nu$ two probability measures. 

A sequence of Markov chains with mixing times $(t^{(n)}_{\rm mix}(\varepsilon))$ exhibits cutoff if for all $\varepsilon\in (0,1)$
\begin{align}\label{eq:defcutoff}
\frac{t^{(n)}_{\rm mix}(\varepsilon)}{t^{(n)}_{\rm mix}(1-\varepsilon)}\to 1 \quad \text{ as } n\to \infty.
\end{align}
We say that a sequence of graphs exhibits cutoff if the corresponding sequence of simple random walks exhibits cutoff. 
Let $G_n$ be a sequence of random graphs as in Definitions~\ref{def:2CommounitiesRandomGraph} and~\ref{def:kCommounitiesRandomGraph}. 
We say that an event $A$ happens with high probability and abbreviate it w.h.p., if $\pr{A}=1-o(1)$ as $n \to \infty$. We say that the sequence of random graphs $G_n$ exhibits cutoff with high probability if~\eqref{eq:defcutoff} holds in distribution. 

{\emph {Notation}}: Let $f,g:\mathbb{N}\to \R$ be two functions. We write $f(n)\lesssim g(n)$ if there is a constant $c>0$ such that for all $n$, $f(n)\le cg(n)$. We write $f(n)\asymp g(n)$, if we have that $f(n)\lesssim g(n)$ and $g(n)\lesssim f(n)$.

We first state our result for the $2$-communities model under the following assumptions as in~\cite{AnnasPaper}:
 \begin{align} \alpha\le 1 &&&&&&&&
 \mathrm{(there\, is\, a\, community\, structure)} \label{eq:ComStruct}\\
 N_1\asymp N_2\asymp N&&&&&&&&\mathrm{(communities\, have\,comparable\, sizes)} \label{eq:ComprableCom}\\ 
 \min_{v\in V}d(v)\ge 3&&&&&&&& \mathrm{(branching\, degree)} \label{eq:BranchDeg}\\
 \Delta=\max_{v\in V}d(v)=O(1)&&&&&&&&\mathrm{(sparse\, regime)}. \label{eq:Sparse}
 \end{align}

\begin{theorem}\label{thrm:Cutoff}\label{thrm:cutoffSimple2} Let $G_k=(V_k,E_k)$ be a sequence of the two communities random graphs on $n_k$ vertices, with $n_k\to \infty$ as $k\to \infty$, which satisfies the assumptions \eqref{eq:ComStruct}-\eqref{eq:Sparse}. For a simple random walk on $G_k$ the following holds. Let $t_{\mathrm{mix}}\left(G_k, \varepsilon\right)$ be $\varepsilon$ mixing time of the walk and let $\alpha_k$ be $\alpha$ from Definition~\ref{def:2CommounitiesRandomGraph} corresponding to the graph $G_k.$ For $\alpha_k\gg \frac{1}{\log |V_k|}$ we have that with high probability  the simple random walk on $G_k$ exhibits cutoff and $t_{\mathrm{mix}}\left(G_k,\frac{1}{4}\right)\asymp \log|V_k|.$

Moreover, in this case for all $\varepsilon\in(0,\frac{1}{2})$  there exists a constant $C(\varepsilon)$ such that with high probability
\[ t_{\mathrm{mix}}\left(G_k, \varepsilon\right)-t_{\mathrm{mix}}\left(G_k,1-\varepsilon\right)\le C( \varepsilon)\sqrt{\frac{ \log|V_k|}{\alpha_k}}.\]

Finally, for $\alpha_k\lesssim \frac{1}{\log |V_k|}$ we have that $t_{\mathrm{mix}}\left(G_k,\frac{1}{4}\right)\asymp\frac{1}{\alpha_k}$ and there is no cutoff. 
\end{theorem}

Ben-Hamou in~\cite{AnnasPaper} established the same dichotomy for cutoff for the non-backtracking random walk on the $2$-communities random graph. The result above answers her open question regarding the mixing of the simple random walk. In Section~\ref{sec:comparison} we establish that under certain assumptions the non-backtracking random walk mixes faster than the simple random walk in the cutoff regime. This gives a partial answer to a problem of comparing the mixing times in cutoff regime, which was suggested to us by Anna Ben-Hamou.

We now move on to the $m$-communities model. We first define a matrix $Q$ indexed by $\{1,\ldots, m\}^2$ via
\[
Q(i,j)=\frac{E_{i,j}}{\sum_{\ell}E_{i,\ell}}.
\]
In words, $Q$ is the transition matrix of a random walk on the graph obtained from $G_n$ by gluing together all vertices from the same community into a single vertex and keeping all edges.

For the $m$-communities model we make the following assumptions:
 \begin{align} 
 Q \text{ is an irreducible matrix } &&&&&&&&\mathrm{(connected \, \,graph)}\label{eq:mComStruct}\\ 
 \forall\ i\in\{1,\ldots m\},\,\, n_i\asymp n&&&&&&&&\mathrm{(communities\, have\,comparable\, size)} \label{eq:mComprableCom}\\ 
\min_{i\in\{1,\ldots m\}, v\in V_i}\deg^{\rm int}(v)\ge 3&&&&&&&& \mathrm{(branching\, degree\, inside\, the\, community)} \label{eq:mBranchDeg}\\
 \Delta=\max_{v\in V}d(v)=O(1)&&&&&&&&\mathrm{(sparse\, regime)} \label{eq:mSparse}
 \end{align}

We note that the 2-communities model cannot be obtained from the $m$-communities one by taking $m=2$. The main difference between the two models is that in the 2-communities model we choose at random which vertices will have edges to the other community. For the 2-communities model the internal degree of every vertex can be smaller than $3$.

We recall that the Cheeger constant of a transition matrix $P$ with invariant distribution $\pi$ is defined to be $$ \Phi_*=\min_{A:\pi(A)\le \frac{1}{2}}\frac{\sum_{x\in A,y\in A^c}\pi(x)P(x,y)}{\pi(A)}.$$
Since we assume that $Q$ is an irreducible matrix and it can be viewed as a random walk on a graph, it follows that it is reversible and it has a unique invariant distribution that we denote by $\pi_Q$. For every $i\in \{1,\ldots,m\}$ we have using the assumptions above
\[
\pi_Q(i) = \frac{\sum_\ell E_{i,\ell}}{\sum_{j,\ell}E_{j,\ell}}\geq \frac{3n_i}{\Delta n} \asymp 1.
\]

We can now state our main result for the $m$-communities random graph model.

\begin{theorem}\label{thrm:mCutoff}\label{thrm:cutoffSimplem} Let $G_k=(V_k,E_k)$ be a sequence of the $m$-communities random graphs on $n_k$ vertices, with $n_k\to \infty$ as $k\to \infty$, which satisfies the assumptions \eqref{eq:mComStruct}-\eqref{eq:mSparse}. Let $t_{\mathrm{mix}}\left(G_k, \varepsilon\right)$ be the $\varepsilon$-mixing time of the simple random walk on $G_k$ and let $\alpha_k$ be the Cheeger constant of the Markov chain~$Q$ corresponding to the graph $G_k.$ For $\alpha_k\gg \frac{1}{\log |V_k|}$ we have that with high probability the simple random walk on $G_k$ exhibits cutoff and $t_{\mathrm{mix}}\left(G_k,\frac{1}{4}\right)\asymp \log|V_k|.$

Moreover, in this case for all $\varepsilon\in(0,\frac{1}{2})$  there exists a constant $C(\varepsilon)$ such that with high probability
\[ t_{\mathrm{mix}}\left(G_k, \varepsilon\right)-t_{\mathrm{mix}}\left(G_k,1-\varepsilon\right)\le C(\varepsilon)\sqrt{\frac{ \log|V_k|}{\alpha_k}}.\] 
Finally, for $\alpha_k\lesssim \frac{1}{\log |V_k|}$, we have that with high probability $t_{\mathrm{mix}}\left(G_k,\frac{1}{4}\right)\asymp\frac{1}{\alpha_k}$ and there is no cutoff. \end{theorem}

  For an irreducible reversible Markov chain on a finite state space with transition matrix P the absolute spectral gap $\gamma$ is defined as $$\gamma_*=1-\max\left\{|\lambda|: \lambda \text{ is an eigenvalue of } P\mathrm{\, with\,} \lambda\neq 1\right\}$$ and the absolute relaxation time is $t_{\mathrm{rel}}^*=\frac{1}{\gamma_*}.$ The spectral gap is defined to be 
  $$\gamma=1-\max\left\{\lambda: \lambda \text{ is an eigenvalue of } P\mathrm{\, with\,} \lambda\neq 1\right\}$$ and the relaxation time is  $t_{\rm rel}=\frac{1}{\gamma}$.

 \begin{prop}\label{prop:TRel} In the setup of Theorems~\ref{thrm:Cutoff} and~\ref{thrm:mCutoff}, there exists a constant $C>0$ such that with high probability $ \frac{1}{C\alpha}\le t_{\mathrm{rel}}\le t_{\mathrm{rel}}^*\le \frac{C}{\alpha}.$ 
 \end{prop}
 
The constants $C(\varepsilon)$ from Theorems~\ref{thrm:Cutoff} and~\ref{thrm:mCutoff} and $C$ from Proposition~\ref{prop:TRel} depend on the maximal degree $\Delta$ and on the minimal ratio of sizes of communities  $\min_{i,j\le m}n_i/n_j$. 

\subsection{Related work}
In this paper we establish cutoff at an entropic time, which has been a recurring theme in a lot of recent works on random walks and non-backtracking random walks on random graphs. For more related work and references we refer the reader to~\cite{PerlasPaper}. Some notable works include~\cite{RWonRG, NBRWvsSRW, PerlasPaper, RWAbelian, DiaconisChatterjee, randomlifts, BordenaveLacoin, CaputoSalzeBordenave, VarjuEberhard}.

\subsection{Organisation}

In Section~\ref{sec:overview} we give a detailed overview of the proof ideas and direct the reader to the relevant sections where the rigorous proofs are presented. We first prove the statements of Theorems~\ref{thrm:Cutoff} and~\ref{thrm:mCutoff} in the case of a lazy simple random walk, so we work with a lazy walk in Sections~\ref{secLimitTree} to~\ref{sec5}. In Section~\ref{section:simplefromlazy} we deduce the simple random walk case from the lazy one. In Section~\ref{secLimitTree} we define a multi type \random tree and prove concentration results on the speed and entropy of a lazy simple random walk on it. In Section~\ref{sec4} we prove bounds on the spectral profile (defined there) of $G_n$ and prove the bounds on $t_{\rm{rel}}$ from Proposition \ref{prop:TRel} for certain values of $\alpha$. In Section~\ref{sec5} we describe a coupling of the walk on the \random tree with the walk on the random graph $G_n$ using similar ideas as in~\cite{PerlasPaper}, and then we prove Theorems~\ref{thrm:Cutoff} and~\ref{thrm:mCutoff}. In Section~\ref{sec:comparison} we establish a comparison result between the mixing times of the simple and the non-backtracking random walk  on the $2$-communities model. Finally, as part of the appendix we present the proofs of some combinatorial statements.

\section{Overview}\label{sec:overview}

We give an overview of the proof ideas in the case when $\alpha\gg 1/\log n$. In this regime, we establish cutoff at an entropic time. There have been many examples of random walks on random graphs exhibiting cutoff at an entropic time. In this work we follow the general outline of~\cite{RWonRG},~\cite{PerlasPaper} and~\cite{NBRWvsSRW}, where the idea is to first bound the $\mathcal{L}^2$ distance at the entropic time and then use the Poincar\'e inequality to bring it down to $o(1)$. The main obstacles we have to overcome are the estimation of the relaxation time and the entropic concentration on the Benjamini-Schramm limit \footnote{In practice, we work with a certain proxy of the Benjamini-Schramm limit, which is more refined and hence more closely related to our random graph. Namely, our random tree depends on $n$. Had we worked with the proper Benjamini-Schramm limit then whenever $\alpha = o(1)$ we would never have crossed community edges in the random tree. This is similar to previous works on random walks on random graphs, such as \cite{RWonRG} and \cite{PerlasPaper}.}  of our random graphs.We explain how we bound the relaxation time for the $2$-communities model in Section~\ref{sec:boundtreloverview} and the entropic concentration for the~$m$-communities model in Section~\ref{sec:mixingtypesoverview}.

We now give a short overview of how the rest of proof is organised.
We first establish the quantitative entropic concentration as in \cite{RWonRG}, \cite{PerlasPaper} and~\cite{NBRWvsSRW} for the Benjamini-Schramm limit of $G_n$ which is a multi type \random tree (see Definitions~\ref{def:MTGWTree} and~\ref{def:2TGWTree}). Loosely speaking, we prove that if $\mathfrak{h}$ is the Avez entropy of the \random tree $T$ (which exists a.s. and is a constant) and $\mu_t^{T}$ is the distribution of $X_t$ given $T$, then\footnote{We do not prove this statement exactly, but instead we prove it for the loop erasure of the random walk and it ends up being sufficient for completing the proof.}
\begin{align}\label{eq:entropicconcentrationweneed}\mathbb{E}\left[ \left(\mathbb{E}[-\log \mu_t^{T} (X_t) \mid T] - \mathfrak{h}t \right)^2\right] \lesssim  \frac{t}{\alpha}.\end{align} 
We explain how to do this in the case of the $m$-communities model in Section~\ref{sec:mixingtypesoverview}.

The obtained quantitative entropic concentration is sufficiently strong in order to establish cutoff with high probability starting from a fixed sequence of starting points (by this we mean that in the definition of ``cutoff with high probability'' we do not consider the worst case starting points, but rather we consider the sequence of the ratios of the $\varepsilon$ and the $1-\varepsilon$ mixing times corresponding to this particular sequence of starting points, and require it to converge in distribution to 1 for all fixed $\varepsilon \in (0,1)$).
Namely, consider the time $t_0$ at which the annealed entropy of the random walk on the \random tree equals $\log n - C_{\varepsilon} \sqrt{\log n/\alpha}$. By the entropic concentration~\eqref{eq:entropicconcentrationweneed}, one can show that with probability at least $1-\varepsilon$ (jointly over $T$ and $(X_t)_{t \ge 0}$)
\[
\mu_{t_0}^{T}(X_{t_{0}})\cdot n \in \left[ \exp \left(\frac{1}{2}C_{\varepsilon} \sqrt{\frac{\log n}{\alpha}} \right) ,  \exp \left(\frac{3}{2}C_{\varepsilon} \sqrt{\frac{\log n}{\alpha} }\right)\right].
\]
In particular, on an event holding with probability at least $1-o(1)$ the random walk is supported on a set of cardinality $o(n)$.
As in  \cite{RWonRG} (as well as~\cite{NBRWvsSRW} and~\cite{PerlasPaper} which followed) we construct a coupling of the random graph rooted at the starting point of the walk, together with the random walk on this graph until time $t_0$, and the rooted \random tree together with the walk on it until time $t_0$. The coupling contains in particular a common ``good'' subtree which, loosely speaking, is contained in the set of vertices $x$ of the tree satisfying that 
\begin{align}\label{eq:goodsubtreeoverview}\mathrm{max}_{t \le t_0}\mu_{t}^{T}(x) \le \frac{1}{n}\cdot  \exp \left(\frac{3}{2}C_{\varepsilon} \sqrt{\frac{\log n}{\alpha}}\right)\end{align} and is such that with probability at least $1-2 \varepsilon$ (w.r.t.\ the aforementioned coupling, jointly over the tree, finite rooted random graph and the walks on both graphs) the event that walks on both graphs stay in this ``good'' subtree and are equal to one another until time $t_0$ holds. On this event the support of the walk on the finite random graph is of size $o(n)$ and hence it has total variation distance at least~$1-o(1)$ from the stationary distribution. By the triangle inequality, this easily implies that the $1-2\varepsilon-o(1)$ mixing time for this starting point is w.h.p.\ at least $t_0$. Conversely,~\eqref{eq:goodsubtreeoverview} together with the Poincar\'e inequality (i.e. exponential decay of the $\ell^2$ distance), the bound on the relaxation time (which we explain how to obtain in Section~\ref{sec:boundtreloverview}) and the spectral profile technique (which is a more refined version of the Poincar\'e inequality, see~\cite{SpectralProfilePaper} for details) allow us to show that with high probability the random walk on the finite graph starting from the considered starting point is well-mixed at time $t_0 + O(C_{\varepsilon} \sqrt{\frac{\log n}{\alpha}})=t_0+o(t_0)$. Here we are using the fact that we are considering a lazy simple random walk, as that is required in order to use the spectral profile technique. Finally to upgrade this result on cutoff from a typical starting point into one about worst-case starting point we follow the approach of~\cite{NBRWvsSRW} by considering $K$-roots which we define in Section~\ref{sec5}.

\subsection{Entropic concentration} \label{sec:mixingtypesoverview}

In Section~\ref{secLimitTree} we establish the existence of the speed and the Avez entropy, as well as prove a quantitative entropic concentration for the random walk on the limiting multi type \random tree. 
Following \cite{RWonRG} it is natural to define a regeneration time to be a time at which the random walk on the \random tree crosses an edge for the first and last time. Let $\sigma_i$ be the time of the $i$-th regeneration time. In the setup of~\cite{PerlasPaper} the regeneration times give rise to an i.i.d.\ decomposition of the pair $(T_{X_{\sigma_1}}, (X_t)_{t=\sigma_1}^{\infty} )$ into pairs~$(T_{X_{\sigma_i}} \setminus T_{X_{\sigma_{i+1}}}, (X_t)_{t=\sigma_i}^{\sigma_{i+1}-1}  )_{i \in \mathbb{N}}$, where~$T_a \setminus T_b$ is the subtree obtained by removing the induced subtree rooted at $b$ from the induced subtree rooted at $a$. In our setup we call the $i$-th regeneration time ``type $j \in [m]$'' if the label of $X_{\sigma_{i}}$ is $j$. The above i.i.d. decomposition from~\cite{PerlasPaper} can now be replaced with a ``Markov chain decomposition''. Namely, the distribution of $(T_{X_{\sigma_{i}}} \setminus T_{X_{\sigma_{i+1}}}, (X_t)_{t=\sigma_{i}}^{\sigma_{i+1}-1}  )$ depends on $(T_{X_{\sigma_{i-1}}} \setminus T_{X_{\sigma_{i}}}, (X_t)_{t=\sigma_{i-1}}^{\sigma_{i}-1}  )$ only through the label of $X_{\sigma_{i}-1}$ and $X_{\sigma_{i}}$. We note here that it can easily be seen by considering for instance the case in which there are three consecutive regeneration times, during which the label of the position of the walk changes twice, that in the case of the $m$ communities model, the decomposition cannot be made by only considering the labels of $X_{\sigma_i}$, but we also need the label of its parent, $X_{\sigma_{i}-1}$.

In order to establish the quantitative entropic concentration from~\eqref{eq:entropicconcentrationweneed} we exploit the aforementioned Markov chain decomposition and apply a certain general result about concentration for averages for stationary Markov chains which is also valid in the non-reversible setup (we do this in Lemma~\ref{lemma:Decorrelation}) involving the mixing time of the chain. In our setup we apply this to the non-reversible Markov chain with transition matrix $\Sigma$, which loosely speaking represents the transitions of types of the regeneration edges $(X_{\sigma_i-1},X_{\sigma_i})$, where $\sigma_i$ only correspond to certain ``special'' regeneration times (see Definition~\ref{def:RegTime}). Writing  $t_{\mathrm{mix}}(\Sigma)=t_{\mathrm{mix}}(\Sigma,1/4)$ for the total variation mixing time of $\Sigma$ and~$\mathfrak{h}$ for the Avez entropy of the \random tree $T$, we essentially obtain 
\[\mathbb{E}\left[ \left(\mathbb{E}[-\log \mu_t^{T} (X_t) \mid T] - \mathfrak{h}t \right)^2\right] \lesssim t\cdot  t_{\mathrm{mix}}(\Sigma).\] 
Therefore, by~\eqref{eq:entropicconcentrationweneed} it is enough to show that mixing time of $\Sigma$ can be bounded  from above (up to constants) by the inverse of  $\alpha$. This is easy to establish in the case of the two communities model and we explain here the difficulties which arise for the $m$-communities one.

The lack of reversibility of $\Sigma$, together with the fact that some of its transition probabilities may tend to zero poses several challenges, when attempting to show that the product of the mixing time of $\Sigma$ with its \Poincare constant (or with its Cheeger constant) is bounded, as well as when trying to compare its \Poincare or Cheeger constant to those of a different Markov chain (such as $Q$ or $Q_2$ described below, or see Definition~\ref{def:chainQ2}). 
Another difficulty is posed by the fact that $\Sigma$ and $Q$ are defined on different state spaces, hence it is not possible to directly compare their Cheeger or \Poincare constants, not to mention their mixing times. These are  generally harder to compare between two Markov chains, even when they are defined on the same probability space, are both reversible and one is obtained from the other by a bounded perturbation of the edge weights (see~\cite{Dingperessensitivity} and~\cite{Hermonsensitivityofmixing} for constructions showing that the mixing times can be of completely different order).

\subsection{Bounding $t_{\rm rel}$}\label{sec:boundtreloverview}

We explain how to bound the spectral gap in the $2$-communities model. It is easy to see that the bottleneck-ratio of a community of stationary probability at most $1/2$ is $\asymp \alpha$, and hence this immediately implies that the spectral gap is $\lesssim \alpha$. We are unaware of any general criterion which allows one to deduce that the relaxation time is comparable to the inverse of the Cheeger constant. This is often achieved by a comparison to another Markov chain which is better understood or by showing that both of them are comparable to the mixing time, which fails in our case when $\alpha \gg 1/\log n$.

It is tempting to apply the decomposition technique (which we can use for the $m$-communities model as the minimal internal degree is $3$, see \cite[Theorem~1.1]{madrasrandall} i.e.\ Theorem~\ref{thm:madrasrandall}). According to this technique, we can bound the relaxation time from above by the maximum of the relaxation times of the induced graphs on the two communities (this general technique involves taking the maximum also with the two state Markov chain obtained by contracting each community to a single state, but this is easily seen to be of order $1/\alpha$). We now explain why this technique fails in our case. Since the internal degree can be $2$ or $1$ it is not true that the restriction of the graph to each community is an expander. In fact, similarly to the giant component of a supercritical Erd\"os-R\'enyi graph (or a supercritical configuration model which has either a constant fraction of its vertices having degrees 1 or 2, or that the fraction of such vertices is taken to vanish arbitrarily slowly) when $\alpha \asymp 1$ (resp., when $\alpha$ vanishes slowly), one can show that the restriction to each community contains paths of length $\Theta(\log n)$ of degree 2 vertices (resp. contains such paths of length $o(\log n)$ but whose order can be made to be arbitrarily close to $\log n$), and so the relaxation time of the restriction is $\Omega \left( \left(\log n \right)^2 \right)$ (resp. of order which is arbitrarily close to this).

In order to circumvent the aforementioned difficulty in applying the usual decomposition theorem, we bound instead the Dirichlet eigenvalue of all sets of size at most $1/2$. In fact it is enough to bound it for sets $D$ containing community $1$ and a small proportion of community $2$ (or the other way around). We achieve this by bounding the tail of the time it takes to leave the set  $D$ by the Dirichlet eigenvalue of the second community and the probability that the chain starting from the first community spends smaller than $\delta$ proportion of time in community $2$ in the first $t$ steps, for some small constant $\delta$ and all large $t$. In order to bound the last probability, we construct a coupling between the random walk on the random graph and the random walk on the \random tree, in which new vertices of the random graph are revealed only when the random walk crosses an edge leading to them for the first time.



\section{Limiting Multi Type \random Tree}\label{secLimitTree}
In this section we only consider the lazy random walk. 
We start by defining a multi-type \random tree, which will be used to approximate the random graph $G_n$ consisting of $m$ communities. 
\begin{definition} \label{def:MTGWTree}
Let $\nu$ be a probability distribution on $\{1,\ldots, m\}.$
We define a multi-type \random tree~$T$ with $m$ types,  $1,2\ldots m$, and the distribution of the type of the root given by $\nu$, to be the infinite tree constructed as follows. The root $\rho$ is first chosen to be of type $i\in \{1,2,\ldots, m\}$ with probability $\nu(i)$. The root is then assigned the internal and outgoing degree of the vertex $v\in V_i$ with probability $\frac{\deg(v)}{\sum_{u\in V_i}\deg(u)}$. Each internal edge has an offspring of type $i$ and each outgoing edge connects to an offspring of type $j$, for $j\ne i$, independently with probability $\frac{E_{i,j}}{\sum_{k\ne i}E_{i,k}}.$  We now define the tree inductively. 

Suppose $T$ has been constructed up to level $\ell$ and we know the degree vectors of all nodes up to level $\ell-1$ as well as the types of all nodes of level $\ell$. To each node of level $\ell$ of type $i$ which is an offspring of a type $j\ne i$ node from level $\ell-1$, we assign the degree vector of vertex~$v$ in community $i$ of $G_n$ with probability $\frac{\deg^{\text{out}}(v)}{\sum_{k\ne i}E_{i,k}}$. The number of offspring of type $i$ of this node is then $\deg^{\text{int}}(v)$, while the number of outgoing offspring is $\deg^{\text{out}}(v)-1$ and each of the outgoing offspring is of type $s\ne i$ with probability $\frac{E_{i,s}}{\sum_{k\ne i}E_{i,k}}.$ For a node of level $\ell$ of type $i$ which is an offspring of a type $i$ node from level $\ell-1$, we assign the degree vector of vertex~$v$ in community $i$ of $G_n$ with probability $\frac{\deg^{\text{int}}(v)}{E_{i,i}}$. The number of offspring of type $i$ of this node is then $\deg^{\text{int}}(v)-1$, while the number of outgoing offspring is $\deg^{\text{out}}(v)$ and each of the outgoing offspring is of type $s\ne i$ with probability $\frac{E_{i,s}}{\sum_{k\ne i}E_{i,k}}.$

We let $\mathcal{T}$ be the topological space of all rooted bounded degree infinite trees and with vertices of types $\{1,\ldots, m\}$. The multi-type \random tree is then a random variable taking values in $\mathcal{T}$. We denote by $\text{MGW}_\nu$ the law of this random variable. We also let $\Theta:T\to \{1,2,\ldots, m\}$ be the function which maps each node of the tree to its type. 
\end{definition}

In the case of the random graph model on two communities from Definition \ref{def:2CommounitiesRandomGraph} the limiting tree is defined as follows.

\begin{definition} \label{def:2TGWTree}
We define a two-type \random tree $T$, with vertices of types $1$ and $2$ to be the infinite tree constructed as follows. The root $\rho$ is first chosen to be of type $1$ with probability $\frac{n_1}{n}$ and of  type $2$, otherwise (so with probability $\frac{n_2}{n}$). Each vertex of type $i$ is then assigned a random number of offspring by sampling its degree from a degree biased distribution of all degrees of vertices of type $i$ in $G_n$, in particular the degree is taken to be $d$ with probability $ \frac{d\sum_{v\in V_i}\mathds{1}(\deg(v)=d)}{\sum_{v\in V_i}\deg(v)}$   (degrees are between $3$ and $\Delta$). 
Each offspring vertex is then, independently of everything else, taken to be of the opposite type with probability $\alpha_i=\frac{p}{N_i}$ and of the same type $i$, otherwise.  We let $\mathcal{T}$ be the topological space of all rooted bounded degree infinite trees and with vertices of types $1$ and $2$. The two type \random tree is then a random variable taking values in $\mathcal{T}$. We also let $\Theta:T\to \{1,2\}$ be the function which maps each vertex to its type. 
\end{definition}

We let $d\left(x,y\right)$ be the graph distance between vertices $x$ and $y$. For a tree $T$ we denote by $T\left(v\right)$ the subtree of $T$ rooted at $v$, i.e. $T\left(v\right)=\left\{y\in T: d\left(\rho,y\right)=d\left(\rho,v\right)+d\left(v,y\right)\right\}$, and all $y\in T\left(v\right)$ are called offspring or descendants of $v$.  We write $\left\{x,y\right\}$ and $\left(x,y\right)$ for the undirected and directed edge in $T$, respectively. Let $p\left(x\right)$ be the parent of $x$, i.e. $p\left(x\right)$ is a vertex such that $d\left(x,p\left(x\right)\right)=1$ and $d\left(\rho,p\left(x\right)\right)+1=d\left(\rho,x\right)$. 
\\ For a discrete time Markov Chain $X$ we define the hitting time of vertex $x$ by $\tau_x=\min\left\{t\ge 0:  X_t=x\right\}$ and the first return time to $x$ as $\tau_x^+=\min\left\{t\ge 1: X_t=x\right\}$.

\begin{definition} \label{def:RegTime}
Let $X$ be a lazy simple random walk on an infinite tree $T$ with vertices of types~$1,\ldots, m$ starting from the root. A random time $\sigma$ is called a regeneration time if the random walk crosses the edge $\left\{X_{\sigma-1},X_{\sigma}\right\}$ for the first and last time at time~$\sigma$ and if the nodes $p(X_{\sigma-1}), X_{\sigma-1}, X_\sigma$ all have the same type. 
We say that the regeneration time $\sigma$ is of type $(i,i)$ for $i\in \{1,\ldots, m\}$, if $\Theta(X_\sigma)=i.$ \end{definition}

We note that the definition of regeneration times given above is crucial for the $m$-communities model. Indeed, it is used in the proof of Lemma~\ref{lemma:PiOfW} in order to upper bound $\pi_W$ and also in the proof of Proposition~\ref{thrm:MixingOfW}).

We call a graph uniformly transient if there exists a positive constant $c$ so that for all $x$
\[
\prstart{\tau_x^+=\infty}{x} \geq c.
\]
We start by stating the following well known result that a tree with vertex degrees lower bounded by $3$ is uniformly transient. Its proof follows by considering the distance of the walk from the root of the tree and stochastically dominating from below by a biased random walk on $\mathbb{Z}_+$.

  \begin{lemma} \label{lemma:ReturnProb}
There exist positive constants $c,c_1$ and $c_2$ so that the following holds. Let $T$ be an infinite tree with root $\rho$ and $\deg(v)\geq 3$ for all $v\in T$. Suppose that $X$ is a lazy simple random walk on $T$ and let $\mathbb{P}_x$ be the law of $X$ when it starts from $x$. Then for all $x\in T$ and all times~$t$ we have 
\begin{align*}
\prstart{\tau_{p\left(x\right)} \wedge \tau_x^+=\infty}{x}
\ge c \quad \text{ and } \quad \prstart{d\left(\rho,X_{t}\right)\le c_1t}{\rho}\le e^{-c_2t}.
\end{align*}
\end{lemma} 

In particular, the lemma above implies that the multi-type tree $T$ constructed as in Definition~\ref{def:MTGWTree} or~\ref{def:2TGWTree} is uniformly transient.

\begin{remark}
        \rm{
        Let $T$ be a multi-type \random tree as in Definition~\ref{def:MTGWTree} and let $X$ be a lazy simple random walk on $T$ started from the root. Then using the lemma above, it is easy to see that the probability that $X$ visits a vertex of type $i$ for the first time, then jumps to an offspring of type $i$, makes one more jump to an offspring of type $i$ and afterwards escapes forever is bounded from below by $c'$, where $c'$ is a positive constant. Therefore, almost surely there exists an infinite sequence of regeneration times that we denote by $(\sigma_k)_{k\in \mathbb{N}}$.
        }
\end{remark}

\begin{definition}\label{def:KBall}
For any graph $G=(V,E)$, vertex $x\in V$ and $K>0,$ we define a ball of radius $K$ around $x$ as $\mathcal{B}_K(x)=\mathcal{B}^G_K(x)=\left\{y\in V: d(x,y)\le K\right\}$ where for $x,y\in V,$ $d(x,y)$ is the graph distance of $x$ and $y$. We also define boundary of the ball as $\partial \mathcal{B}_K(x)=\mathcal{B}_K(x)\setminus \mathcal{B}_{K-1}(x).$
\end{definition}

The proof of the following lemma follows analogously to \cite[Lemma~3.6]{PerlasPaper} with the only difference being that here we need to condition on the type of the regeneration edge being crossed.

\begin{lemma}\label{lemma:RegIndep} 
Let $T$ be a multi-type \random tree with root $\rho$ and $m$ types as in Definition~\ref{def:MTGWTree} or~\ref{def:2TGWTree}. Fix $K\ge 0$ and let $T_0$ be a realisation of the first $K$ levels of $T$. Let $X$ be a lazy simple random walk on~$T$ started from $\rho$. Let $\sigma_0$ be the first time that $X$ reaches $\partial \mathcal{B}_K\left(\rho\right)$ and let  $\left(\sigma_i\right)_{i=1}^\infty$ be the  almost surely infinite sequence of regeneration times occurring after $\sigma_0$, where $\sigma_i$ is the $i$-th regeneration time for which ${\phi_i=d\left(\rho,X_{\sigma_i}\right)>K}$. Then, for $\theta\in \left\{1,\ldots, m\right\}$, conditional on $\left(\Theta\left(X_{\sigma_i-1}\right),\Theta\left(X_{\sigma_i}\right)\right)=(\theta,\theta)$ and $\mathcal{B}\left(\rho,K\right)=T_0,$ $\left(T\left(X_{\sigma_{i-1}}\right)\backslash T\left(X_{\sigma_{i}}\right),\left(X_t\right)_{\sigma_{i-1}\le t\le \sigma_{i}}\right)$ and $\left(T\left(X_{\sigma_i}\right)\backslash T\left(X_{\sigma_{i+1}}\right),\left(X_t\right)_{\sigma_i\le t\le \sigma_{i+1}}\right)$ are independent. Moreover, $\left(\sigma_i-\sigma_{i-1}\right)_{i\geq 1}$ and $\left(\phi_i-\phi_{i-1}\right)_{i\geq 1}$ have exponential tails. Let $T^a$ be the graph obtained by taking $T(v)\cup{\rho}$, for a random offspring~$v$ of $\rho$. Then, for all $i\geq 1$ and $\theta\in \left\{1\ldots, m\right\}$ conditional on $\left(\Theta\left(X_{\sigma_i-1}\right),\Theta\left(X_{\sigma_i}\right)\right)=\left(\Theta\left(\rho\right),\Theta(v)\right)=(\theta, \theta)$,  the pair $\left(T\left(X_{\sigma_i}\right), \left(X_t\right)_{t\ge \sigma_i}\right)$ has the law of $\left(T(v),\widetilde{X }\right)$, where $\widetilde{X}$ is a lazy simple random walk on $T^a$ started from $v$ conditional to never visit $\rho$.
\end{lemma}

From now on, slightly abusing notation, for an edge $e=(x,y)$ we will write $\Theta(e)=(\Theta(x),\Theta(y))$. Also for a regeneration time $\sigma$ we will write $\Theta(X_\sigma)=\Theta(X_{\sigma-1},X_\sigma)$. 

The above independence and stationarity conditional on the type of the regeneration edge gives that we can define the Markov chain corresponding to the types of regeneration edges as follows.

\begin{definition} \label{def:MCRegTypes}
Let $(\sigma_k)$ be the regeneration times as in Lemma~\ref{lemma:RegIndep}. We define the Markov chain~$\Sigma$  on the state space of types of regeneration edges, i.e.\ on $S= \{(i,i): i\in 
\left\{1, \ldots, m\right\}\}$, which, for $\theta_1,\theta_2\in S$, has the transition probabilities $$\Sigma(\theta_1,\theta_2)=\prcond{\Theta\left(X_{\sigma_2}\right)=\theta_2}{\Theta\left(X_{\sigma_1}\right)=\theta_1}{}.$$ 
By the assumption on $Q$ being irreducible, it follows that $\Sigma$ is an irreducible Markov chain, and since it takes values in a finite state space, it has a unique invariant distribution that we denote by~$\pi_{\Sigma}$. We call $\Sigma$ the types Markov chain.
\end{definition}

We can further define $\mathcal{Z}_{\theta_1,\theta_2}$ to be the law of the pair $(T(X_{\sigma_2})\setminus T(X_{\sigma_3}), (X_t)_{\sigma_2\leq t\leq \sigma_3})$ conditional on the types of the edges $(X_{\sigma_2-1},X_{\sigma_2})$ and $(X_{\sigma_3-1},X_{\sigma_3})$ being $\theta_1$ and $\theta_2$ respectively. 

The Markov chain $\Sigma$ defined above represents the types of the regeneration edges. Given the sequence of types of the regeneration edges, which we sample from this Markov chain, we can subsequently generate the tree and the walk between the two regeneration times, by sampling them from the distributions $\mathcal{Z}_{\Theta\left(X_{\sigma_i}\right), \Theta\left(X_{\sigma_{i+1}} \right)}.$ In the next section we will analyse the mixing time of the Markov chain $\Sigma$. 

\subsection{Mixing time of the types Markov chain}\label{secMixingW}
In this section we estimate the mixing time of the Markov chain $\Sigma$. 
We start with the case of two communities, where it is easy to estimate the mixing time. The rest of the section is then devoted to the multi community setting. 

\begin{lemma}\label{lemma:2comMixTypes} In the two communities model, there exists a constant $C$ depending only on the maximal degree $\Delta$ and $ \frac{n_1}{n_2}$  such that  
$$t_{{\mathrm{mix}}}^{{\mathrm{TV}}, \Sigma}\le \frac{C}{\alpha}.$$ 
\end{lemma}
\begin{proof} 
As the tree has only two types of vertices, the Markov chain $\Sigma$ has state space $\{(1,1), (2,2)\}$, of size two. We claim that it suffices to show that the transition probabilities of $\Sigma$ are all bounded from below by $c\alpha$, for some constant $c$. Indeed, then the hitting time of any state will be stochastically dominated by a geometric random variable of parameter $\gtrsim \alpha$, which implies the bound on the mixing time. 
Using Lemma~\ref{lemma:ReturnProb} we see that the probability to cross from a new vertex $v$  of type $i$ to an offspring of type $j$, then make two steps to two new offsprings of type $j$ and then escape to infinity is bounded from below by $\gtrsim \alpha$ if $i\ne j$ and by $\gtrsim 1$ if $i=j$. Conditioning on never returning to $v$ just increases this probability.
\end{proof}

The main result of this section is the following proposition which bounds the mixing time of $\Sigma$ in terms of the Cheeger constant~$\Phi_*^{Q}$ of the matrix $Q$.

\begin{prop} \label{thrm:MixingOfW} In the $m$-communities model, there exists a constant $C$ depending only on the number of communities $m$, the maximal degree $\Delta$ and $\min_{i,j\le m} \frac{n_i}{n_j}$  such that  
$$t_{{\mathrm{mix}}}^{{\mathrm{TV}}, \Sigma}\le \frac{C}{\Phi_*^{Q}}.$$ 
\end{prop}

Before we give an overview of the proof of this proposition, we recall the definition of the \Poincare constant~$\gamma$ of a Markov chain with transition matrix $P$, state space $S$ and invariant distribution $\pi$ as 
$$\gamma= \inf_{\phi \text{ non-constant}} \frac{\mathcal{E}_{P}(\phi)}{\text{Var}_{\pi}(\phi)},$$ where for any function $f$ on $S$, we define the Dirichlet form of $f$, $\mathcal{E}_P(f)$ as $$ \mathcal{E}_{P}(f)= \langle (I-P)f, f\rangle_\pi=\sum_{x} ((I-P)f(x)) f(x) \pi(x).$$

\begin{definition}
\label{def:WMC} We define a sequence of times $(\widetilde{\sigma}_k)_{k\in \bN}$ to be the times when the walk crosses the edge $\{X_{\widetilde{\sigma}_k-1},X_{\widetilde{\sigma}_k}\}$ for the first and last time at time~$\widetilde{\sigma}_k$ and if the type of~$X_{\widetilde{\sigma}_k-1}$ and its parent in the tree are the same. Lemma \ref{lemma:RegIndep} also holds for this sequence of times and therefore we can define the Markov chain~$W$ to be the chain on the state space $\{1,\ldots, m\}^2$ which has transitions $$W(\theta_1,\theta_2)=\prcond{\Theta\left(X_{\widetilde{\sigma}_2}\right)=\theta_2}{ \Theta\left(X_{\widetilde{\sigma}_1}\right)=\theta_1}{},$$ where~$\theta_1,\theta_2\in\{1,\ldots, m\}^2$. We denote the invariant distribution of $W$ (which clearly exists by irreducibility of $Q$) by $\pi_W$.
\end{definition}

\begin{definition}\label{def:chainQ2}
We let $Q_2$ be the transition matrix given by  $$Q_2((i,j),(j,k))=Q(j,k)=\frac{E_{j,k}}{\sum_{\ell}E_{j,\ell}}.$$ We write $\pi_{Q_2}$ for the invariant distribution of $Q_2$ given by $\pi_{Q_2}(i,j)=\pi_Q(i)Q(i,j)$ for all~$i,j\in \{1,\ldots,m\}$.
\end{definition}

In order to control the mixing time of~$\Sigma$, we need to bound the \Poincare constant of~$\Sigma$ by $\Phi_*^Q$ and also show that the invariant distribution of $\Sigma$ has entries bounded away from $0$ uniformly. The chain $\Sigma$ can be viewed as the induced chain of $W$ when it visits the set $\{(i,i): i \in \{1,\ldots,m\}$, and hence we can bound its \Poincare constant by the one of $W$.  We then bound the \Poincare constant of $W$ by the one of $Q_2$ by showing that the transition probabilities of $W$ are lower bounded by the corresponding ones for $Q_2$, and that $W$ and $Q_2$ have comparable invariant distributions. Finally, we need to bound the \Poincare constant of~$Q_2$ by $\Phi_*^Q$.

We start by stating the following result on the comparison between the \Poincare constants of a chain $Z$ and the induced chain of $Z$ observed when it visits a set $A$.

\begin{lemma}
\label{lemma:InducedPoincare} Let $Z$ be a Markov chain on the state space~$S$ and let~$A\subset S.$ Let $\widetilde{Z}$ be a Markov chain on the state space $A$ with transition probabilities $P_A$ given by $$P_A(x,y)=\prstart{Z_{\tau^+_A}=y}{x}\quad  \text{ for } x,y\in A,$$ where $\tau^+_A$ is the first return time to $A$. Let~$\gamma$ and~$\gamma_A$ be the \Poincare constants of~$Z$ and~$\widetilde{Z}$, respectively. We have that $$\gamma_A\ge \gamma. $$
\end{lemma}

\begin{proof}
        The proof of this lemma in the reversible case is given in~\cite[Theorem~13.20]{MixingBook} and is due to Aldous. The only place where reversibility is being used is to show that the invariant distribution of $\widetilde{Z}$ is $\pi\vert_A$ given by $\pi\vert_A(x) = \pi(x)/\pi(A)$ for $x\in A$. This is also true in the non-reversible case, by considering the limit of the time each chain spends in any state $x\in A,$ rather than using the detailed balance equations. The  rest of the proof goes through verbatim in the non-reversible case. 
\end{proof}

\begin{corollary}\label{cor:PoincareWSigma} Let $\gamma_\Sigma$ and $\gamma_W$ be the \Poincare constants of the chains~$\Sigma$ and~$W$ respectively. Then $$\gamma_\Sigma\ge \gamma_W. $$
\end{corollary}

\begin{proof}

From the definition, as the regeneration times $(\sigma_k)_{k\in \bN}$ given by Lemma~\ref{lemma:RegIndep} are a subsequence of  $(\widetilde{\sigma}_k)_{k\in \bN}$, it is clear that we can construct the chain~$\Sigma$ by running the chain~$W$ and only keeping the steps of it which are in the set $\{(i,i):i\in \{1,\ldots, m\}\}.$ Therefore, the chain~$\Sigma$ is the induced chain of $W$ on the set $\{(i,i):i\in \{1,\ldots, m\}\},$ and hence the statement of the corollary follows from Lemma~\ref{lemma:InducedPoincare}.
\end{proof}

\begin{lemma}\label{lemma:PoincareWQ2} 
There exists a positive constant $c$ depending on $m,\Delta,\min_{i,j}n_i/n_j$ so that 
$$\gamma_{Q_2}\leq c\cdot \gamma_{W}. $$
\end{lemma}

To prove this we first need to compare the invariant distributions and the non-zero transition probabilities of~$W$ and $Q_2$.

\begin{lemma} \label{lemma:PiOfW} 
There exist positive constants $c, c_1$ and $c_2$ depending on $\min_{i,j}n_i/n_j,$ $m, \Delta$ so that the following holds. For all $i_1,i_2,j_1,j_2\in \{1,\ldots,m\}$
\begin{align*}
&cQ_2((i_1,j_1),(i_2,j_2))\le W((i_1,j_1),(i_2,j_2)) \quad \text{ and } \\& c_1\pi_{W}((i_1,j_1))\le \pi_{Q_2}((i_1,j_1))\le c_2\pi_{W}((i_1,j_1)).
\end{align*}
\end{lemma}

We prove this lemma by first establishing that the expected proportion of times the walk visits a vertex of type $i$, for $i\in \{1,\ldots,m\}$, for the first time is strictly positive. In order to do this, we will need the following result whose proof is given in Appendix~\ref{appendixA} and closely follows the proof from \cite{GWTSpeedHarmMeasure}.

\begin{theorem}\label{claim:StationaryMGW} Let $p_{\mathrm{SRW}}$ be the transition probability on $\mathcal{T}$ which moves the rooted tree $T$ to the rooted tree $T'$, obtained by picking a random neighbour of the root of $T$ and setting it to be the root of $T'.$ The Markov chain on $\mathcal{T}$ with transition probabilities $p_{\mathrm{SRW}}$ and initial distribution $\mathrm{MGW}_{\pi_Q}$ is stationary and reversible.
\end{theorem}

\begin{lemma}\label{claim:TypeRegenVertex} 
There exists a constant $c>0$ depending on $m,\Delta, \min_{i,j}n_i/n_j$ so that the following holds. Suppose the type of the root of $T$ is sampled according to~$\pi_Q$. Then for all types $i\in \{1,\ldots,m\}$ and all times $\ell \in \bN$ we have
 $$ \mathbb{E}[|\left\{t\le \ell: \Theta(X_t)=i \ \text{ and } \  X_t\neq X_s, \forall \ s<t\right\}|]\ge c\ell.$$  
\end{lemma}

{\proof Let $C$ be the expected number of visits to the root of a random walk starting from the root on the 3-regular infinite tree. Then, since every vertex of $T$ has degree at least $3$, it follows that the expected number of visits to any vertex $ v$ of $T$ by $X$ is at most $C.$  Writing $\tau_v=\inf\{t\geq 0:X_t=v\}$ and using that the expected number of visits to any vertex is at most~$C$, we have that 
 \begin{align*}
        \prcond{X_t=v}{T}{}\leq C\cdot \prcond{\tau_v=t}{T}{}.
 \end{align*}
 This now implies that 
 \begin{align*}
        &\econd{\left|\{t\le \ell: \Theta(X_t)=i, \forall s\le t, \,X_t\neq X_s\right\}|}{T} = \sum_{v\in T: \Theta(v)=i}\prcond{\tau_v\leq \ell}{T}{} \\&=
        \sum_{v\in T: \Theta(v)=i}\sum_{t=0}^{\ell}\prcond{\tau_v=t}{T}{}\geq \frac{1}{C}\cdot \sum_{v\in T: \Theta(v)=i} \sum_{t=0}^{\ell} \prcond{X_t=v}{T}{} \\&= \frac{1}{C}\cdot \econd{|\{t\le \ell: \Theta(X_t)=i\}|}{T},
 \end{align*}
 and hence averaging over the randomness of $T$ we get 
 \begin{align*}
        \E{\left|\{t\le \ell: \Theta(X_t)=i, \forall s\le t, \,X_t\neq X_s\right\}|} \geq \frac{1}{C}\E{|\{t\le \ell: \Theta(X_t)=i\}|}.
 \end{align*}
 Therefore, it suffices to prove that there exists a positive constant $c$ so that 
 \begin{align*}
        \E{|\{t\le \ell: \Theta(X_t)=i\}|}\geq c\ell.
 \end{align*}
From Theorem~\ref{claim:StationaryMGW} we have that $(\Theta(X_t))$ is a stationary process. Therefore $$\mathbb{E}[|\{t\le \ell: \Theta(X_t)=i\}|]=\ell\pi_{Q}(i).$$  
Since the number of communities $m$ is of order $1$, there exists a positive constant $c$ depending on $m, \Delta, \min_{i,j}n_i/n_j$ so that $\pi_Q(i)\geq  c.$ This now concludes the proof. \qed}

\begin{proof}[Proof of Lemma~\ref{lemma:PiOfW}]
 We first observe that $Q_2((i_1,j_1),(i_2,j_2))>0$ only if $j_1=i_2$. Let $p$ be the probability that a simple random walk on the binary tree never returns to the parent of the starting vertex. We now claim that
 $$W((i_1,j_1),(j_1,j_2))\ge \frac{p}{\Delta^5}Q(j_1,j_2)=\frac{p}{\Delta^5}Q_2((i_1,j_1)(j_1,j_2)).$$
 Indeed, a transition from $(i_1,j_1)$ to $(j_1,j_2)$ can happen for the chain $W$, if after a regeneration of the form $(i_1,j_1),$ the walk makes a step to a child $u$ of type $j_1$, then it backtracks to the parent of $u$, jumps immediately back to  $u$, then makes a jump to a child of type $j_2$ of $u$ and finally escapes to infinity without backtracking. We note that by the definition of the tree, the probability that a child of type~$j_2$ of $u$ is generated in the tree from an edge of type $(j_1,j_1)$ is lower bounded by $Q(j_1,j_2)/\Delta$. 

We now turn to prove the comparison of the invariant distributions. To do this, we first sample the root of $T$ according to $\pi_Q$.  Using Lemma~\ref{claim:TypeRegenVertex} and that the probability that after $X$ visits a vertex of type $i$ for the first time, then it jumps to a child of type $i,$ and then to a child of type $j$ and escapes is at least $pQ(i,j)/\Delta^3$, we obtain
\[
 \mathbb{E}\left[\sum_{t=1}^\ell\mathds{1}\left\{\exists \ k: \widetilde{\sigma}_k=t \text{ and } (\Theta(X_{t-1}), \Theta(X_t))=(i,j) \right\}\right]\geq  \frac{pQ(i,j)}{\Delta^3}\cdot c\ell.
 \]
By the ergodic theorem we have that almost surely
\[
\frac{\sum_{s=1}^L\mathds{1}\left\{\Theta(X_{\widetilde{\sigma}_{s}})=(i,j)\right\}}{L}\to \pi_{W}(i,j)\,\,\text{ as } L\to \infty.
\] 
Setting $N_\ell=\max\{k: \widetilde{\sigma}_k\leq \ell\}$, we have that 
\begin{align*}
        \E{\frac{\sum_{t=1}^\ell\mathds{1}\left\{\exists \ k: \widetilde{\sigma}_k=t \text{ and } (\Theta(X_{t-1}), \Theta(X_t))=(i,j) \right\}}{N_\ell}} \geq \frac{pQ(i,j)}{\Delta^3}\cdot c.\end{align*}
Since $N_\ell\to \infty$ as $\ell\to\infty$, we therefore obtain that 
\[
\pi_W(i,j)\geq \frac{pQ(i,j)}{\Delta^3}\cdot c.
\]
Using that $\pi_Q(i)\geq c_1$ for a constant $c_1$ depending on $m,\Delta,\min_{i,j}n_i/n_j$, we get that 
\[
\pi_W(i,j)\geq c_2 \pi_{Q_2}(i,j)
\]
for a positive constant $c_2$. 

To get the upper bound it is enough to show that $W((i_1,j_1),(i_2,j_2))\lesssim Q(i_2,j_2),$ for any~$i_1,i_2,j_1,j_2.$  
This clearly holds as if we condition on the first vertex of a regeneration edge being of type $i_2$, then
  the probability that the second one is of type $j_2$ is bounded by $\lesssim Q(i_2,j_2).$ 
\end{proof}

{\proof[Proof of Lemma \ref{lemma:PoincareWQ2}] 
Using Lemma~\ref{lemma:PiOfW} immediately gives that for all functions $f:\{1,\ldots,m\}^2\to \bR$
\[
\mathcal{E}_{Q_2}(f)\lesssim \mathcal{E}_{W}(f) \quad \text{ and } \quad \text{Var}_{\pi_W}(f)\asymp \text{Var}_{\pi_{Q_2}}(f).
\]
The statement of the lemma then follows from the definition of the \Poincare constant.  \qed}

Our next goal is to prove that $\Phi_*^{Q}\lesssim \gamma_{Q_2}.$ For the proof we will use the following theorem~\cite{MixHitPaper}. 

\begin{theorem}\label{thrm:TmixHit} For a reversible transition matrix $P$, let $t_{\mathrm{mix}}^{\mathrm{TV}, P_L}$ be a mixing time of the lazy version of $P$, i.e. of the Markov chain with transitions $\frac{P+I}{2}$ and $T_A$ be the first time this Markov chain hits a set~$A.$ It holds that $$t_{\mathrm{mix}}^{\mathrm{TV}, P_L}\asymp \max_{x,\pi_P(A)\ge \frac{1}{4}}\mathbb{E}_{x}[T_A].$$   Moreover, if the chain $P$ is such that $P(x,x)\ge \delta$ for some $\delta\in (0,1)$ then for the mixing time of the chain with transitions according to $P$ we have $$t_{\mathrm{mix}}^{\mathrm{TV}, P}\lesssim\max\left\{\frac{1}{\delta(1-\delta)},t_{\mathrm{mix}}^{\mathrm{TV}, P_L} \right\}.$$
\end{theorem}
 We will show that we can upper bound the hitting times of large sets of the chain $Q$ by the inverse of its Cheeger constant. This in turn will yield an upper bound on the mixing time. Since the mixing times of the chains $Q_2$ and $Q$ are essentially the same, and the total variation mixing time is always larger than the inverse of the Cheeger constant, this will give the comparison of $\Phi_*^Q$ and $\Phi^{Q_2}_*.$ Furthermore, using similar method as for~$Q$, we will bound the hitting times of large sets by the additive symmetrisation of $Q_2$ by $\frac{1}{\Phi^{Q_2}_*}$ which will again by the above theorem give the bound on the mixing time (of the lazy version). As the mixing time of a reversible Markov chain is always larger than the inverse of the \Poincare constant, and given that the \Poincare constant is the same for any chain and its additive symmetrisation, this will complete the proof.

If $P$ is a transition matrix on the state space $S$ with invariant distribution~$\pi$, we denote by $P^*$ the time-reversal of $P$ defined as
\[
P^*(x,y) = \frac{\pi(y)P(y,x)}{\pi(x)}.
\]
Also, for the subset $A\subseteq S$ we let 
\[
\Phi_P(A)=\frac{\sum_{x\in A, y\in A^c}\pi(x)P(x,y)}{\pi(A)}.
\]

\begin{lemma}\label{lemma:TmixPHit} There exists a positive constant $C$ depending on $m,\Delta,\min_{i,j}n_i/n_j$ so that if $T_A$ stands for the hitting time of the set $A$ by the chain with transition matrix $P\in \left\{Q,\frac{Q_2+Q_2^*}{2}\right\}$, then we have that for any $x$ in the state space of $P$ $$\max_{A:\pi_P(A)\ge \frac{1}{4}}\mathbb{E}_{x}[T_A]\leq \frac{C}{\Phi_*^P}. $$
\end{lemma}

{\proof Fix an element of the state space $x_0$ and a set $A$ with $\pi_P(A)\ge \frac{1}{4}$ and let $A_0=\{x_0\}.$ Let $x_1$ be the maximiser of $P(x_0,z)$ over all $z\ne x_0$ and $A_{1}=\{x_0,x_1\}$. Define the sets $A_\ell$ recursively by letting $A_{\ell+1}=A_\ell\cup \{x_{\ell+1}\},$ where $x_{\ell+1}$ is not in $A_\ell$ and is the maximiser of $\max_{x\in A_\ell}\pi_P(x)P(x,z)$ over all~$z\notin A_\ell.$ Let $k$ be minimal such that~$x_k\in A.$ Let $x_{\ell_0}=x_0, x_{\ell_1},\ldots x_{\ell_r}=x_k$ be a path from $x_0$ to $x_k$ consisting of states in $A_k$, such that $\ell_t$ is increasing and satisfies that $$\pi_P(x_{\ell_t})P(x_{\ell_t},x_{\ell_{t+1}})=\max_{\substack{x\in A_{\ell_{t+1}-1},\\  y\in A_k\setminus A_{\ell_{t+1}-1}}} \pi_P(x)P(x,y).$$
We can construct such a sequence recursively backwards as follows: 
let $f:\{1,\ldots, k\}\to\{0,\ldots, k\}$ defined by 
\[
f(s)=\inf\{u<s:\pi_P(x_u)P(x_u,x_s)=\max_{z\in A_{s-1}}\pi_P(z)P(z,x_s)\},
\]
i.e.\ $f(s)$ is the first index of the element in $A_{s-1}$ with maximal transition to $x_s.$ We apply~$f$ repeatedly to $k$ until reaching $0,$ which will happen at some finite time $r$ as $f(s)<s.$ The sequence $x_0=x_{f^r(k)}, x_{f^{r-1}(k)},\ldots, x_{f^1(k)}, x_k$ then has the required  property. This is because $x_s$ was chosen to maximise the transition $\max_{z\in A_{s-1}}\pi_P(z)P(z,x_s)$ and therefore 
\[
\pi_P(x_{f(s)})P(x_{f(s)},x_s)=\max_{x\in A_{s-1},y\in A_k\setminus A_{s-1}}\pi_P(x)P(x,y).  
\]
Using that $\estart{T_y}{x}\leq 1/(\pi_P(x) P(x,y))$ for $x$ and $y$ neighbours and as $r$ is bounded by $m^2$ which  is of order $1$, we get that 
\begin{align}\label{eq:upperboundonexpectation}
\mathbb{E}_{x_0}[T_A]\le \sum_{s=1}^r\mathbb{E}_{x_{\ell_{s-1}}}[T_{x_{\ell_s}}] \asymp \max_{1\le s\le r}\mathbb{E}_{x_{\ell_{s-1}}}[T_{x_{\ell_s}}]\leq \max_{1\le s\le r}\frac{1}{\pi_P(x_{\ell_{s-1}})P(x_{\ell_{s-1}},x_{\ell_s})}. 
\end{align} 
As the size of the state space of $P$ is at most $m^2$ which is of order $1$, we have that for any non-empty set $B$
$$\Phi_P(B)\asymp \max_{x\in B}\frac{\pi_P(x)P(x,B^c)}{\pi_P(B)}\asymp \max_{x\in B,y\in B^c}\frac{\pi_P(x)P(x,y)}{\pi_P(B)}.$$
By the definition of the sequence $(\ell_s)$ and the sets $(A_{\ell_s})$ we have  
\[
\max_{x\in A_{\ell_s-1}}\max_{z\in A_{\ell_s-1}^c}\pi_P(x)P(x,z) = \pi_P(x_{\ell_{s-1}})P(x_{\ell_{s-1}},x_{\ell_s}),
\]
and hence
\[
\frac{1}{\pi_P(x_{\ell_{s-1}})P(x_{\ell_{s-1}},x_{\ell_s})}\asymp \frac{1}{\pi_P(A_{\ell_s-1})\Phi_P(A_{\ell_s-1})}.
\]
Plugging this into~\eqref{eq:upperboundonexpectation} yields
\begin{align*}
        \mathbb{E}_{x_0}[T_A]\lesssim \max_{1\leq s\leq r} \frac{1}{\pi_P(A_{\ell_s-1})\Phi_P(A_{\ell_s-1})}.
\end{align*}
As $A\cap A_{\ell_s-1}=\emptyset$ and $\pi_P(A)\ge \frac{1}{4}$, then $\pi_P(A_{\ell_s-1})\le \frac{3}{4}$. Using this together with the fact that  for any set $B$ we have $\Phi_P(B)=\frac{\pi_P(B^c)}{\pi_P(B)}\Phi_P(B^c)$, we get 
 $$\max_{1\le s\le r}\frac{1}{\Phi_P(A_{\ell_s-1})}\lesssim \frac{1}{\Phi_*^P}.$$ 
In the case of $P=Q$, the proof of the lemma follows  easily using that $\pi_Q(i)\asymp 1$ for any $i\in \{1,\ldots, m\}$ and  that  $\pi_Q(A_{\ell_s-1})\asymp 1$, since $A_{\ell_s-1}\neq \emptyset$ for any $s>0$. 
To conclude the proof when $P=\frac{Q_2+Q_2^*}{2}$, notice that if $x_0=(i,i)$ for some $i\in\{1,\ldots, m\}$, then $\pi_P(A_{\ell_s-1})\asymp 1$ and therefore $\mathbb{E}_{(i,i)}[T_A]\lesssim \frac{1}{\Phi_*^P}$. Further, if we let $\tau$ be the hitting time of the set $\{(x,x):x\in\{1,\ldots m\}\}$ by the Markov chain with transition matrix $P$ then for any $i,j\in\{1,\ldots m\}$ 
\[\mathbb{E}_{(i,j)}[T_A]\le \mathbb{E}_{(i,j)}[\tau]+ \sum_{x\in \{1,\ldots, m\l\}}\mathbb{E}_{(x,x)}[T_{A}]\lesssim  \mathbb{E}_{(i,j)}[\tau]+\frac{1}{\Phi_*^P}.\] As the probability of jumping along an edge of the form $(x,x)$ for $x\in \{1,\ldots, m\}$ is at least~$\frac{3}{2\Delta}$, we get that $\mathbb{E}[\tau]\le \frac{2\Delta}{3}$, which concludes the proof. 
\qed}

 \begin{corollary}\label{cor:cheegerQpoincareQ} There exists a constant $c$ depending on $m,\Delta, \min_{i,j}n_i/n_j$ such that 
 $$c\Phi_*^Q\le {\gamma_Q}. $$
 \end{corollary}
 {\proof Lemma~\ref{lemma:TmixPHit} and Theorem~\ref{thrm:TmixHit} give that $t_{\text{mix}}^{\text{TV}, Q}\lesssim \frac{1}{\Phi_*^Q},$ as the chain $Q$ is reversible and also has diagonal entries bounded away from $0$. As $Q$ is reversible we also know that $t_{\text{mix}}^{\text{TV}, Q}\gtrsim \frac{1}{\gamma_Q},$ which completes the proof.\qed}

\begin{lemma}\label{lemma:CheegerQCheegerQ2} There exists a positive constant $C$ depending on $m, \Delta, \min_{i,j}n_i/n_j$ so that
 $$\frac{1}{\Phi_*^{Q_2}}\leq\frac{C}{\Phi_*^Q}. $$
\end{lemma}
{\proof Lemma~\ref{lemma:TmixPHit} and Theorem~\ref{thrm:TmixHit} give that $t_{\text{mix}}^{\text{TV}, Q}\lesssim \frac{1}{\Phi_*^Q},$ as the chain $Q$ is reversible and also has diagonal entries bounded away from $0$. In general, for any Markov chain it holds that the inverse of the Cheeger constant is bounded by the mixing time (see for instance~\cite[Theorem~7.4]{MixingBook}) and therefore $\frac{1}{\Phi_*^{Q_2}}\lesssim t_{\text{mix}}^{\text{TV},Q_2}.$
We will complete the proof by showing that $t_{\text{mix}}^{\text{TV},Q}=t_{\text{mix}}^{\text{TV},Q_2}-1.$ 
It is elementary to check that for all $x=(x_1,x_2)$ 
$$\|Q_2^t(x,\cdot)-\pi_{Q_2}(\cdot)\|_{\text{TV}}= \|Q^{t-1}(x_2,\cdot)-\pi_{Q}(\cdot)\|_{\text{TV}}.$$
Maximising over all $x$ completes the proof. \qed}

We will now use the following general fact about the \Poincare constant in the reversible case. 
\begin{lemma}[{\cite[Theorem 12.5]{MixingBook}}]\label{lemma:PoincareMix} There exists a universal constant $c$ so that if $P$ is an aperiodic reversible Markov chain with \Poincare constant $\gamma_P$, then
 $$ \frac{c}{\gamma_{P}}\leq t_{\mathrm{mix}}^{\mathrm{TV},P}.$$
\end{lemma}
We now bound the \Poincare constant of $Q$ by its Cheeger constant from below. 
 
\begin{lemma}\label{lemma:PoincareQ2CheegerQ} There exists a positive constant $c_1$ depending on $m,\Delta,\min_{i,j}n_i/n_j$ so that
$$\Phi_*^Q\leq c_1 \gamma_{Q_2}.$$
\end{lemma}
{\proof Using Theorem~\ref{thrm:TmixHit}, Lemmas~\ref{lemma:TmixPHit} and~\ref{lemma:CheegerQCheegerQ2} we have for the mixing time of the lazy version~$P_L$ of $P=\frac{Q_2+Q_2^*}{2}$ that 
$$t_{\text{mix}}^{\text{TV},P_L}\lesssim \frac{1}{\Phi_*^Q}.$$ 
Further, since $P_L$ is a reversible chain, we have  from Lemma \ref{lemma:PoincareMix} that 
$$ \frac{1}{\gamma_{P_L}}\lesssim t_{\text{mix}}^{\text{TV},P_L}.$$ 
It is easy to check that $\mathcal{E}_{Q_2}(f)=\mathcal{E}_{P}(f)=2\mathcal{E}_{P_L}(f)$ and, since $P$, $Q_2$ and $P_L$ have the same invariant distributions, it follows that $\gamma_{Q_2}=\gamma_{P}=2\gamma_{P_L}$, which completes the proof. \qed}

We are now ready to bound the mixing time of~$\Sigma$ by the inverse of the Cheeger constant of~$Q$.

{\proof[Proof of Proposition~\ref{thrm:MixingOfW}] Combining the results from Lemmas \ref{lemma:PoincareQ2CheegerQ} and \ref{lemma:PoincareWQ2} and Corollary  \ref{cor:PoincareWSigma} gives that $$\Phi_*^Q\lesssim \gamma_{Q_2}\lesssim\gamma_W\leq\gamma_\Sigma.$$ 

As for all types $i\in \{1,\ldots,m\}$, we have that $\pi_{Q_2}((i,i))\asymp 1$, it follows from Lemma~\ref{lemma:PiOfW} that 
$\pi_W((i,i))\asymp 1$, and therefore also~$\pi_\Sigma((i,i))\asymp 1$, as $\Sigma$ is the induced chain. This now implies that for any~$\theta$ in the state space of $\Sigma$ $$\|\Sigma(\theta,\cdot)-\pi_{\Sigma}(\cdot)\|_{2}\asymp 1.$$ 
Since for all $i,$ we have $$1\asymp Q((i,i),(i,i))\lesssim W((i,i),(i,i))\leq \Sigma((i,i),(i,i)),$$ it follows that $\Sigma((i,i),(i,i))\ge c$ for all $i$, for some positive constant $c$. Using~\cite[eq.~(2.11) and Remark~2.16]{MontenegroTetali}, we get 
 $$ \|\Sigma^t(\theta,\cdot)-\pi_{\Sigma}(\cdot)\|_{\text{TV}}\leq \|\Sigma(\theta,\cdot)-\pi_{\Sigma}(\cdot)\|_{2}\cdot e^{-2c\cdot t\cdot \gamma_{\Sigma}}\lesssim e^{-c't\Phi_*^Q}, $$ 
 where $c'$ is another positive constant and this now completes the proof. \qed}

\subsection{Speed and entropy of (lazy) simple random walk on multi-type \random tree}\label{secEntropy}
From now on we will let $\alpha=\Phi_*^Q$ in the case of a multi-type \random tree from Definition \ref{def:MTGWTree}, while in the case of the two-type tree, $\alpha$ is as defined in Definition \ref{def:2TGWTree}. In both cases we have that  $t_{\text{mix}}^{\text{TV},\Sigma}\lesssim\frac{1}{\alpha}.$

\begin{definition}
\label{def:RegTimes} 
Define, as in Lemma \ref{lemma:RegIndep}, for a tree  $T$ with root $\rho$ as in Definition \ref{def:MTGWTree} or Definition \ref{def:2TGWTree} and for  $K\ge 0$, $\sigma_0$ to be the first time that a simple random walk $X$ on $T$ started from $\rho$, reaches $\partial \mathcal{B}_K\left(\rho\right)$ and $\left(\sigma_i\right)_{i=1}^\infty$ to be the almost surely infinite sequence of regeneration times occurring after~$\sigma_0$, where $\sigma_i$ is the $i$-th regeneration time such that ${\phi_i=d\left(\rho,X_{\sigma_i}\right)>K}$.
\end{definition}
\begin{definition}
\label{def:LoopErasedRW}
Let $T$ be as in Definition \ref{def:MTGWTree} or \ref{def:2TGWTree}. Let $\xi$ be a loop erased random walk on $T$ which is obtained by erasing loops from a simple random walk started from $\rho$ in the chronological order in which they were created. (This procedure is possible, since $T$ is uniformly transient by Lemma~\ref{lemma:ReturnProb}.) We call $\xi_i$ the $i$-th edge crossed by $\xi$. Unless otherwise specified, the loop erased random walk always starts from $\rho$. 
\end{definition}
In this section the goal is to prove results on the speed and entropy of the lazy simple random walk on $T$. We start by stating and proving a decorrelation result for a general Markov chain.

\begin{lemma}\label{lemma:Decorrelation} Let $\Sigma$ be a Markov chain on the state space $S$  with invariant distribution $\pi$ and let~$t_\mathrm{mix}$ be its total variation mixing time. Let $f, g:S\to \R$ be bounded functions satisfying $\estart{g}{\pi}=0$. Then there exists a positive constant $C$ (depending on the bound on $f$ and $g$) so that for all $i\le j$ and~$\theta\in S$ 
\[\estart{f(\Sigma_i)g(\Sigma_j)}{\theta}\leq C\cdot 2^{-\frac{j-i}{2t_{\mathrm{mix}}}}.
\]
\end{lemma}
\begin{proof}
Let $P$ be the transition matrix of $\Sigma$. Then we have 
\[
\estart{f(\Sigma_i)g(\Sigma_j)}{\theta} =\sum_{\eta\in S} P^i(\theta,\eta) f(\eta) \estart{g(\Sigma_{j-i})}{\eta}.
\]
Let $(Z_\eta,\Sigma_{j-i})$ be the optimal coupling of  ${\pi}$ and $P^{j-i}(\eta,\cdot)$ and define $A$ to be the event that the coupling fails, i.e.\ $A=\left \{Z_\eta\neq\Sigma_{j-i}\right\}.$ Given that by assumption $\mathbb{E}[g(Z_\eta)]=0$, we get 
\begin{align*}
&\sum_{\eta\in S} P^i(\theta,\eta) f(\eta) \estart{g(\Sigma_{j-i})}{\eta}=\sum_{\eta\in S}P^{i}(\theta,\eta)f(\eta)\estart{\left(g\left(\Sigma_{j-i}\right)-g\left(Z_\eta\right)\right)}{\eta}\\
&=\sum_{\eta\in S}P^{i}(\theta,\eta)f(\eta)\estart{\left(g\left(\Sigma_{j-i}\right)-g\left(Z_\eta\right)\right)\mathds{1}(A)}{\eta}
\\ &\leq  \sum_{\eta\in S}P^{i}(\theta,\eta)|f(\eta)|\sqrt{\estart{\left(g\left(\Sigma_{j-i}\right)-g\left(Z_\eta\right)\right)^2}{\eta}\prstart{A}{\eta}} \lesssim \sum_{\eta\in S}P^{i}(\theta,\eta) \sqrt{\prstart{A}{\eta}},
\end{align*}
where the inequality follows by Cauchy-Schwartz and the implied constants come from the bounds on~$f$ and $g$. Using~\cite[eq.~(4.34)]{MixingBook} we immediately deduce
\[
\prstart{A}{\eta}=\|\pi(\cdot)-P^{j-i}(\eta,\cdot)\|_{\text{TV}}
\leq 2^{-\frac{j-i}{t_{\text{mix}}}}
\]
and this concludes the proof.\end{proof}

Before giving the next result we define $\prstart{\cdot}{T_0}=\prcond{\cdot}{\mathcal{B}_{K}\left(\rho\right)=T_0}{}$ to be the probability measure conditional on the first $K$ levels of $T$ being equal to $T_0$ and $\mathbb{E}_{T_0}$ and $\mathrm{Var}_{T_0}$ be the corresponding expectation and variance. 

We next give a concentration result on the number of regeneration times which occur before some fixed level. This will allow us to show the convergence of the speed of a simple random walk on $T$. 

\begin{lemma}\label{lemma:ConcentrationPhi}
Let $T$ be a tree with root $\rho$ as in Definitions~\ref{def:MTGWTree} or~\ref{def:2TGWTree} and let $\sigma_i$ and $\phi_i$ be as in Definition \ref{def:RegTimes}. Fix $K>0$ and let $T_0$ be realisation of the first $K$ levels of $T$. For each $\ell\in \mathbb{N}$ let
$$N_\ell =\max\left\{i\ge 0: \phi_i\le \ell+K\right\} $$ be the number of regeneration times occurring before level $\ell+K.$
Then for all $\varepsilon>0$ there exists a sufficiently large constant $C$ such that for all $k\ge K^2$ $$\prcond{\left|N_k-\frac{k}{\estart{\phi_2-\phi_1}{\pi_\Sigma}}\right|> C\left(\sqrt{\frac{k}{\alpha}}+\frac{1}{\alpha}\right)}{\mathcal{B}_K\left(\rho\right)=T_0}{}\le \varepsilon,$$ where ${\pi_{\Sigma}}$ is the stationary distribution from Definition~\ref{def:MCRegTypes} and where $\estart{\cdot}{\pi_\Sigma}$ is the expectation when~$X_{\sigma_1}\sim\pi_\Sigma$.
\end{lemma}

{\proof Set $\zeta_i=\phi_i-\phi_{i-1}$ for $i\ge 2$ and $\zeta_1=\phi_1$. 
We will first show that
\[\estart{\left(\sum_{i=1}^k\l\zeta_i-k\estart{\zeta_2}{\pi_{\Sigma}}\right)^2}{T_0}\lesssim \frac{k}{\alpha}. 
\] 
In fact the left hand side above is equal to 
\[\sum_{i\le k}\estart{\left(\zeta_i-\estart{\zeta_2}{\pi_{\Sigma}}\right)^2}{T_0}+2\sum_{i<j\le k}\estart{\left(\zeta_j-\estart{\zeta_2}{\pi_{\Sigma}}\right)\left(\zeta_i-\estart{\zeta_2}{\pi_{\Sigma}}\right)}{T_0}, 
\]
and as the $\zeta_i$'s have exponential tails by Lemma~\ref{lemma:RegIndep}, we get that the first term above is $\lesssim k.$ To bound the second term notice that for $i<j$ 
\begin{align*}
&\estart{\left(\zeta_j-\estart{\zeta_2}{\pi_{\Sigma}}\right)\left(\zeta_i-\estart{\zeta_2}{\pi_{\Sigma}}\right)}{T_0}=\estart{\escond{\left(\zeta_j-\estart{\zeta_2}{\pi_{\Sigma}}\right)\left(\zeta_i-\estart{\zeta_2}{\pi_{\Sigma}}\right)}{\Sigma_i,\Sigma_{j-1}}{T_0}}{T_0}\\
&=\estart{\escond{\left(\zeta_j-\estart{\zeta_2}{\pi_{\Sigma}}\right)}{\Sigma_{j-1}}{T_0}\escond{\left(\zeta_i-\estart{\zeta_2}{\pi_{\Sigma}}\right)}{\Sigma_i}{T_0}}{T_0}, 
\end{align*}
 where the last equality holds by Lemma~\ref{lemma:RegIndep}.
We define the functions \[f(\theta)=\escond{\left(\zeta_i-\estart{\zeta_2}{\pi_{\Sigma}}\right)}{\Sigma_i=\theta}{T_0} \ \text{ and } \ 
  g(\theta)=\escond{\left(\zeta_j-\estart{\zeta_2}{\pi_{\Sigma}}\right)}{ \Sigma_{j-1}=\theta}{T_0}.\] Notice that $f$ and $g$ are bounded, since $\zeta_i$ have exponential tails, and hence we can now apply Lemma~\ref{lemma:Decorrelation} to deduce for a positive constant $c$
  \begin{align*}
        \estart{\escond{\left(\zeta_j-\estart{\zeta_2}{\pi_{\Sigma}}\right)}{\Sigma_{j-1}}{T_0}\escond{\left(\zeta_i-\estart{\zeta_2}{\pi_{\Sigma}}\right)}{\Sigma_i}{T_0}}{T_0} = \estart{f(\Sigma_i)g(\Sigma_{j-1})}{T_0}\lesssim 2^{-c\alpha(j-i)}. 
                \end{align*}
        Therefore,  
        \begin{align}\label{eq:boundformarkov}
        \estart{\left(\sum_{i=1}^k\l\zeta_i-k\estart{\zeta_2}{\pi_{\Sigma}}\right)^2}{T_0}\lesssim k+\sum_{i<j\le k}2^{-c\alpha(j-i)}\lesssim  k+k\sum_{s\le k} 2^{-c\alpha s}\lesssim \frac{k}{\alpha}.
        \end{align}
        We now let 
        \[\ell=  \left\lfloor\frac{k}{\estart{\phi_2-\phi_1}{\pi_{\Sigma}}} +C\left(\sqrt{\frac{k}{\alpha}}+\frac{1}{\alpha}\right)\right\rfloor.
        \]
By definition of $N_k$ and $\ell$, taking $C$ large enough and using the assumption on $K$ we obtain
\begin{align*}
\prstart{N_k>\ell}{T_0}&=\prstart{\sum_{i=1}^\ell  \zeta_i<k+K}{T_0}\\&\le \prstart{\sum_{i=1}^\ell  \zeta_i-\ell\estart{\zeta_2}{\pi_{\Sigma}}< -C\estart{\zeta_2}{\pi_{\Sigma}}\left(\sqrt{\frac{k}{\alpha}}+\frac{1}{\alpha}\right)+K}{T_0}\\
&\leq \prstart{\left|\sum_{i=1}^\ell  \zeta_i-\ell \estart{\zeta_2}{\pi_{\Sigma}}\right|\ge \frac{C\estart{\zeta_2}{\pi_\Sigma} }{2}\left(\sqrt{\frac{k}{\alpha}}+\frac{1}{\alpha}\right)}{T_0} \leq \varepsilon,
\end{align*}
where for the last inequality we used Markov's inequality together with~\eqref{eq:boundformarkov}.
 Similarly it follows that
 \[
 \prstart{N_k<\left\lfloor\frac{k}{\estart{\phi_2-\phi_1}{\pi_{\Sigma}}} -C\left(\sqrt{\frac{k}{\alpha}}+\frac{1}{\alpha}\right)\right\rfloor}{T_0}<\varepsilon
 \]
 and this concludes the proof.
  \qed

We prove the following lemma in Appendix~\ref{appendixA}
\begin{lemma}\label{lemma:Speed}
Let $X$ be a simple random walk on $T$. Then for $\nu=\frac{\estart{\phi_2-\phi_1}{{\pi_{\Sigma}}}}{\estart{\sigma_2-\sigma_1}{{\pi_{\Sigma}}}}$ almost surely 
\[\frac{d\left(\rho,X_t\right)}{t}\to \nu \,\,\,\,\mathrm{as}\,\, t\to \infty. 
\]
Moreover, for all $\varepsilon>0$ there is a positive constant $C$ so that for all $t$ sufficiently large
\begin{align*}&\pr{|d\left(\rho,X_t\right)-\nu t|\ge C\left(\sqrt{\frac{t}{\alpha}}+\frac{1}{\alpha}\right)}\leq \varepsilon\quad \text{and}\\& \pr{\sup_{s:s\le t}d\left(\rho,X_s\right)>\nu t+2C\left(\sqrt{\frac{t}{\alpha}}+\frac{1}{\alpha}\right)}\le \varepsilon. 
\end{align*}
\end{lemma}

 We now state a lemma which provides a crucial bound for studying the fluctuations of the entropy of loop erased random walk. The proof of the bound on the moments of $Y_k$ follows in the same way as in~\cite[Lemma~3.14]{PerlasPaper}, but the proof of the variance bound is substantially different, as in our case we need to condition on the types of the regeneration vertices. To do this, we use the results of Section~\ref{secMixingW}. We give the proof in Appendix~\ref{appendixA}. 

\begin{lemma} \label{lemma:EandVar}
There exist positive constants $\left(C_\ell \right)_{\ell  \ge 1}$ and $C'$ so that the following hold: let T be a multi-type \random tree with root $\rho$ as in Definition \ref{def:MTGWTree} or \ref{def:2TGWTree}. Fix $K\ge 0$ and let $T_0$ be a realisation of the first $K$ levels of $T$. Let $X$ be a simple random walk on $T$ started from $\rho$ and let $\widetilde{\xi}$ be an independent loop erased random walk on $T$. For $k\ge 1$ define 
\[Y_k=-\log \prcond{\left(X_{\sigma_k-1},X_{\sigma_k}\right)\in \widetilde{\xi}}{X,T}{}+\log\prcond{\left(X_{\sigma_{k-1}-1},X_{\sigma_{k-1}}\right)\in \widetilde{\xi}}{X,T}{}. \]
Then, for $\theta \in \left\{1,\ldots, m\right\}$, the distribution of $Y_i$ conditioned on $\Theta(X_{\sigma_{i-1}})=\theta$ is same as the distribution of $Y_j$ conditioned on $\Theta(X_{\sigma_{j-1}})=\theta$. The sequence $\left(Y_k\right)_{k\geq 2}$ is also independent of $\mathcal{B}_K\left(\rho\right)$ conditional on the type of the vertex $X_{\sigma_1}$. Moreover, for all $l\ge 1$ 
\begin{equation}\label{eq:MomentSigma0}
\econd{\left(-\log\prcond{\left(X_{\sigma_0-1},X_{\sigma_0}\right)\in \widetilde{\xi}}{ X,T \right)}{}^\ell }{ \mathcal{B}_K\left(\rho\right)=T_0}{}\leq C_\ell K^\ell   \end{equation} and for all $k\ge 1$ 
\begin{equation}\label{eq:MomentY}
\econd{\left(Y_k\right)^\ell}{\mathcal{B}_K\left(\rho\right)=T_0}{}\leq C_\ell .\end{equation} In addition, there exists a positive constant $C'$ so that for all $k\geq 1$ we have 
\begin{equation}\label{eq:VarY}
\econd{\left(\sum_{i=1}^k(Y_i-\estart{Y_2}{\pi_{\Sigma}})\right)^2}{ \mathcal{B}_K\left(\rho\right)=T_0}{}\leq \frac{C'k}{\alpha}, \end{equation} 
where $\estart{\cdot}{\pi_\Sigma}$ is the expectation when $X_{\sigma_1}\sim \pi_\Sigma$.\end{lemma}

We conclude this section by obtaining the entropy result for the loop erased random walk. The proof follows similarly as in~\cite[Lemma~3.15]{PerlasPaper} and it is given in Appendix~\ref{appendixA} 

\begin{prop} 
\label{prop:LERWLimit}
Let T be a multi-type \random tree and let $\xi$ and $\widetilde{\xi}$ be two independent loop erased random walks on $T$ both started from the root. Then there exists a positive constant $\mathfrak{h}$ so that almost surely $$\frac{-\log\prcond{\xi_k\in \widetilde{\xi}}{ T,\xi}{}}{k}\to \mathfrak{h} \quad \text{as}\quad k\to \infty.$$ Fix $K>0$ and let $T_0$ be a realisation of the first $K$ levels of $T$. For all $\varepsilon>0$, there exists a positive constant $C$ so that for all $k\ge K^2$ 
\[\prcond{\left|-\log\prcond{\xi_k\in\widetilde{\xi}}{T,\xi}{}-\mathfrak{h}k\right| > C\left(\sqrt{\frac{k}{\alpha}}+\frac{1}{\alpha}\right)}{\mathcal{B}_K\left(\rho\right)=T_0}{} \leq \varepsilon.
\]
\end{prop}

\subsection{Truncation}\label{sec3}
Following~\cite{PerlasPaper} we define in this section the truncation event which will be used to define the coupling needed for the proof of the main result in case $\alpha \gg \frac{1}{\log n}$.
\begin{definition} \label{def:W} Let $e$ be an edge of $T$ and $\xi$ a loop erased random walk started from the root of $T$ as in Definition \ref{def:LoopErasedRW}. Define
\[ W_T\left(e\right)=-\log\prcond{e\in \xi}{T}{}.
\]
 For an edge $e=\left(x,y\right)$ with $d\left(\rho,x\right)\le d\left(\rho, y\right)$ write $\ell\left(e\right)=d\left(\rho,y\right).$ Define 
 \[\widetilde{W}_T\left(e\right)=-\log\prcond{\left(X_{\tau^{\ell\left(e\right)}-1}X_{\tau^{\ell\left(e\right)}}\right)=e}{T}{}, \] 
 where $X$ is a simple random walk on $T$ started from the root and \[\tau^{\ell\left(e\right)}=\inf\left\{t\ge 0: d\left(\rho, X_t\right)=\ell\left(e\right)\right\}.\]
\end{definition}
The proof of the following lemma follows in exactly the same way as~\cite[Lemma~4.3]{PerlasPaper} and therefore we omit it. 
\begin{lemma}\label{lemma:IneqW} There exists a positive constant $c$, such that for all realisations of $T$ and all edges $e$ of~$T$ we have $$W_T\left(e\right)\ge \widetilde{W}_T\left(e\right)-c. $$
\end{lemma}

\begin{definition}\label{def:Truncation} Let $A>0$ and $K=\lceil{C_2\log\log n}\rceil$ for a constant $C_2$ to be determined later. For~$e\in T$ we define the {\it truncation event} $\mathrm{Tr}\left(e,A\right)$ to be 
\[\mathrm{Tr}\left(e,A\right)=\left\{\widetilde{W}_T\left(e\right)>\log n-A \sqrt{\frac{\log n}{\alpha}}\right\}\cap \left\{\ell\left(e\right)\ge K\right\}, 
\] 
where $\ell\left(e\right)$ again stands for the level of $e$. 
\end{definition}

The following lemma will be useful in Section~\ref{sec5} as it shows that with high probability the walk on the tree will not visit truncated edges by the mixing time. 

\begin{lemma} \label{lemma:TruncatedEdgeNotVisited}
Let $K\le \left(\log\log n\right)^2$ and let $T_0$ be a realisation of the first $K$ levels of $T$. Let $X$ be a simple random walk on $T$ started from its root and set $t=\frac{\log n}{\nu \mathfrak{h}}-B\sqrt{\frac{\log n}{\alpha}},$ where $\nu$ and $\mathfrak{h}$ are as in Lemma~\ref{lemma:Speed} and Proposition~\ref{prop:LERWLimit}. If $\alpha \gg \frac{1}{\log n}$, for all $\varepsilon\in\left(0,1\right)$ there exist sufficiently large constants $B$ and $A$ (depending on both $\varepsilon$ and $B$) such that 
\[\prcond{\bigcup_{k\le t}\mathrm{Tr}\left(\left(X_{k-1},X_k\right),A\right)}{\mathcal{B}_{K}\left(\rho\right)=T_0}{}\leq \varepsilon. 
\]
\end{lemma}
\begin{proof} The proof of this lemma follows in exactly the same way as the proof of~\cite[Lemma~4.5]{PerlasPaper}. Note that when using Lemma~\ref{lemma:Speed} and Proposition~\ref{prop:LERWLimit} we only keep the term $\sqrt{t/\alpha}$ by our assumption on~$t$ and $\alpha$. 
 \end{proof}
 
 \section{Controlling the spectral profile}\label{sec4}

In this section we prove Proposition~\ref{prop:TRel} for $t_{\mathrm{rel}}$ for the $m$-communities model and for the 2-communities one in the case when $\alpha\gtrsim 1/\log n$. The proof for $\alpha\ll 1/\log n$ is given at the end of Section~\ref{sec5}.
For this section we take $X$ to be a lazy simple random walk. We also obtain bounds on the spectral profile of $G_n$ that will be used in Section~\ref{sec5} in the proof of Theorems~\ref{thrm:mCutoff} and~\ref{thrm:Cutoff}. 

We start by stating a result from~\cite{madrasrandall}, which immediately implies the upper bound on the relaxation time for the $m$-communities model. 

\begin{theorem}[{\cite[Theorem~1.1]{madrasrandall}}]\label{thm:madrasrandall}
Let $X$ be a reversible Markov chain with transition matrix $P$, invariant distribution $\pi$ and spectral gap $\gamma$. Let $V_1,\ldots,V_M$ be a partition of $V$ and let $P_i$ be the transition matrix on $V_i$ with off-diagonal transitions $P_i(x,y)=P(x,y)$ for all $x\neq y \in V_i$ and $P_i(x,x)=1-\sum_{z \in V_i \setminus \{x \}}P(x,z)$. Denote its spectral gap  by $\gamma(P_i)$ and let $\gamma_{*}:=\min_{i \in [M]}\gamma(P_i)$. Let $\widehat P$ be a Markov chain on $[M]$ with transition probabilities given by 
\begin{align}\label{eq:defphat}
\widehat P(i,j)=\prcond{X_1 \in V_j}{ X_0 \in V_{i}}{\pi}=\sum_{x \in V_i}\frac{\pi(x)}{\pi(V_i)}P(x,V_j).
\end{align}
and spectral gap given by $\widehat \gamma$. Then 
\begin{equation}
\label{e:decom}
\gamma \ge   \widehat\gamma \gamma_{*}. 
\end{equation}
\end{theorem}

\begin{proof}[Proof of Proposition~\ref{prop:TRel}] ($t_{\rm{rel}}$ for $m$-communities model)
        We apply the previous theorem to the partition of $V$ corresponding to the communities. As the internal degree of each vertex is at least $3$ it follows that each community without considering outgoing edges is an expander with high probability \cite{expander}. Using the notation from the previous theorem and the bounded degree assumption implies that $\gamma(P_i)\gtrsim 1,$ for all $i$ and so $\gamma_*\gtrsim 1.$ Corollary~\ref{cor:cheegerQpoincareQ} gives that $\widehat{\gamma}\gtrsim \Phi_*^Q=\alpha,$ as the chain $\widehat{P}$ from the previous theorem is exactly the chain $Q$ from Corollary~\ref{cor:cheegerQpoincareQ}. The lower bound on $t_{\mathrm{rel}}$ follows by Cheeger's inequality. Indeed, from~\cite[Theorem 13.10]{MixingBook} it follows that $t_{\mathrm{rel}}\gtrsim \frac{1}{\Phi(A,A^c)}$, for any set $A$, where $\Phi(A,A^c)=\frac{\sum_{x\in A, y\in A^c} \pi(x)P(x,y)}{\pi(A)}.$ The proof of the lower bound then follows by choosing $A$ to be the union of communities achieving the bottleneck ratio of $Q$. 
           This completes the proof. 
\end{proof}

Theorem~\ref{thm:madrasrandall} cannot be used for the 2-communities model, as we cannot control the relaxation time of a random walk restricted to  one community. Instead we bound the relaxation time by controlling the spectral profile which we now recall. 

Let $P$ be a transition matrix on a finite state space $S$ which is reversible with respect to the invariant distribution $\pi$. For $A\subset S$, we let $P_A$ be the restriction of $P$ to $A$, i.e.\ $P_A(x,y)=P(x,y)$ for all $x,y\in A$. As the matrix $(\sqrt{\frac{\pi(x)}{\pi(y)}}P(x,y))_{x,y\in A}$ is symmetric, it is also diagonalisable, and therefore $P_A$ is diagonalisable so we can define the Dirichlet eigenvalue and the spectral profile as follows.

\begin{definition}\label{def:DirEV} The Dirichlet eigenvalue of $A$ denoted by $\lambda(A)=\lambda_{P}(A)$ is defined as $$\lambda\left(A\right)=1-\max\{\lambda: \lambda \mathrm{\, \, eigenvalue \, of \,} P_A\}.$$ 
Writing $\pi_*=\min_x \pi(x)$, we define the spectral profile $\Lambda_{P}:\left[\pi_*,\infty\right)\to \mathbb{R}$ of $P$ as $$\Lambda_{P}\left(r\right)=\inf_{\pi_*\le \pi(A)\le r}\lambda(A).$$ 
\end{definition}

The following lemma relating the Dirichlet eigenvalue of a set $A$ to the tail of the hitting probability of $A$ is a standard result and can be obtained using the spectral theorem. 
It will be very useful to us as in our setting it is easier to control hitting times rather than eigenvalues. 

The following lemma follows directly from the spectral theorem for $P_A$. 
\begin{lemma}\label{lem:dirichletdef}
Let $A\subseteq S$ and write $\pi_A$ for the measure $\pi$ conditional on $A$, i.e.\ $\pi_A(x)=\pi(x)/\pi(A)$ for $x\in A$. We have $$\frac{\log \prstart{T_{A^c}>t}{\pi_A}}{t} \to \log \left(1- \lambda\left(A\right)\right)\quad \text{as}\quad t \to \infty.$$ 
\end{lemma}

In the following lemma we obtain a lower bound on the spectral profile for small subsets of $G_n$. 

\begin{lemma}\label{lemma:SpectralProfileSmallSets}
There exist constants $\hat{c}, \tilde{c}>0$ such that the graph $G_n$ for both models (two and $m$-communities) satisfies with high probability that for every $r\le\hat{c}$ it holds that $\Lambda(r)\ge \tilde{c}$.  
\end{lemma}
\begin{remark}\rm{ In fact for all $\delta>0$, the lemma  above holds by taking $\hat{c}=(1-\delta)\min_{i\le 2}\pi(V_i)$ and some constant $\tilde{c}=\tilde{c}(\delta)>0$.}
\end{remark}

\begin{definition} \label{def:expander} For a graph $G=(V,E)$ and $A\subset V,$ the edge boundary of $A$, labeled by $\partial A$ is a set of edges between $A$ and $A^c.$ For constant $\varepsilon>0$ we say that $G$ is an $\varepsilon-$expander if  $$\min_{A\subset V,\pi(A)\le \frac{1}{2}}\frac{|\partial A|}{|A|}>\varepsilon,$$ where $\pi$ is the invariant distribution of simple random walk on $G.$ We say that the collection of graphs $G_n$ is an expander family if there is $\varepsilon>0$ such that each $G_n$ is an $\varepsilon-$expander.
\end{definition}
\begin{lemma}\label{lemma:SmallSetsBoundary} Let $G_n$ be a $2$-community model. There exist constants $\hat{c}, \delta>0$ such that the graph~$G_n$ satisfies with high probability that for every subset of vertices $D$ with $\pi_{G_n}(D)\le\hat{c}$ it holds that $|\partial D|\ge \delta |D|.$ 
\end{lemma}
{\proof The proof of this lemma is given in Appendix~\ref{appendixSmallSets}. 
\qed

\begin{proof}[Proof of Lemma~\ref{lemma:SpectralProfileSmallSets}]

 Recall that the conductance profile, $(\Phi(r))_r$, is defined 
to be 
\[
\Phi(r) = \min\left\{\frac{|\partial A|}{|A|}: \pi(A)\leq r \right\}.
\]
By~\cite[Lemma 2.4.]{SpectralProfilePaper}  we have 
\[
\frac{\Phi^2\left(r\right)}{2}\le \Lambda\left(r\right).
\]
Therefore, it is enough to bound the conductance profile. In the case of two communities this is the content of Lemma~\ref{lemma:SmallSetsBoundary}.
In the case of more communities this follows as restricted to one community the random graph is a configuration model with degree at each vertex at least three, which implies that it is an expander with high probability \cite{expander}. Therefore if $\hat{c}$ is small enough and $D$ is a set with $\pi_{G_n}(D)\le \hat{c}$, then $|D\cap V_i|\leq n_i/2$, and hence, the expander property gives us that the size of the boundary in community $i$ is bounded by a constant proportion of $|D\cap V_i|.$ This gives that $\Phi(r)\gtrsim 1$ for $r\le \hat{c}$ and completes the proof.
\end{proof}

\begin{prop}\label{prop:2comrelax}
There exists a positive constant $c$ so that for all $n$, if $G_n$ is a two-community graph and $\alpha \gtrsim 1/\log n$, then with high probability
\[
\Lambda(1/2)\geq c \cdot \alpha.
\]
\end{prop}

\begin{proof}[Proof of Proposition~\ref{prop:TRel}]($t_{\rm rel}$ for $2$-communities model when $\alpha\gtrsim 1/\log n$)

        It is proved in~\cite[Lemma 2.2.]{SpectralProfilePaper} that $2\frac{1}{t_{\mathrm{rel}}}\ge \Lambda\left(\frac{1}{2}\right)$ so Proposition~\ref{prop:2comrelax} gives that $t_{\mathrm{rel}}\lesssim \frac{1}{\alpha}$ with high probability. From~\cite[Theorem 13.10]{MixingBook} it follows that $t_{\mathrm{rel}}\gtrsim \frac{1}{\Phi(A,A^c)}$, for any set $A$, where $\Phi(A,A^c)=\frac{\sum_{x\in A, y\in A^c} \pi(x)P(x,y)}{\pi(A)}.$ The proof of the lower bound then follows by choosing $A$ to be the first community, as the proportion of edges between the two communities is~$\alpha$.
\end{proof}

\begin{remark}
	\rm{
	Note that the statement of Proposition~\ref{prop:2comrelax} remains true also when $\alpha\ll 1/\log n$, as it is implied by the upper bound on $t_{\mathrm{mix}}$ proved in Section~\ref{sec5}. 
	}
\end{remark}

The rest of the section is devoted to the proof of Proposition~\ref{prop:2comrelax} for the graph with two communities. In the next section we bound the Dirichlet eigenvalue of each community and in Section~\ref{section:DEVcontains} we bound the Dirichlet eigenvalue of sets containing one community.

\subsection{Dirichlet eigenvalue of one community}\label{section:DEV1com}

\begin{lemma} \label{lemma:DirEVCom1}
There exists a positive constant $c$ so that for the $2$-community model for all $n$ with high probability 
$$\lambda\left(V_1\right)\ge c\alpha \text{ and } \lambda(V_2)\geq c\alpha.$$  
\end{lemma}
\begin{proof}
By Lemma~\ref{lem:dirichletdef} it is enough to find a positive constant $c$ such that with high probability for all  $t$ larger than some constant, we have $\prcond{T_{V_2}>t}{G_n}{\pi_{V_1}} \leq e^{-{c}\alpha t}.$
\\Let $D$ be the subset of $V_1$ consisting of those vertices connected to community $2$. Then we have 
 \begin{align*}\prcond{T_{V_2}>t}{G_n}{\pi_{V_1}}=& \prcond{\sum_{i=1}^t\mathds{1}\left(X_i\in D\right)< \frac{t\pi_{V_1}\left(D\right)}{2},T_{V_2}>t}{G_n}{\pi_{V_1}}
 \\ &+\prcond{\sum_{i=1}^t\mathds{1}\left(X_i\in D\right)\ge \frac{t\pi_{V_1}\left(D\right)}{2},T_{V_2}>t}{G_n}{\pi_{V_1}}.
  \end{align*}
  Once the chain is in $D$ the probability to hit $V_2$ in the next move is at least $\frac{1}{2\Delta}$, as there is at least one edge going to $V_2$ from every vertex in $D$. Setting $k=\lfloor\frac{t\pi_{V_1}\left(D\right)}{2}\rfloor-1$ and writing $R_1,\ldots,R_k$ to be the times of the first $k$ visits to the set $D$, we get for any realisation $G$ of $G_n$ 
\begin{align*}
        &\prcond{\sum_{i=1}^t\mathds{1}\left(X_i\in D\right)\ge \frac{t\pi_{V_1}\left(D\right)}{2},T_{V_2}>t}{G_n=G}{\pi_{V_1}} \leq \prcond{T_{V_2}>R_k+1}{G_n=G}{\pi_{V_1}}
        \\&=\prcond{T_{V_2}>R_k+1}{T_{V_2}>R_{k-1}+1, G_n=G}{\pi_{V_1}}\prcond{T_{V_2}>R_{k-1}+1}{G_n=G}{\pi_{V_1}}\\
        &\leq \left(1-\frac{1}{2\Delta}\right)\prcond{T_{V_2}>R_{k-1}+1}{G_n=G}{\pi_{V_1}}\le \left(1-\frac{1}{2\Delta}\right)^k \leq \exp\left(- c_1 \alpha t\right),
\end{align*}
        for large $t$ and a positive constant $c_1$,
         where the penultimate inequality follows by induction and the last inequality follows since $\pi_{V_1}(D)\asymp \alpha$. 
\\It suffices to prove that for a positive constant $c_2$ with high probability for all large $t$
\begin{equation} \label{eq:ProbNotManyVisitsToD} \prcond{\sum_{i=1}^t\mathds{1}\left(X_i\in D\right)< \frac{t\pi_{V_1}\left(D\right)}{2},T_{V_2}>t}{G_n}{\pi_{V_1}}\leq  \exp\left(-{c_2}\pi_{V_1}\left(D\right)t\right),
\end{equation}
In order to bound this probability we are going to consider a new graph $\widehat{G}_n$ and a walk $\widehat{X}$ on $\widehat{G}_n$ which is coupled with $X$ in a natural way.  
 For $i\in\{1,2\}$ let $E_i=\left\{e^i_1,e^i_2, \ldots, e^i_{N_i}\right\}$ be all of the half-edges in community $i$ and let $\tau_i$ be permutations chosen uniformly at random among all permutations of~$\left\{1,\ldots, N_i\right\}.$ Then $G_n$ can be generating by choosing the outgoing half-edges in community $i\in \{1,2\}$ to be exactly the edges $E_i^{O}=\left\{e^i_{\tau_i(1)}, e^i_{\tau_i(2)}\ldots, e_{\tau_i(p)}^i\right\}$. The half-edges are further paired by matching the edge $e^1_{\tau_1(j)}$ with the edge $e^2_{\tau_2(j)}$
for $j\le p$ and for $j> \frac{p}{2}$ and $i\in \{1,2\}$ matching $e^i_{\tau_i(2j-1)}$ with $e^i_{\tau_i(2j)}$. We now define the graph $\widehat{G}_n$ by instead for $i\in\{1,2\}$ and $j\le \frac{N_i}{2}$ matching the half-edge $e^i_{\tau_i(2j-1)}$ with the half-edge $e^i_{\tau_i(2j)}.$  Note that $\widehat{G}_n$ has two disconnected components corresponding to each community. Moreover, all the internal half-edges of $G_n$ are matched in exactly the same way in~$\widehat{G}_n$ and the outgoing half-edges in $\widehat{G}_n$ in each community are matched uniformly at random inside the community.

Further we couple the walks $X$ on $G_n$ and $\widehat{X}$ on $\widehat{G}_n$, by first starting both walks from the same vertex chosen according to $\pi_{V_1}$, which is also the invariant distribution of the graph $\widehat{G}_n$ restricted to the first community.  At time $t-1$, if the walks are in the same vertex, with probability $\frac{1}{2}$ they both move at time $t$ along the same uniformly chosen half-edge which we label by $e(t)$, otherwise they both stay at the same vertex. The coupling fails the first time the walks are in different vertices. As half-edges in $E_1\setminus E_1^O$ are matched with the same half-edges in $G_n$ and in $\widetilde{G}_n$, the coupling fails at time $$T_{E_1^O}=\inf\left\{t\ge1: X_t\ne\widehat{X}_t\right\}=\inf\left\{t: e(t)\in {E_1^O} \right\}=\inf \left\{t:X_t\in V_2\right\}=T_{V_2}.$$ 
Therefore, since the randomness on $\widehat{G}_n$ only comes from the permutation $\tau_1$, we get 
\begin{align*}
        & \prcond{\sum_{i=1}^t\mathds{1}\left(X_i\in D\right)<\frac{t\pi_{V_1}\left(D\right)}{2},T_{V_2}>t}{G_n}{\pi_{V_1}} \\&\leq   \prcond{\sum_{i=1}^t\mathds{1}\left(\widehat{X}_i\in D\right)<\frac{t\pi_{V_1}\left(D\right)}{2}}{\tau_1}{\pi_{V_1}}. 
\end{align*}
As $\tau_1$ is chosen uniformly at random from all permutations of $\left\{1,\ldots N_1\right\}$ the graph $\widehat{G}_n$ restricted to the first community is exactly the configuration model. It is well known that the configuration model with minimum degree at least $3$ is with high probability an expander \cite{expander}. Therefore the relaxation time, $\widehat{t}_{\rm{rel}},$ of the lazy simple random walk $\widehat{X}$ on $\widehat{G}_n$,  is with high probability of order $1$.

 We now use Theorem 1.1 from Lezaud~\cite{Lezaud} for  function $f\left(x\right)=\pi_{V_1}\left(D\right)-\mathds{1}\left(x\in D\right)$, $\gamma=\pi_{V_1}\left(D\right)/2$ and $b^2=\pi_{V_1}\left(D\right)\left(1-\pi_{V_1}\left(D\right)\right)$. Notice that the conditions $\pi_{V_1}f=0$, $\|f\|_{\infty}\le 1$ and $\|f\|_2^2\le b^2$ are satisfied and that $N_q$ and $\varepsilon\left(P\right)=\frac{1}{\widehat{t}_{\text{rel}}}$ from the theorem, are both of order 1. Therefore, with high probability 
\begin{align*}
         &\prcond{t^{-1}\sum_{i=1}^t\left(\pi_{V_1}\left(D\right)-\mathds{1}\left(\widehat{X}_i\in D\right)\right)>\frac{\pi_{V_1}\left(D\right)}{2}}{\tau_1}{\pi_{V_1}} \\& \le \widetilde{C}\exp\left(-\frac{t \left(\pi_{V_1}\left(D\right)\right)^2}{16\pi_{V_1}\left(D\right)\left(1-\pi_{V_1}\left(D\right)\right)\left(1+h\left(\frac{5\pi_{V_1}\left(D\right)}{2\pi_{V_1}\left(D\right)\left(1-\pi_{V_1}\left(D\right)\right)}\right)\right)}\right)
         \end{align*}
 for $h\left(x\right)=\frac{1}{2} \left(\sqrt{1+x}-\left(1-\frac{x}{2}\right)\right)$ and a positive constant $\widetilde{C}$. We can rewrite this as
 \begin{align*}
  &\prcond{\sum_{i=1}^t\mathds{1}\left(\widehat{X}_i\in D\right)<\frac{t\pi_{V_1}\left(D\right)}{2}}{\tau_1}{\pi_{V_1}}
  \\ &\le\widetilde{C} \exp\left(-\frac{t\pi_{V_1}\left(D\right) }{16\left(1-\pi_{V_1}\left(D\right)\right)\left(1+h\left(\frac{5}{2(1-\pi_{V_1}\left(D\right))}\right)\right)}\right)\leq \widetilde{C}\exp\left(-\tilde{c}\pi_{V_1}\left(D\right)t\right),
  \end{align*} 
  with high probability for a suitable positive constant $\tilde{c}$, which completes the proof.  
  \end{proof}

\subsection{Dirichlet eigenvalue for sets containing one community}\label{section:DEVcontains}
The following lemma gives a bound on the Dirichlet eigenvalue of sets containing all vertices from the first community and some from the second. Because of the symmetry of the problem, the analogous result holds for sets containing all vertices in the second community and some from the first one. 

\begin{lemma}\label{lemma:DirEVContainingCom1} Let $G_n$ be the $2$-community model and let $\alpha \gtrsim 1/\log n$. Let $D$ be the set containing the first community and some vertices from the second community which we label by $D_{2}$. Let $\pi(D_2)\leq \hat{c} \frac{n_2}{n}$ for some constant $\hat{c}<1$. Then there is a positive constant $\tilde{c}$ such that with high probability  $$ \lambda\left(D\right)\ge \tilde{c}\alpha.$$ 
\end{lemma}
In order to prove the above lemma we first bound the probability that the proportion of time spent in the second community by some time $t$, when starting from a $K-$root (vertex around which the graph looks like a tree for $K$ levels--see Definition~\ref{def:KRoot}) in the first community, is smaller than some constant $\delta$.

 \begin{definition}\label{def:KRoot} We call a vertex $x$ a $K-$root of $G_{n}$ if $\mathcal{B}_{K}(x)=\mathcal{B}^{G_n}_{K}(x)$ from  Definition  \ref{def:KBall} is a possible realisation of the first $K$ levels of the $2$-type \random tree $T$ (corresponding to $G_{n}$). If $x$ is a $K-$root and $i \leq K,$ we denote by $\partial \mathcal{B}_{i}(x)$ the collection of vertices at distance $i$ from $x$.
\end{definition}

\begin{lemma}\label{lemma:TimeInCom2fromKrootOfCom1}
For $u>0$ define $T_2(u)$ to be the total time spent in the second community by time~$u$ and let $x$ be a vertex in $G_n$. 
There exist positive constants $C_1,$ $\varepsilon_0=\varepsilon(\Delta)$ and $\delta<1$ so that for all $\varepsilon<\varepsilon_0$  the following holds. 
If $\alpha=o(1)$, for every $c>0$ there exists $C>0$, so that taking $t=C\log n$ and  $K=\lfloor\varepsilon \log n\rfloor$, then with high probability we have 
\[\forall \, x, \quad \mathds{1}(x \text{ is a $K$-root} )\cdot \prcond{T_2(t)< \delta t}{G_n}{x} \le \frac{C_1}{n^{c \alpha}}. 
\]
If $\alpha\asymp 1$, then there exist positive constants $c'=c'(\varepsilon)$ and $C$ so that with $t$ and $K$ as above with high probability we have 
\[\forall \, x, \quad \mathds{1}(x \text{ is a $K$-root} )\cdot \prcond{T_2(t)< \delta t}{G_n}{x} \le \frac{C_1}{n^{c'\varepsilon \alpha}}. 
\]
\end{lemma}

\begin{corollary}\label{cor:notkroot}
The same statement as Lemma~\ref{lemma:TimeInCom2fromKrootOfCom1} holds for every $x\in G_n$ which is not necessarily a $K$-root. 
\end{corollary}

\begin{proof}

         We already established above for a $K$-root $x$ in Lemma \ref{lemma:TimeInCom2fromKrootOfCom1}. 
The result now follows from the fact that the probability that a $K-$root is not reached by time of order $\log n$ is bounded by~$\frac{1}{n^c}$ with high probability. Indeed, the ball $\B_{5K}(x)$ can be generated by revealing the type of the half-edge and its pair one at a time until we reach level $5K$. As at most $k=\frac{\Delta\left(\left(\Delta-1\right)^{5K}-1\right)}{\Delta-2}$ pairs are formed and at each step we have probability at most $\max_{i\in \{1,2\}}\left(\frac{\Delta(\Delta-1)^{5K}-1}{N_i-2k}\right)$ to create a cycle, we get that $\B_{5K}(x)$ contains at least two cycles with probability $O\left(\frac{\Delta^{4\cdot 5K}}{N^2}\right)=O \left(N^{20\varepsilon-2}\right).$ A union bound then gives that with probability $O(\frac{1}{N^{20\varepsilon-1}})$ the graph has no vertices $x$ whose $K$ ball contains two cycles. Taking $\varepsilon<\frac{1}{20}$, the proof now directly follows from Lemma 2.3  in~\cite{NBRWvsSRW}. 
\end{proof}

We now give a brief overview of the proof of Lemma~\ref{lemma:TimeInCom2fromKrootOfCom1} and of the importance of working with $K$-roots. Using a coupling of the random graph with a tree we prove the desired estimate for a typical point (see Lemma~\ref{lemma:T2conditional} for a precise statement). To upgrade the result from a typical point to all points, we need to first prove it for all $K$-roots, and then as explained in Corollary~\ref{cor:notkroot} above
 it can be easily extended to all points. The advantage of working with a $K$-root is that starting from there, we show that with high probability at the first hitting time of the boundary of the tree rooted at the $K$-root the walk will be at a \emph{good} point, i.e.\ a vertex with the property that starting a walk from there we have a good control on the time it spends in community $2$. This is achieved using that the hitting distribution of the boundary of the tree is sufficiently spread out and that starting from a typical point the walk spends a constant proportion of time in community $2$ with high probability. With these two ingredients, we then conclude the proof using a concentration result.

Before proving Lemma~\ref{lemma:TimeInCom2fromKrootOfCom1}, we define the following coupling and exploration process. 

\begin{definition}\label{def:CouplingFirst} For a $K-$root $x$ in $G_n$ and a tree $T_0$ with $\pr{\mathcal{B}_K(x)=T_0}>0$, we define the exploration process of the graph $G_n$ conditional on $\mathcal{B}_K(x)=T_0$ (this includes conditioning on the types and degree of vertices in $\partial \mathcal{B}_K(x)$ but not on the value of their internal and outgoing degree, which will be chosen later at random),  by coupling $G_n$ with the multi-type \random tree $T$ with root~${x}$ and  $\mathcal{B}_K\left({x}\right)=T_0$ in the following way. We first label all vertices of $\partial\mathcal{B}_{\lfloor\frac{K}{2}\rfloor}(x)$ by $z_1,\ldots, z_L$ and define $V_{z_i}$ to be all descendants of $z_i$ in $\partial\mathcal{B}_{K}(x).$ For $i=1,\ldots, L$ and for all $z\in V_{z_i}$ we define the exploration process of $G_n$ corresponding to $z$ by coupling a simple random walk $X$ on $G_n$ started from $z$,  and $\widetilde{X}$ on $T$ started from ${z}.$ For any $t\ge 1$ the coupling is performed as follows: If the coupling has succeeded so far and the walk has moved to a vertex which was not seen before but whose total degree is revealed, we proceed by performing the optimal couplings for the distributions for each of the following steps: 

 (1) for each half-edge going from vertex $\widetilde{X}_t$ of type $i$ in $T$, we choose it to have probability $\frac{p}{N_i}$ to be outgoing, while for the corresponding vertex $X_t$ in $G_n$ we choose this with probability $\frac{p-p(t)}{N_i-N_i(t)}$ where $p(t)$ is the number of outgoing half-edges revealed in community $i$ by time $t$ and $N_i(t)$ is the number of half-edges in community $i$ revealed by time $t$.
 
(2) In $T$ we match each of these half-edges from $\widetilde{X}_t$ to a uniformly chosen half-edge of the $n_i$ vertices in the appropriate community, while in $G_n$ we match each half-edge from $X_t$ to the half-edges (corresponding to the vertices of the appropriate type) which had not already been matched previously.

If the two optimal couplings above succeed, then we move both walks from $\widetilde{X}_t$ and $X_t$ to the same uniformly chosen neighbour.  

After we have revealed the matchings of all the half-edges from  we perform the optimal coupling of the next step of these two random walks. 
 
The coupling fails if one of the following happens: the optimal couplings of steps 1) or 2) fail (i.e. if we either do not choose the same types for half-edges or do not connect them to the same vertex), if we created a cycle in $G_n$, or if the walk $\widetilde{X}$ reaches level $\frac{K}{2}$ in $T_0$. The exploration process from~$z$ ends when the coupling fails or after $C\log n$ steps for a constant $C$ to be determined later.

 We define $\mathcal{F}_i$ to be the $\sigma$-algebra generated by the exploration processes corresponding to vertices $z\in \cup_{j=1}^iV_{z_j}.$ We  call $z_i\in \partial\mathcal{B}_{\frac{K}{2}}\left(x\right)$ \emph{good} if none of its descendants in $\partial T_0$ have been  explored in the exploration processes of $V_{z_1},\ldots V_{z_{i-1}}$. Otherwise $z$ is called \emph{bad}. 
\end{definition}

We first prove the following lemma, where we recall $T_2(u)$ is the total time spent in the second community by time~$u$. 

\begin{lemma}\label{lemma:T2conditional}
There exist a positive constant $C_1,$ $\varepsilon_0=\varepsilon(\Delta)$ and $\delta<1$ so that the following holds for all $\varepsilon<\varepsilon_0$. For every $c>0$ there exists $C>0$, so that if $t=C\log n$ and  $K=\lfloor\varepsilon \log n\rfloor$, then on the event $\{x \text{ is a $K$-root in } G_n\} \cap \{z_i\text{ is good}\},$  for $z\in V_{z_i}$, we have 
\[\prcond{T_2\left(\frac{t}{2}\right)< \delta t}{\mathcal{F}_{i-1}}{z} \le \frac{C_1}{n^{c\alpha}}
\]
for all $n$ sufficiently large.
\end{lemma}

\begin{proof} We use the coupling from Definition~\ref{def:CouplingFirst} for $C\log n$ steps. If it has succeeded, then we explore all the unexplored parts of the tree $T$ and we continue running the walk $\widetilde{X}$ forever as this will not affect the events of interest for $X$ up to time $t$. 

On the event $\{z_i \text{ is good}\}$, we can bound 
\begin{align}\label{eq:timeincom2}
\begin{split}   
&\prcond{T_2\left(\frac{t}{2}\right)< \delta t}{\mathcal{F}_{i-1}}{z}\leq \prcond{\text{coupling failed}}{\mathcal{F}_{i-1}}{z} 
        \\ &+ \prcond{\widetilde{X}\text{ has less than } \delta t\text{ regeneration times in community 2 by } \frac{t}{2}}{\mathcal{F}_{i-1}}{z}.      
\end{split}
\end{align}
 We first bound the probability that the coupling with the tree failed as follows. 
Firstly, the probability that the walk backtracked to level $\frac{K}{2}$ is by Lemma \ref{lemma:ReturnProb} bounded by~$n^{-c_1\varepsilon}$ for a suitable constant~$c_1$.
The probability of failure in step (1) of the coupling is the probability that at some step an edge in the tree was assigned a different type to the one in the graph. The probability of this happening at time $s$ when the two walks are in community $i$ is equal to the total variation distance between two Bernoulli random variables with parameters $\frac{p}{N_i}$ and $\frac{p-p(s)}{N_i-N_i(s)}$ which is equal to 
\[
\left|\frac{p}{N_i} - \frac{p-p(s)}{N_i-N_i(s)} \right|.
\] 
As the number of half-edges in $T_0$ is of order $n^{\varepsilon\log\Delta}$, and for each vertex in $\partial T_0$ we reveal at most $O(\log n)$ half edges during the exploration process, we have that the total number of explored half-edges is $O(n^{\varepsilon\log\Delta} \log n).$ Therefore,
\[
\left|\frac{p}{N_i} - \frac{p-p(s)}{N_i-N_i(s)} \right|\leq O\left(\frac{n^{\varepsilon\log\Delta} \log n}{n}\right).
\] 
Taking a union bound over all $C\log n$ steps of the exploration process we get that the probability of the coupling failing in step (1) is bounded by $O\left(n^{\varepsilon\log\Delta}(\log n)^2/n\right)\leq n^{-c_1},$ for a suitable positive constant $c_1$ and $n$ large enough.  

The probability that step (2) of the coupling fails can be upper bounded similarly. The probability that step (2) fails at time $s$ when both walks are in community $i$ is equal to the total variation distance between two uniform random variables on $\{1,\ldots,N_i\}$ and $\{1,\ldots,N_i-N_i(s)\}$. This total variation distance is equal to $N_i(s)/N_i$. Using the bound from above and taking a union bound over all time steps gives that the probability that step (2) of the coupling fails is upper bounded by $O\left(n^{\varepsilon\log\Delta}(\log n)^2/n\right)\leq n^{-c_1}.$

The last possibility for the coupling to fail is if we created a cycle. As above this is bounded by the probability that we chose to connect to an already revealed vertex or one of its neighbours. Using the bounded degree assumption we thus obtain the same bound as in step (2), i.e.\ $O\left(n^{\varepsilon \log\Delta}(\log n)^2/n\right)\leq n^{-c_1}.$

Therefore, we conclude that for a suitable constant $\tilde{c}>0$ we have that 
\[\prcond{\text{coupling\,failed}}{\mathcal{F}_{i-1}}{z}\le n^{-\tilde{c}}.
\]
 It is left to bound the second probability appearing on the right hand side of~\eqref{eq:timeincom2}, i.e.\ the probability that the random walk on the \random tree had less than $\delta t$ regeneration times in community $2$ by time $t/2$. 
     As in Lemma~\ref{lemma:RegIndep} we take $\sigma_1, \sigma_2,\ldots$ to be the regeneration times of $\widetilde{X}$ that occur after level $K$. We first note that if $\widetilde{X}_{\sigma_1}$ is among the descendants of $z_i$ (which happens with probability at least $1-e^{-cK}$ for a positive $c$), then the conditioning on $\mathcal{F}_{i-1}$ is irrelevant, since it concerns the part of the tree the walk will never visit. So using Chernoff's bound we get that for any $\delta_1>0$
          \begin{align}\label{eq:regspentincom2}
          \begin{split}
                 \prcond{\sigma_{\delta_1 t}>\frac{t}{2}}{\mathcal{F}_{i-1}}{z}\le e^{-cK}&+
    e^{-\theta\frac{t}{2}}\max_{a\in\{1,2\}^2}\escond{e^{\theta\sigma_1}}{\Theta\left(\widetilde{X}_{\sigma_1}\right)=a}{z}\\ &\cdot\left(\max_{a,b\in\{1,2\}^2}\escond{e^{\theta(\sigma_{2}-\sigma_1)}}{\Theta\left(\widetilde{X}_{\sigma_1}\right)=a,\Theta\left(\widetilde{X}_{\sigma_2}\right)=b}{z}\right)^{\delta_1t}.
          \end{split}   
          \end{align}
  Here we have also used that after conditioning on the types of all of the first $\delta_1 t$ regeneration times we have by Lemma~\ref{lemma:RegIndep} that $\sigma_s-\sigma_{s-1}$ for $s\in \{1,\ldots,\delta_1t\}$ are independent. Using that $(\sigma_2-\sigma_1)$ has exponential tails it is standard to find $\theta$ and $\delta_1$ sufficiently small so that the second term on the right hand side of~\eqref{eq:regspentincom2} becomes smaller than $e^{-c_1t}$ for a positive constant $c_1$, and hence for a positive constant $c_2$ we get
  \begin{align}\label{eq:lessreg}
        \prcond{\sigma_{\delta_1 t}>\frac{t}{2}}{\mathcal{F}_{i-1}}{z}\le n^{-c_2}.
  \end{align}
   Let $\mu$ be the invariant distribution of the types Markov chain as in Definition~\ref{def:MCRegTypes}. Then
  \begin{align*}
\prcond{\widetilde{X}\text{ has\,}\leq \delta t\text{\,regeneration\,times\,in\,community\,2\,by\,} \frac{t}{2}}{\mathcal{F}_{i-1}}{z} \leq \prcond{\sigma_{\delta_1 t}>\frac{t}{2}}{\F_{i-1}}{z} \\+ 
        \prstart{\sum_{i=1}^{\delta_1t}\left(\mu(2,2)-\mathds{1}\left(\Theta(\widetilde{X}_{\sigma_i})=(2,2)\right)\right)>\left(\mu(2,2)- \frac{\delta}{\delta_1}\right)\delta_1t}{z}.
 \end{align*}
 Using that $\mu(2,2)\asymp 1/2$, Lemma~\ref{lemma:2comMixTypes}, which implies an upper bound on the relaxation time of the types Markov chain, and~\cite[Theorem~1.1]{Lezaud} we obtain that there is a constant $c'>0$ such that for $\delta>0$ sufficiently small we have 
 \[
 \prstart{\sum_{i=1}^{\delta_1t}\left(\mu(2,2)-\mathds{1}\left(\Theta(\widetilde{X}_{\sigma_i})=(2,2)\right)\right)>\left(\mu(2,2)- \frac{\delta}{\delta_1}\right)\delta_1t}{z}\lesssim \exp\left(-c'\delta_1t\alpha\right).
 \]
 Taking $C$ in the definition of $t$ sufficiently large we can make the exponential above less than $n^{-c\alpha}$. This together with~\eqref{eq:lessreg} concludes the proof.
 \end{proof}

 \begin{proof}[Proof of Lemma \ref{lemma:TimeInCom2fromKrootOfCom1}] 
We first set $h\left(z\right)=\prcond{X_{\tau_{\partial \mathcal{B}_K\left(x\right)}}=z}{G_n}{x}$ and we work on the event that~$x$ is a $K$-root. Then we have  
\begin{align*}
        \prcond{T_2(t)< \delta t}{G_n}{x}&\le \prcond{T_2(t)-T_2(\tau_{\partial \mathcal{B}_K\left(x\right)})< \delta t}{G_n}{x} \\ 
        &\leq \sum_{z\in \partial \mathcal{B}_K\left(x\right)}h\left(z\right) \prcond{T_2\left(\frac{t}{2}\right)< \delta t}{G_n}{z}+ \prcond{\tau_{\partial\mathcal{B}_K\left(x\right)}\ge\frac{t}{2}}{G_n}{x}.
\end{align*}
By Lemma \ref{lemma:ReturnProb} we get that for $\varepsilon$ (in the definition of $K$) sufficiently small   
\begin{equation}\label{eq:ProbT2Bound} \prcond{T_2(t)< \delta t}{G_n}{x}\leq\sum_{z\in \partial \mathcal{B}_K\left(x\right)}h\left(z\right) \prcond{T_2\left(\frac{t}{2}\right)< \delta t}{G_n}{z}+ n^{-c_1},
\end{equation}
where $c_1$ is a positive constant. We turn to bounding the first term on the right hand side above. 
For $\beta$ to be determined later we define the set 
\[
V=\left\{z\in\partial \mathcal{B}_K\left(x\right):\prcond{T_2\left(\frac{t}{2}\right)< \delta t}{G_n}{z}>n^{-\beta\alpha} \right\}.
\]
We then have 
\begin{equation}\label{eq:BoundSumh}\sum_{z\in \partial \mathcal{B}_K\left(x\right)}h\left(z\right) \prcond{T_2\left(\frac{t}{2}\right)< \delta t}{G_n}{z}\le\sum_{z\in V}h\left(z\right)+n^{-\beta\alpha}.
\end{equation}

Let $T_0$ be a tree with $\pr{\B_K(x)=T_0}>0$. We now work on the event $\{\B_K(x)=T_0\}$. Using the definition of good and bad vertices from Definition \ref{def:CouplingFirst} and writing $h\left(U\right)=\sum_{z\in U}h\left(z\right)$ for any set~$U$, we obtain 
\[
\sum_{z\in V}h\left(z\right)= \sum_{i=1}^L\left(h\left(V\cap V_{z_i}\right)\mathds{1}\left(z_i\text{ is\,good}\right)+h\left(V\cap V_{z_i}\right)\mathds{1}\left(z_i\text{ is\,bad}\right)\right).
\]

 We first bound the number of bad vertices as follows. Notice that the exploration process continues for time of order $\log\left(n\right)$ and the number of leaves in the tree $T_0$ is of order $n^{\varepsilon\log(\Delta)}$ giving that we explore at most order $\log(n)n^{\log(\Delta)\varepsilon}\le n^{c' \varepsilon}$ vertices, for suitable $c'$ and large $n$. The probability of connecting to a vertex in $\partial T_0$ at any step is also bounded by $\frac{O(n^{\log(\Delta)\varepsilon})}{n}\le \frac{n^{c''\varepsilon}}{n}$, for suitable $c''$ and large $n$. This gives  that for $R>4$, taking $\varepsilon$ sufficiently small we obtain
\begin{align}\label{eq:badvertices}
\prcond{\left| \left\{z_i\in \partial\mathcal{B}_{\frac{K}{2}}\left(x\right): z_i\text{ is\,bad} \right\} \right| \ge R}{ \mathcal{B}_K\left(x\right)=T_0}{}\le {n^{c'\varepsilon} \choose R}\left(\frac{n^{c''\varepsilon}}{n}\right)^R\le \frac{1}{n^2}.
\end{align} 
We now claim that for a positive constant $c_2$ 
\begin{equation}\label{eq:h(Vzi)}  h\left(V_{z_i}\right)\le n^{-\varepsilon c_2}. \end{equation}  Indeed, this follows from the observation that if $\xi$ is the loop erasure of $(X_t)_{t\le \tau_K},$ then  $h\left(V_{z_i}\right)=\prcond{\xi_{\lfloor\frac{K}{2}\rfloor}=z_i}{G_n}{}\le e^{-c'\frac{K}{2}}$, for some constant $c'$, where the last inequality follows from the assumption that all vertices on the tree have degree at least $3$ and the definition of the loop erasure (see also \cite[Lemma~3.13]{PerlasPaper}).

 Combining~\eqref{eq:badvertices} and~\eqref{eq:h(Vzi)} we deduce for $R>4$ that 
\begin{equation}\label{eq:sumBad}\prcond{ \sum_{i=1}^Lh\left(V\cap V_{z_i}\right)\mathds{1}\left(z_i\text{ is\,bad}\right)\le R n^{-\varepsilon c_2}}{\mathcal{B}_K\left(x\right)=T_0}{} \ge 1-\frac{1}{n^2}.\end{equation}
We next bound the sum over the good vertices and we get
\begin{align*}&\econd{h\left(V\cap V_{z_i}\right)\mathds{1}\left(z_i \text{ is\,good}\right)}{\mathcal{F}_{i-1}}
\\&=\sum_{z\in V_{z_i}}h\left(z\right)\prcond{\prcond{T_2\left(\frac{t}{2}\right)< \delta t}{G_n}{z}>n^{-\beta\alpha}}{\mathcal{F}_{i-1}}{}\mathds{1}\left(z_i \text{ is\,good}\right)
\\&\leq \sum_{z\in V_{z_i}}h\left(z\right)\frac{\econd{\prcond{T_2\left(\frac{t}{2}\right)< \delta t}{G_n}{z}}{\mathcal{F}_{i-1}}}{n^{-\beta\alpha}}\mathds{1}\left(z_i \text{ is\,good}\right) 
\\&= \sum_{z\in V_{z_i}}h\left(z\right)\frac{\prcond{T_2\left(\frac{t}{2}\right)< \delta t}{\mathcal{F}_{i-1}}{z}}{n^{-\beta\alpha}}\mathds{1}\left(z_i \text{ is\,good}\right).
\end{align*}
Taking $c=2\beta$ in Lemma~\ref{lemma:T2conditional} and taking $C$ (in the definition of $t$) sufficiently large we get for $n$ sufficiently large
\[\econd{h\left(V\cap V_{z_i}\right)\mathds{1}\left(z_i \text{ is\, good}\right)}{\mathcal{F}_{i-1}}\le  C_1 h\left(V_{z_i}\right) \mathds{1}\left(z_i \text{ is \,good} \right) n^{-\beta\alpha}.
\]
Writing $R_i=h\left(V\cap V_{z_i}\right)\mathds{1}\left(z_i \text{\,is\, good}\right)$, we consider the Doob martingale defined conditional on $\mathcal{B}_K\left(x\right)=T_0$ as $M_0=0$ and for $1\le k\le L$ (recall $L=|\partial \B_{K/2}(x)|$ from Definition~\ref{def:CouplingFirst}) 
\[M_k=\sum_{i=1}^k\left(R_i-\econd{R_i}{\mathcal{F}_{i-1}}\right).
\]
Using that $\left|M_k-M_{k-1}\right|\le 2h\left(V_{z_i}\right)$, we can apply Azuma-Hoeffding's inequality  to this martingale to get
\begin{align*}&\prcond{\sum_{i=1}^L h\left(V\cap V_{z_i}\right)\mathds{1}\left(z_i \text{ is\, good}\right)\ge (C_1+1)n^{-\beta\alpha}}{\mathcal{B}_K\left(x\right)=T_0}{}
\\&\le\prcond{M_L\ge n^{-\beta\alpha}}{\mathcal{B}_K\left(x\right)=T_0}{} \le \exp\left(-\frac{n^{-2\beta\alpha}}{2\sum_{i=1}^L\left(2h\left(V_{z_i}\right)\right)^2}\right)\le \exp\left(-\frac{n^{c_2\varepsilon-2\beta\alpha}}{8}\right).\end{align*}
 The last inequality above follows from the fact that using \eqref{eq:h(Vzi)} $$\sum_{i=1}^L \left(h\left(V_{z_i}\right)\right)^2\le n^{-c_2\varepsilon}\sum_{i=1}^L h\left(V_{z_i}\right)=n^{-c_2\varepsilon}. $$ 
 If $\alpha=o(1)$, then for all $n$ sufficiently large the exponential above can be made smaller than $1/n^2$. Otherwise, we take $\beta=c_2\varepsilon/4$ so that the exponential is smaller than $1/n^2$.
 
 Using the above together with \eqref{eq:ProbT2Bound}, \eqref{eq:BoundSumh} and  \eqref{eq:sumBad} gives that for $n$ sufficiently large
\begin{align*} \prcond{\prcond{T_2\left(t\right)\le \delta t}{G_n}{x}\le (C_1+2)n^{-\beta\alpha}+Rn^{-\varepsilon c_2}+n^{-c_1}}{\mathcal{B}_K\left(x\right)=T_0}{}\ge1- \frac{2}{n^2}.
\end{align*}
If $\alpha=o(1)$, then we can take any $\beta$ by changing $C$ in the definition of $t$. Otherwise, we take $\beta=c_2\varepsilon/4$.
Taking a union bound over all $x\in G_n$ completes the proof. 
\end{proof}

\begin{lemma}\label{lem:boundonrevlazymc}
        Let $P$ be the transition matrix of a finite reversible lazy Markov chain with invariant distribution~$\pi$, relaxation time $t_{\rm{rel}}$ and let $A$ be a non-empty subset of the state space.  Then for every distribution $\eta$ and all times $t$ we have 
        \begin{align*}
                \prstart{T_A>t}{\eta}\leq \left\|\frac{\eta}{\pi}\right\|_{2,\pi} \cdot \exp\left(-\frac{t \pi(A)}{t_{\rm{rel}}} \right),
        \end{align*}
        where $\left\|\frac{\eta}{\pi}\right\|^2_{2,\pi}  = \sum_x (\eta(x))^2/\pi(x)$.
\end{lemma}

\begin{proof}

Applying the Cauchy-Schwartz inequality we obtain
\begin{align*}
        \left(\prstart{T_A>t}{\eta}\right)^2 = \left(\sum_{x} \frac{\eta(x)}{\pi(x)} \prstart{T_A>t}{x} \pi(x)\right)^2 \leq \left\|\frac{\eta}{\pi}\right\|^2_{2,\pi} \cdot \sum_{x\in A^c} \left(\prstart{T_A>t}{x}\right)^2\pi(x). 
\end{align*}
Writing $P_{A^c}$ for the restriction of the transition matrix $P$ to the set $A^c$, i.e.\ for all $x,y\in A^c$ we have $P_{A^c}(x,y) = P(x,y)$, we have
\begin{align*}
         \sum_{x\in A^c} \left(\prstart{T_A>t}{x}\right)^2\pi(x) &=
 \sum_{x\in A^c} \sum_{y,z\in A^c}P_{A^c}^t(x,y) P_{A^c}^t(x,z) \pi(x) \\&=  \sum_{x\in A^c} \sum_{y,z\in A^c}P_{A^c}^t(x,y) P_{A^c}^t(z,x) \pi(z) \\ &= 
          \sum_{y,z\in A^c}P_{A^c}^{2t}(z,y) \pi(z) = \prstart{T_A>2t}{\pi} \leq  \exp\left(-\frac{2t \pi(A)}{t_{\rm{rel}}} \right),
\end{align*}
        where for the second equality we used reversibility and for the last inequality we used~\cite[Lemma~3.5]{CharacterizationCutoff}. This now concludes the proof.
\end{proof}

\begin{proof}[Proof of Lemma \ref{lemma:DirEVContainingCom1}] By Lemma~\ref{lem:dirichletdef} it is enough to show  that with high probability, for all large enough t, 
\[\prcond{T_{D^c}>t}{G_n}{\pi_D}\lesssim n \exp\left(-\tilde{c} t \alpha\right),
\]
for a positive constant $\tilde{c}$.
Let $\widetilde{X}$ be the induced chain in community $2$, i.e.\ it is the Markov chain~$X$ viewed at the times when it visits the second community. We write $\widetilde{T}_A$ for the first hitting time of the set $A$ by $\widetilde{X}$. Recalling that $T_2(t)$ is the total time spent in community $2$ by time $t$, we have that for any starting point $y\in D$
\begin{equation}\label{eq:BoundTimeToLeaveD}\prcond{T_{D^c}>t}{G_n}{y}\leq \prcond{T_2(t)\le \delta t}{G_n}{y}
+\prcond{\widetilde{T}_{D_{2}^{c}}>\delta t}{G_n}{y}.
\end{equation}
We start by bounding the second probability on the right hand side above. Let $\widetilde{t}_{\mathrm{rel}}$ and $\widetilde{\pi}$ be the relaxation time and invariant distribution respectively of the chain $\widetilde{X}$. Then for any distribution $\eta$ on $D_{2}$ using Lemma~\ref{lem:boundonrevlazymc} we get  
\[
\prcond{\widetilde{T}_{D_{2}^{c}}>\delta t}{G_n}{\eta} \leq\left\|\frac{\eta}{\widetilde{\pi}}\right\|_{2,\widetilde{\pi}} \cdot\exp\left({-\frac{\delta t \widetilde{\pi}\left(D_{2}^{c}\right)}{\widetilde{t}_{\mathrm{rel}}}}\right).
\]
Let $f$ be an eigenvector corresponding to the second largest eigenvalue $\widetilde{\lambda}_2$ of $\widetilde{X}$.
Then
\[\frac{1}{\widetilde{t}_{\mathrm{rel}}}=1- \widetilde{\lambda}_2=\frac{\widetilde{\mathcal{E}}(f,f)}{\|f \|_{2,\widetilde\pi}},\]
where $\widetilde{\mathcal{E}}$ is the Dirichlet form corresponding to $\widetilde{X}$.

It is easy to check that for $g=f \mathds{1}\left\{f \ge 0 \right\},$  \[ \widetilde{\mathcal{E}}(f,f) \ge \widetilde{\mathcal{E}}(g,g). \]
By replacing $f$ with $-f$ if necessary, we can assume that \[ \|f\|_{2,\widetilde\pi} \le 2 \|g\|_{2,\widetilde{\pi}}.\]
Putting the last two inequalities together, we get that
 \[ 1- \widetilde\lambda_2=\frac{\widetilde{\mathcal{E}}(f,f)}{\|f \|_{2,\widetilde\pi}} \gtrsim   \frac{\widetilde{\mathcal{E}}(g,g)}{\|g \|_{2,\widetilde\pi}}\gtrsim \frac{{\mathcal{E}}(g,g)}{\|g \|_{2,\pi}},\]
where for the last inequality we used that $\widetilde\pi\asymp \pi$.
As $g$ is non-negative (supported inside the second community), using the variational characterisation of the Dirichlet eigenvalue of the second community (see for instance~\cite{aldous-fill-2014}[Section 3.6.2]) and Lemma~\ref{lemma:DirEVCom1} we get that with high probability $$\alpha\lesssim \frac{1}{\widetilde{t}_{\mathrm{rel}}}.$$ This together with the fact that $\widetilde{\pi}(D_2^c)$ is at least a positive constant and that $\|\eta\|_2\le n$ gives that with high probability over $G_n$ for a positive constant $c_1$
\[\prcond{\widetilde{T}_{D_{2}^{c}}>\delta t}{G_n}{\eta}\leq n\cdot \exp(-c_1\delta \alpha t).
 \] 
 We proceed to bound the first term on the right hand side of \eqref{eq:BoundTimeToLeaveD}. By Corollary~\ref{cor:notkroot} we have that with high probability over the graph $G_n$ for all $x\in G_n$
 \[\prcond{T_2(C\log n)< 2\delta C\log n}{G_n}{x}  \le C_1n^{-c \alpha}.
 \] 
In case $1\gg\alpha \gtrsim 1/\log n$, we can take $c$ as large as we like by increasing $C$. In particular, we take $c$ so that $n^{\alpha c}/(2C_1)>1$. If $\alpha\asymp 1$, then for any constant $c$ we have $n^{\alpha c}/(2C_1)>1$ for large enough $n$. Since the bound above holds uniformly over all $x\in G_n$, using the Markov property we can perform independent experiments and obtain for any $k\in \N$
\begin{align*}
        &\prcond{T_2(Ck\log n)<\delta k C\log n}{G_n}{x} \\ &\leq \prcond{\sum_{i=1}^{k} \mathds{1}(T_2(iC\log n)-T_2((i-1)C\log n) <2\delta C\log n)>\frac{k}{2}}{G_n}{x}\\& \leq 
         \pr{{\rm{Bin}}(k,C_1n^{-c\alpha})>\frac{k}{2}}.
\end{align*}
 Using large deviations for a binomial distribution and that $n^{\alpha c}/(2C_1)>1$ we deduce for a positive constant $c_2$
 \begin{align*}
         \pr{{\rm{Bin}}(k,C_1n^{-c\alpha})>\frac{k}{2}}\lesssim \exp(-c_2\alpha k\log n)
 \end{align*}
 and this completes the proof.
\end{proof}

 \begin{proof}[Proof of Proposition \ref{prop:2comrelax}]
For all sets $D\subset V$, with $\pi(D)\le\frac{1}{2}$ there exists a set $D'$ with $D\subset D'$ and such that $D'$ contains one full community $i$ and less than half of the other community, i.e. $\pi(D'\cap V_{3-i})\le\frac{1}{2}\pi(V_{3-i})$.  Then Lemma~\ref{lemma:DirEVContainingCom1} gives that $\lambda(D')\geq c\alpha$ for a positive constant $c$. By monotonicity we have that $\lambda(D)\gtrsim \alpha, $ which completes the proof.  
\end{proof}

\section{Coupling} \label{sec5}

In this section we prove Theorems~\ref{thrm:Cutoff} and~\ref{thrm:mCutoff} for the lazy simple random walk. So in this section~$X$ is again taken to be a lazy simple random walk on $G_n$.

We denote by $d(x,y)$ the graph distance from $x$ to $y$ in $G_n$ and write $\mathcal{B}_{K}(x)$ for the ball of radius~$K$ and centre $x$ in this metric. As in Lemma \ref{lemma:TruncatedEdgeNotVisited} let
$$ t=\frac{\log n}{\nu \mathfrak{h}}-B \sqrt{\frac{\log n}{\alpha}} $$
for a constant $B$ to be determined and also let $A$ be as in Lemma \ref{lemma:TruncatedEdgeNotVisited}.
\begin{definition}\label{def:KRoot} We call a vertex $x$ a $K-$root of $G_{n}$ if $\mathcal{B}_{K}(x)$ is a possible realisation of the first $K$ levels of the \random tree $T$ (corresponding to $G_{n}$). If $x$ is a $K-$root and $i \leq K,$ we denote by $\partial \mathcal{B}_{i}(x)$ the collection of vertices at distance $i$ from $x$.
\end{definition}
We next define an exploration process of $G_{n}$ and a coupling between the walk $X$ on $G_{n}$ and a walk~$\widetilde{X}$ on the \random tree $T$ corresponding to $G_{n}$. The coupling below for the 2 community case is similar to the one used in Section~\ref{sec4}, however here we are not exploring only the parts of the graph that the walk visits and we truncate more edges than in Section~\ref{sec4}. Both the coupling and the exploration process are very similar to the ones used in~\cite[Definition~5.3]{PerlasPaper}.

\begin{definition}\label{def:Coupling} Let $K=\left\lceil C_{2} \log \log n\right\rceil$ for a constant $C_{2}$ to be determined and suppose we work conditional on the event that $x_{0}$ is a $K-$root and $\mathcal{B}_{K}\left(x_{0}\right)=T_{0},$ where $T_{0}$ is a realisation of the first $K$ levels of a \random tree. By this we mean that we know which vertices in $V_n$ are part of $B_K(x_0)$ but we have not yet revealed the full degree vector of vertices at distance $K$ from $x_0,$ meaning that in the case of the two community model we do not know which of the half edges coming out of $\partial\mathcal{B}_K(x_0)$ are internal and which are outgoing, and in the case of multiple communities we do not know which communities outgoing edges lead to. Let $L=\left|\partial \mathcal{B}_{K / 2}\left(x_{0}\right)\right|$ and $\left\{z_{1}, \ldots, z_{L}\right\}$ be some ordering of the set $\partial \mathcal{B}_{K / 2}\left(x_{0}\right).$ For each $z \in \partial \mathcal{B}_{K / 2}\left(x_{0}\right)$ we denote by $V_{z}$ the set of offspring of $z$ which are also in $\partial \mathcal{B}_{K}\left(x_{0}\right) .$ Fix $z \in \partial \mathcal{B}_{K / 2}\left(x_{0}\right) .$ We now describe the exploration process of $G_{n}$ corresponding to the set $V_{z}$ by constructing a coupling of a subset of $G_{n}$ with a subset of a  \random tree $T$, conditioned on its first  levels $K$ being exactly $T_{0}$. We first reveal all edges of $T$ with one endpoint in $\partial T_{0}$, including the ones which are not descendants of $z$. For the edges originating in $V_{z}$ we couple them with the edges of $G_{n}$ by using the optimal coupling between the two distributions at every step.

Choosing half-edge types differs in the cases of the two models we consider. For the two communities model we write $p(s)$ for the number of outgoing half-edges in community $i$ which were revealed so far  and $N_i(s)$ for number of half-edges in community $i$ whose types were revealed by time $s$. At step $s$ of the exploration process for every half-edge corresponding to the vertex in $V_z$ we assign it an outgoing type  with probability $\frac{p-p(s)}{N_i-N_i(s)}$ and otherwise we assign it an internal type. For the~$m$ communities model, if the current vertex is in community $i$ we choose the community to which an outgoing edge from it leads to be $j\ne i$ with probability $\frac{E_{i,j}-E_{i,j}(s)}{\sum_{k\ne i}\left(E_{i,k}-E_{i,k}(s)\right)}$, where $E_{i,j}(s)$ is the number of outgoing edges between $i$ and $j$ which have already been revealed.

In both cases, for each half-edge we choose a half-edge to connect it to uniformly at random from all of the half-edges  in the appropriate community, which were not chosen yet (and in $m$ community case we take care that the type of the edge is appropriate). 

If at some point one of these couplings fails, which meant that either the corresponding types of the edges or just the corresponding vertices in $T$ and $G_n$ differ, then we truncate the edge where this happened and stop the exploration for this edge in $G_{n}$ but we continue it in $T$. We also truncate an edge and stop the exploration in $G_{n}$ if we create a cycle. Once all edges joining levels $K$ and $K+1$ of $T$ have been revealed, we examine which of those satisfy the truncation criterion $\operatorname{Tr}(e, A)$ from Definition~\ref{def:Truncation} (which is defined w.r.t. $T,$ not $G_{n}$). We then stop the exploration at these edges for the graph $G_{n},$ but we continue the exploration of their offspring for the  tree $T$. Suppose we have explored all $k$ level edges of the  tree $T$ and also the corresponding ones in $G_{n}$ that were coming from $V_z$ and have not been truncated. Then for the edges of level $k+1$ we explore all of them in $T$, (including the ones which are not descendants of $z$) and we use the optimal coupling to reveal the corresponding ones that come from non-truncated descendants of $z$ in $G_{n}$ in the same way as when revealing level $K+1$. We truncate an edge and stop the exploration process at this edge if the optimal coupling between the two distributions fails for the type or at the endpoint of the edge or if we created a cycle. We again truncate edges in $G_n$ corresponding to the ones in $T$ which satisfy the truncation criterion. We continue the exploration process for $t$ levels.
\\We now describe a coupling of the walk $X$ on $G_{n}$ starting from $x \in V_{z}$ with a walk $\widetilde{X}$ on $T$ starting from $x$ as follows: we move $X$ and $\widetilde{X}$ together for time $t$ and we say the coupling is successful as long as none of the following happen:
\\(i) $\widetilde{X}$ crosses a truncated edge;
\\(ii) $\widetilde{X}$ visits a vertex $w \in \partial \mathcal{B}_{K / 2}\left(x_{0}\right)$ for some $w \neq z.$
\\We write $\mathcal{F}_{i}$ for the $\sigma$ -algebra generated by $T_{0}$ and the exploration processes starting from all the vertices of $V_{z_{1}}, \ldots, V_{z_{i}} .$ We call $z_{i} \in \partial \mathcal{B}_{K / 2}\left(x_{0}\right)$ good if none of its descendants in $\partial T_{0}$ (i.e. those vertices $y \in \partial T_{0}$ such that $\left.d\left(x_{0}, y\right)=d\left(x_{0}, z_{i}\right)+d\left(z_{i}, y\right)\right)$ has been explored during the exploration processes corresponding to the sets $V_{z_1}, \ldots, V_{z_{i-1}} .$ Otherwise, $z_{i}$ is called bad. Note that the event $\left\{z_{i}\mathrm{\, is\, bad} \right\}$ is $\mathcal{F}_{i-1}$ measurable. Finally, we denote by $\mathcal{D}_{i}$ the collection of vertices of $G_{n}$ explored in the exploration process of the set $V_{z_{i}}$.
\end{definition}
\begin{remark}\label{remark:StayInD}We note that if the coupling between $X$ and $\widetilde{X}$ starting from $x \in V_{z_{i}},$ for ${z_{i} \in \partial \mathcal{B}_{K / 2}\left(x_{0}\right)},$ succeeds for $t$ steps, then $\widetilde{X}_{s} \in \mathcal{D}_{i}$ for all $s \leq t$.
\end{remark}

\begin{lemma} \label{lemma:NumberOfBadVertices} Let $\alpha\gg 1/\log n$, in the setup of Definition \ref{def:Coupling}, we have that $\left|\mathcal{D}_{i}\right| \leq n \exp \left(-A \sqrt{\frac{\log n}{\alpha}}/3\right)$ for all $i \in L $ deterministically. Moreover, there exists a positive constant $C$ (independent of $T_{0}$ ) so that the number $\mathrm{Bad}$ of bad vertices $z$ satisfies
\[\prcond{\mathrm{Bad} \geq C \sqrt{\log n}}{\mathcal{B}_{K}\left(x_{0}\right)=T_{0}}{} \leq \frac{1}{n^{2}}.\]
\end{lemma}

\begin{proof} 
This proof follows along the lines of the proof of~\cite[Lemma~5.5]{PerlasPaper}. We include the details only where they differ. 
Exactly as in the proof of~\cite[Lemma~5.5]{PerlasPaper} using the definition of the truncation criterion and that the exploration process continues for $t\asymp \log n$ steps, the total number of explored vertices is at most  
\[
n\exp\left(- \frac{A}{3}\cdot \sqrt{\frac{\log n}{\alpha}} \right),
\]
where $A$ is as in Lemma~\ref{lemma:TruncatedEdgeNotVisited}.  
At every step of the exploration process the probability of intersecting a vertex of $\partial T_{0}$ is upper bounded by
$$\frac{\Delta^ {K}}{ c_{1}\alpha n-\Delta n \exp \left(-A \sqrt{\frac{\log n}{\alpha}}/3\right)} \lesssim \frac{\Delta^{C'\log \log n}}{n}$$
where $c_{1}$ and $C'$ are positive constants. This is clear in the case of a two communities model, as there are order $n$ edges from vertices of each type and therefore the above bounds the probability to match the current half-edge with one of at most $\Delta^K$ from the boundary of $T_0.$ In the case of multiple communities, as there are at least order $\alpha n$ outgoing edges from vertices of each type (as the Cheeger constant would otherwise be smaller), and at least order $n$ internal ones, once the community and type of an edge we are currently matching to is known, the probability to match it with one of at most $\Delta^K$ half edges in the boundary of $T_0$ is again bounded by the above. We therefore obtain 
\begin{align*}
\prcond{\mathrm{Bad}>C\sqrt{\log n}}{\mathcal{B}_{K}\left(x_{0}\right)=T_{0}}{}&\le \binom{ n \exp \left(-A \sqrt{\frac{\log n}{\alpha}}/3\right)} {C\sqrt{\log n}} \left(\frac{\Delta^{C'\log \log n}}{n}\right)^{C\sqrt{\log n}}
\\&\le \left(\frac{1}{n}\right)^{AC/6}  
\end{align*}
for large enough $n$. Taking $C$ large enough completes the proof.\end{proof}

\begin{lemma}\label{lemma:CouplingSuccedes}  In the same setup as in Definition \ref{def:Coupling}, for all $\varepsilon>0,$ there exist $B$ (in the definition of $t$) and $A$ (in the definition of the truncation criterion, depending on $\varepsilon$ and $B$) sufficiently large so that for all $n$ large enough on the event $\left\{\mathcal{B}_{K}\left(x_{0}\right)=T_{0}\right\}$, for all $i$ and all descendants $x \in \partial \mathcal{B}_{K}\left(x_{0}\right)$ of $z_{i}$, the coupling of Definition \ref{def:Coupling} satisfies
$$\prcond{\text{the coupling of } X \text{ and } \tilde{X} \text{ succeeds}}{\mathcal{F}_{i-1}}{x} \geq \mathds{1}\left(z_{i} \text { is good}\right) \cdot(1-\varepsilon)$$

\end{lemma}
{\proof  
This proof is identical to the proof of~\cite[Lemma~5.6]{PerlasPaper} with the only difference being that the probability that the optimal couplings fail at any given step or that a cycle is created is upper bounded by 
\[
c\frac{ n \exp \left(-A \sqrt{\frac{\log n}{\alpha}}/3\right)}{n\alpha}
\]
for a positive constant $c$, using arguments similar to the proof of Lemma~\ref{lemma:NumberOfBadVertices}.
\qed}

We recall that the $\mathcal{L}^2$ distance of a probability measure $\mu$ from $\pi$ is given by
\[
\|\mu - \pi\|_2^2 = \sum_x\left(\frac{\mu(x)}{\pi(x)}-1 \right)^2 \pi(x).
\]

\begin{prop}\label{prop:t+s+r} In the same setup as in Definition \ref{def:Coupling}, for all $\varepsilon>0,$ there exist $B$ (in the definition of $t$), $A$ (in the definition of the truncation criterion) depending on $\varepsilon$ and $B$ and a positive constant $\Gamma$ sufficiently large such that for all $n$ sufficiently large the following holds. For all $M>0$, on the event $\left\{\mathcal{B}_{K}\left(x_{0}\right)=T_{0}\right\}$, for all $i$ and all $x \in \partial \mathcal{B}_{K}\left(x_{0}\right)$  which are descendants of $z_{i} \in \partial \mathcal{B}_{K / 2}\left(x_{0}\right),$ we have for $r>0$ and $$s\left(G_n\right)=\int_{\frac{4}{\frac{\Delta}{\left(1-\varepsilon\right)^2} \exp \left(2\Gamma \sqrt{\frac{\log n}{\alpha}}\right)}}^{\frac{4}{M^2}}\frac{d\delta}{\delta \Lambda_{G_n}\left(\delta\right) },$$ with $\Lambda_{G_n}$ being the spectral profile for random walk on the graph $G_n,$ that
\[
\prcond{d_{x}(t+s(G_n)+r)<e^{-\frac{r}{t_{\mathrm{rel}}\left(G_{n}\right)}} \cdot M+\varepsilon}{\mathcal{F}_{i-1}}{} \geq \left(1-2 \varepsilon\right)\mathds{1}\left(z_{i}\,\text{ is good}\right) 
\]
where $d_{x}(s)=\left\|\prcond{X_{s} \in \cdot}{G_{n}}{x}-\pi\right\|_{\mathrm{TV }}$ for every $s \in \mathbb{N}.$
\end{prop}
{\proof 
This proof follows in exactly the same way as the proof of \cite[Proposition~5.7]{PerlasPaper} with the only difference being that here we bound the $\mathcal{L}^2$ distance using not only the Poincar\'e inequality, but also the spectral profile technique. Using the same notation as in~\cite[Proposition~5.7]{PerlasPaper} we only point out the places where the two proofs differ. 

We set $\ell=\frac{\log n}{\mathfrak{h}}-2 \nu B \sqrt{\frac{\log n}{\alpha}}$ and recall that $t=\frac{\log n }{\nu \mathfrak{h}}-B \sqrt{\frac{\log n}{\alpha}}$. 
In~\cite[Proposition~5.7]{PerlasPaper} a class of graphs $\mathcal{G}$ is defined so that 
\[
\prcond{G_n\in \mathcal{G}}{\mathcal{F}_{i-1}}{} \geq 1-\varepsilon,
\]
which implies for all $u\in \mathbb{N}$
\begin{align}\label{eq:tvwithouts}
\prcond{\left\|\prcond{X_{u} \in \cdot}{G_{n}}{x}-\pi\right\|_{\mathrm{TV}}=\mathds{1}\left(G_{n} \in \mathcal{G}\right)\left\|\prcond{X_{u} \in \cdot}{G_{n}}{x}-\pi\right\|_{\mathrm{TV}}}{\mathcal{F}_{i-1}}{} \geq 1-\varepsilon.
\end{align}
(Note that in the definition of the sets $A_1$ and $\widehat{B}$ here we need to take $\sqrt{\log n/\alpha}$ instead of $\sqrt{\log n}$.) 
For the event $S$ as defined in~\cite[Proposition~5.7]{PerlasPaper} we get for any $u\in \mathbb{N}$
\begin{align}\label{eq:tvwiththeevents}
\left\|\prcond{X_{u} \in \cdot}{G_{n}=G}{x}-\pi\right\|_{\mathrm{TV}} \leq\left\|\prcond{X_{u} \in \cdot}{S, G_{n}=G}{x}-\pi\right\|_{\mathrm{TV}}+\varepsilon.
\end{align}
In exactly the same way as in~\cite{PerlasPaper} we get for $G\in\mathcal{G}$
\begin{align}\label{eq:boundonl2timet}
\left\|\prcond{X_{t} \in \cdot}{G_{n}=G, S}{x}-\pi\right\|_{2} \leq \frac{\sqrt{\Delta}}{1-\varepsilon} \exp  \left(\Gamma \sqrt{\frac{\log n}{\alpha}}\right) 
\end{align}
By the \Poincare inequality and the fact that conditional on $X_{t}$, the event $S$ is independent of $\left(X_{u}\right)_{u \geq t}$ we have 
\begin{align*} \left\|\prcond{X_{t+s(G)+r}\in \cdot}{S,G_n=G}{x}-\pi \right\|_{\mathrm{TV}}&\le \left\|\prcond{X_{t+s(G)+r}\in \cdot}{S,G_n=G}{x}-\pi\right\|_{2}\\
     &\le e^{-\frac{r}{\mathrm{trel}\left(G\right)}}\left\|\prcond{X_{t+s(G)}\in \cdot}{S,G_n=G}{x}-\pi\right\|_2,
     \end{align*}
where $s(G)$ is defined as in the statement of the proposition.
Using~\eqref{eq:boundonl2timet} and~\cite[Theorem 1.1]{SpectralProfilePaper} (see also~\cite[Proposition 5.2]{DynamicalPercolations} for the form we use here) we get 
\begin{align*}
        \left\|\prcond{X_{t+s(G)}\in \cdot}{S,G_n=G}{x}-\pi\right\|_2 \leq M.
\end{align*} 
This together with~\eqref{eq:tvwiththeevents} and~\eqref{eq:tvwithouts} finishes the proof.
\qed}

\begin{lemma} \label{lemma:HittingKRoot}There exists a positive constant $\beta,$ so that starting from any vertex the random walk will hit a $K-$root by time $\beta K$ with probability $1-o(1)$ as $n \rightarrow \infty.$
\end{lemma}

{\proof The proof of this follows directly from Lemma 2.3 from \cite{NBRWvsSRW} (in the same way as discussed in the proof of Corollary~\ref{cor:notkroot}) and could also be derived analogously to the proof of Lemma~5.9 in~\cite{PerlasPaper}. \qed}

\proof[Proof of Theorem~\ref{thrm:Cutoff} for $\alpha\gtrsim \frac{1}{\log n}$ and Theorem~\ref{thrm:mCutoff}](lazy walk case)
We first prove cutoff for both models when $\alpha\gg 1/\log n$.
 Recall that $t=\frac{\log n }{\nu \mathfrak{h}}-B \sqrt{\frac{\log n}{\alpha}},$ where $B$ is a positive constant to be chosen later. We first prove the upper bound on the mixing time. From Lemma~\ref{lemma:SpectralProfileSmallSets} we know that for $M$ larger than a  certain constant we have that with high probability $\Lambda\left(\delta\right)\ge \tilde{c}$ for $\delta\le \frac{4}{M^2}$ and so with high probability
$$s\left(G_n\right)=\int_{\frac{4}{\frac{\Delta}{\left(1-\varepsilon\right)^2} \exp \left(2\Gamma \sqrt{\frac{\log n}{\alpha}}\right)}}^{\frac{4}{M^2}}\frac{d\delta}{\delta \Lambda_{G_n}\left(\delta\right)}\le  \int_{\frac{4}{\frac{\Delta}{\left(1-\varepsilon\right)^2} \exp \left(2\Gamma \sqrt{\frac{\log n}{\alpha}}\right)}}^{\frac{4}{M^2}}\frac{d\delta}{ \tilde{c}\delta}\lesssim \Gamma\sqrt{\frac{\log n}{\alpha}}.$$
Let $r=t_{\text {rel }}\left(G_{n}\right)\log\left(\frac{M}{\varepsilon}\right),$ where $\Gamma$ is as in Proposition \ref{prop:t+s+r}  so that
$$
e^{-\frac{r}{t_{\mathrm{rel}}\left(G_{n}\right)}} \cdot M=\varepsilon.
$$
We now claim that it suffices to prove that with high probability
\begin{equation}\label{eq:TMix}
t_{\operatorname{mix}}\left(G_{n}, 5 \varepsilon\right) \leq t+s(G_n)+r+(\beta+c) K,
\end{equation}
where $\beta$ is as in Lemma \ref{lemma:HittingKRoot} and $c$ is a positive constant to be determined later. Indeed, once this is established, the proof then follows from the bound on $s(G_n)$ above together with the fact that with high probability $t_{\text{rel}}\left(G_{n}\right)\asymp \frac{1}{\alpha}$ by Proposition \ref{prop:TRel}.

It remains to prove~\eqref{eq:TMix}. This now follows in exactly the same way as the proof of (5.9) in~\cite{PerlasPaper} with the only difference being that the set $V$ we need to consider here is 
$$V=\left\{x \in \partial \mathcal{B}_{K}\left(x_{0}\right): d_{x}(t+s(G_n)+r) \geq 2 \varepsilon\right\}$$
for $x_0$ a $K$-root and use Proposition~\ref{prop:t+s+r} in place of Proposition~5.7 in~\cite{PerlasPaper}. This completes the proof of the upper bound on the mixing time. The lower bound follows identically to~\cite{PerlasPaper}.
This completes the proof in the case of $\alpha\gg \frac{1}{\log n}.$
 
For $\alpha\lesssim \frac{1}{\log n}$ in the case of the $m$-communities model and for $\alpha \asymp \frac{1}{\log n}$ in the case of the $2$-communities model, we will show using the spectral profile technique that with high probability $t_{\mathrm{mix}}\lesssim \frac{1}{\alpha}$ and then Proposition \ref{prop:TRel} completes the proof in this case, as by \cite[Theorem 12.5] {MixingBook} $t_{\text{mix}}(\varepsilon)\ge (t_{\text{rel}}-1)\log\left(\frac{1}{2\varepsilon}\right)$ and as there can be no cutoff when  $t_{\text{mix}}\asymp t_{\text{rel}}$  \cite[Proposition 18.4] {MixingBook}. We will condition on the graph $G_n$ and similarly as before define \begin{align}\label{eq:defs(G)}s\left(G_n\right)=\int_{4\pi_*}^{\frac{4}{M}}\frac{2d\delta}{\delta \Lambda_{G_n}\left(\delta\right)},\end{align} where $M>0$ will be determined later and $\pi_*=\min_{x\in G_n}\pi(x).$  Notice that for any starting state $x$ using first the \Poincare inequality and the spectral profile bounds \cite[Theorem 1.1]{SpectralProfilePaper}  we have 
\begin{align*}
&\|\prcond{X_{s(G_n)+ \log(\frac{M}{\varepsilon})t_{\text{rel}}(G_n)}\in \cdot}{G_n}{x}-\pi(\cdot)\|_{\mathrm{TV}}\\&\le \|\prcond{X_{s(G_n)+\log\left(\frac{M}{\varepsilon}\right)t_{\text{rel}}(G_n)}\in \cdot}{G_n}{x}-\pi(\cdot)\|_2 \\&\le  \frac{\varepsilon}{M}\|\prcond{X_{s(G_n)}\in \cdot}{G_n}{x}-\pi(\cdot)\|_{2} \le \varepsilon.
 \end{align*}  
 By Proposition \ref{prop:TRel} with high probability $t_{\text{rel}}(G_n)\lesssim \frac{1}{\alpha}$. Also using Lemma~\ref{lemma:SpectralProfileSmallSets} there exists a constant~$\hat{c}$ so that with high probability  $\Lambda_{G_n}(\delta)\gtrsim 1$ for $\delta\le \hat{c},$ and hence $s(G_n)\lesssim \log n$ holds for $M\ge \frac{4}{\hat{c}}$. Choosing any constant $M\ge \frac{4}{\hat{c}}$ gives that with high probability $t_{\text{mix}}(\varepsilon)\lesssim \frac{1}{\alpha}+\log n$  which completes the proof of the upper bound in both cases.
 \qed

Before giving the proof of Theorem~\ref{thrm:Cutoff} in the case when 
$\alpha\ll \frac{1}{\log n}$, we give a short argument that shows that in this regime there is no cutoff with high probability. Indeed, as in the proof of Proposition~\ref{prop:TRel} for $\alpha\gtrsim 1/\log n$, we see that the relaxation time is lower bounded by $1/\alpha$ with high probability, as this follows directly from Cheeger's inequality. Moreover, in exactly the same way as in the proof above in the case when $\alpha\asymp 1/\log n$ we get that with high probability
\[
t_{\rm mix}(\varepsilon)\lesssim t_{\rm rel}+{\log n}.
\]
Thus using also that $t_{\rm mix}(\varepsilon) \gtrsim t_{\rm rel},$
the assumption that $\alpha \ll 1/\log n$ and that $t_{\rm rel} \gtrsim 1/\alpha$ ,we conclude that $t_{\rm mix}(\varepsilon)\asymp t_{\rm rel}$. This implies that there is no cutoff with high probability, but since it does not determine the order of $t_{\rm mix}(\varepsilon)$  we do this in the proof below.

\proof[Proof of Theorem~\ref{thrm:Cutoff}] ($\alpha\ll \frac{1}{\log n}$ and lazy walk case)

 We now prove an upper bound on the mixing time for the two community model and $\alpha \ll 1/\log n$.

 We will prove that there exist constants $\delta, C>0$ such that with high probability $G_n$ satisfies that for any subset $A$ with $\pi(A)\geq 1/2$ and any $x_0\in G_n$ we have for the lazy simple random walk~$X$ started from $x_0$ that \begin{align}\label{eq:hitA}\prcond{T_A\le\frac{C}{\alpha}}{G_n}{x_0}\ge \delta,\end{align} where $T_A= \min\left\{t:X_t\in A\right\}$ is the hitting time of set $A$. By performing repeated experiments of length $C/\alpha$ this will imply that with high probability for all sets $A\subset G_n$ with $\pi(A)\geq 1/2$ and~$x\in G_n$ \[\escond{T_A}{G_n}{x}\le \frac{C}{\alpha \delta},\] which gives by~\cite[Theorem 1.1]{MixHitPaper} and~\cite[Theorem~1.3]{IneqHit} that $t_{\text{mix}}(1/4)\lesssim 1/\alpha.$ This also implies an upper bound for $t_{\text{rel}}$ and completes the proof of Proposition~\ref{prop:TRel} and also gives us the desired bounds on $t_{\text{mix}}(\varepsilon)$ for any~$\varepsilon$. 
 
We are left with proving~\eqref{eq:hitA}. For $x\in G_n$ we write $\theta(x)\in \{1,2\}$ for the community $x$ belongs to. We also write $\tau_i$ for $i\in \{1,2\}$ for the first hitting time of community $i$ by $X$. We now define the random set $D=D(G_n)$ which depends on the sampling of the random graph $G$ by 
\[
D=\left\{x\in G_n:\prcond{\tau_{3-\theta(x)}<u}{G_n}{x}\le \sqrt{u\alpha}\right\},
\] 
where $u=C' \log n,$ for a constant $C'$ to be determined. Notice that for $\theta\in \{1,2\},$ \[\E{\prcond{\tau_{3-\theta}< u}{G_n}{x}}\lesssim u\alpha\ll 1\] for any $x\in V_\theta$, as when we generate the graph and the walk together, the probability of revealing an outgoing edge at any step of the simple random walk is $\alpha.$  Therefore we get
\begin{align*}u\alpha &\gtrsim \E{\sum_{x\in G_n}\pi(x)\prcond{\tau_{3-\theta(X_0)}< u}{G_n}{x}}\\&\ge \E{\sum_{x\in D^c}\pi(x)\prcond{\tau_{3-\theta(x)}< u}{G_n}{x}} \ge \sqrt{u\alpha}\E{\pi(D^c)}.\end{align*} This gives that by Markov's inequality \[\pr{\pi(D^c)>(u\alpha)^{1/4}}\lesssim (u\alpha)^{1/4}.\]
Recall the definition of $s(G_n)$ from~\eqref{eq:defs(G)} and that with high probability $s(G_n)\lesssim \log n$. We now set $s=C\log n$ where $C$ is a sufficiently large positive constant. Then with high probability we have that 
\[
\|\prcond{X_s\in \cdot}{G_n}{x_0} - \pi\|_\infty \leq M.
\]
This means that for large enough~$n$ 
\begin{align}\label{eq:boundonDc}
\prcond{X_s\in D^c}{G_n}{x_0}\le (M+1)\pi(D^c)\lesssim (u\alpha)^{1/4}
\end{align}
with high probability over the graph $G_n.$ 
 
 For $\theta\in \{1,2\}$ we now let $G^\theta_n$ be the graph obtained from $G_n$ by a uniform rewiring of the outgoing edges in community $\theta$ (this is the same as the definition of the graph $\widehat{G}_n$ in the proof of Lemma~\ref{lemma:DirEVCom1}). We write $\pi_\theta$ for the corresponding invariant distribution. Then with high probability (over the randomness of $G_n^\theta$) for every $\delta>0$ there exists $n$ sufficiently large so that for $x\in D$ and $C'$ (in the definition of $u$) sufficiently large we have
  \begin{align}\label{eq:boundontvwithrewired}
 \left\|\prcond{X_u\in \cdot}{G_n}{x}-\pi_\theta(\cdot)\right\|_{\rm{TV}} \leq \left\|\prcond{X_u\in \cdot}{G^\theta_n}{x}-\pi_\theta(\cdot)\right\|_{\rm{TV}}+\prcond{\tau_{3-\theta}<u}{G_n}{x} \leq \delta,
 \end{align}
 where we used that a lazy simple random walk on the configuration model exhibits cutoff at time of order $\log n$ with high probability, see for instance~\cite{RWonRG}.
 
 Since $\pi(A) = \pi_1(A)\pi(V_1)+ \pi_2(A)\pi(V_2)$ and $\pi(A)\geq 1/2$, we can assume without loss of generality that $\pi_1(A)\geq 1/2$. Let $C$ be a positive constant to be determined later. Writing $\nu_{x_0}(\cdot) = \prcond{X_s\in \cdot}{G_n}{x_0}$ we now get 
 \begin{align*}
        \prcond{T_A>\frac{C}{\alpha}+s}{G_n}{x_0} \leq \nu_{x_0}(D^c) + \sum_{x\in D\cap V_1} \nu_{x_0}(x) \prcond{X_{\frac{C}{\alpha}}\in A^c}{G_n}{x} \\ + \sum_{x\in D\cap V_2} \nu_{x_0}(x) \prcond{T_A>\frac{C}{\alpha}}{G_n}{x}.
 \end{align*}
 From~\eqref{eq:boundontvwithrewired} we see that for $x\in D\cap V_1$
 \[
 \prcond{X_{\frac{C}{\alpha}}\in A^c}{G_n}{x}\leq \left\|\prcond{X_{\frac{C}{\alpha}}\in \cdot}{G_n}{x} - \pi_{1}\right\|_{\rm TV} + \pi_1(A^c) \leq \delta + \frac{1}{2},
 \]
 and hence plugging this above we deduce
 \begin{align*}
        \prcond{T_A>\frac{C}{\alpha}+s}{G_n}{x_0}& \leq \nu_{x_0}(D^c) + \left(\frac{1}{2}+\delta \right)\nu_{x_0}(D\cap V_1) \\&+ \sum_{x\in D\cap V_2} \nu_{x_0}(x) \prcond{T_A>\frac{C}{\alpha}}{G_n}{x}.
 \end{align*}

Now for $x\in D\cap V_2$ using the Markov property we have that with high probability 
\begin{align*}&\prcond{T_A>\frac{C}{\alpha}}{G_n}{x}\le \sum_{y\in G_n}\prcond{X_u=y}{G_n}{x}\prcond{T_A>\frac{C}{\alpha}-u}{G_n}{y}\\&\le 2\left\|\prcond{X_u\in \cdot}{G_n}{x}-\pi_2(\cdot)\right\|_{\text{TV}}+ \sum_{y\in G_n}\pi_2(y)\prcond{T_A>\frac{C}{\alpha}-u}{G_n}{y}\\& \le 2\delta + \prcond{T_A>\frac{C}{2\alpha}}{G_n}{ \pi_2},\end{align*} 
where for the last inequality we used~\eqref{eq:boundontvwithrewired} and that for all large enough $n$ we have that $C/2\alpha>u$ using the assumption on $\alpha$. Therefore, it is enough to show that \begin{align}\label{eq:boundonTa}\prcond{T_A>\frac{C}{2\alpha}}{G_n}{ \pi_2}\le 1-4\delta,\end{align}
as plugging this above would imply that with high probability for all large enough $n$
 \begin{align*}
        \prcond{T_A>\frac{C}{\alpha}+s}{G_n}{x_0}\le \nu_{x_0}(D^c)+\left(\frac{1}{2}+\delta \right)\nu_{x_0}(D\cap V_1)+\left(1-2\delta \right)\nu_{x_0}(D\cap V_2)\le 1-\delta,
 \end{align*}
since we have from \eqref{eq:boundonDc} that $\nu_{x_0}(D^c)\le \delta$ for large enough $n$.

So now we focus on proving~\eqref{eq:boundonTa}.
Recall that $s=C\log n$. For $(x,y)\in V_1\times V_2$ and time $t$ we define the event 
\begin{align*}\mathcal{A}(t,x,y)=&\left\{\exists i\in\{t,\ldots, t+s\}: \,\theta(X_i)=2,\right.\\&\left.(X_{i-1},X_{i})\neq (x,y),\forall j\in\{t,\ldots, i-1\} \,\theta(X_j)=1\right\},
\end{align*} which is the event that the walk which is in community $1$ at time $t$ exits this community by time~$t+s$ through an edge different to $(x,y)$. We now consider a random set $U_1=U_1(G_n)$ of edges in $V_1\times V_2$ defined by 
\[
U_1=\left\{(x,y)\in V_1\times V_2: \prcond{\mathcal{A}(0,x,y)}{G_n}{x}\le \sqrt{\alpha s} \right\}.
\] 
 We also define a random set $U_2=U_2(G_n)$ of edges in $V_1\times V_2$ as
  \[
  U_2=\left\{(x,y)\in V_1\times V_2: \widehat{\mathcal{B}}_K(x) \text{ is a tree} \right\},
  \] 
  where $\widehat{\mathcal{B}}_K(x)$ is the connected component of $\mathcal{B}_K(x)\cap V_1$ containing $x$ with $K=C_2\log\log n$ for a large constant $C_2$ to be determined.

\begin{lemma}\label{lem:lemma1}
        Let $U=U_1\cap U_2.$ For an edge $(x,y)\in U$ we will show that \[
        \prcond{X_{s}\in V_2}{G_n}{x}\le 1-\frac{c}{2}
        \]
        where $c$ is the constant from Lemma~\ref{lemma:ReturnProb}. 
\end{lemma}

\begin{lemma}\label{lem:lemma2}
        Let $\tau$ be the hitting time of community $1$. There is a small constant $\beta<1$ such that with high probability the random graph $G_n$ is such that 
        \[
        \prcond{(X_{\tau},X_{\tau-1})\in U^c}{G_n}{\pi_2}\le \beta.
        \] 
\end{lemma}

We defer the proof of the two lemmas above until the end of the proof of the theorem.

We now set $\eta(z)=\prcond{X_{\tau+s}={z}}{G_n}{\pi_2}$ and see that from Lemmas~\ref{lem:lemma1} and~\ref{lem:lemma2} and~\eqref{eq:boundonDc} we have for $n$ sufficiently large 
\begin{align}\label{eq:boundonv1d}
\eta(V_1\cap D)&\ge \sum_{(w_1,w_2)\in U}\prcond{(X_{\tau-1},X_{\tau})=(w_2,w_1)}{G_n}{\pi_2}\prcond{X_{s}\in V_1\cap D}{G_n}{w_1}
\ge \frac{c}{4}(1-\beta).
\end{align} 
The proof of Lemma~\ref{lemma:DirEVCom1} shows that for positive constants $c_1$ and $c_2$ with high probability by choosing~$C$ sufficiently large
\begin{align*}
\prcond{\tau_{V_{3-\theta}}>\frac{C}{4\alpha}}{G_n}{\pi_\theta}\le c_1e^{-c_2\alpha C/(4\alpha)}\le {\delta}.
\end{align*} 
Using this we now get for $n$ sufficiently large
\begin{align*}
&\prcond{T_A>\frac{C}{2\alpha}}{G_n}{\pi_2} \\&\le\prcond{\tau>\frac{C}{4\alpha}}{G_n}{\pi_2}+ \sum_{z\in G_n}\prcond{X_{\tau+s}=z, \frac{C}{4\alpha}>\tau}{G_n}{\pi_2}\prcond{X_u\in A^c}{G_n}{z}
\\& \le {\delta}+\sum_{z\in V_2\cup D^c}\eta(z)+\sum_{z\in V_1\cap D}\eta(z)\left(\left\|\prcond{X_u\in \cdot}{G_n}{z}-\pi_1(\cdot)\right\|_{\text{TV}}+\pi_1(A^c)\right)
 \\&\le {\delta}+\sum_{z\in G_n}\eta(z)-\sum_{z\in V_1\cap D}\eta(z)\left(\frac{1}{2}-\delta\right)\le {\delta} +1-\frac{c}{4}(1-\beta)\left(\frac{1}{2}-\delta\right), 
 \end{align*}
 where the third inequality holds using~\eqref{eq:boundontvwithrewired} and the assumption that $\pi_1(A)\geq 1/2$ and the last one uses~\eqref{eq:boundonv1d}.  This gives \eqref{eq:boundonTa} and completes the proof as we can take $\delta$ as small as we wish. 
 \qed

\begin{proof}[Proof of Lemma~\ref{lem:lemma1}]
Recall $\tau_2$ is the hitting time of community $2$.  Then we have
\begin{align*}
\prcond{X_{s}\in V_2}{G_n}{x}&\le \prcond{\tau_2\leq s, (X_{\tau_2-1},X_{\tau_2})=(x,y)}{G_n}{x} \\&+    \prcond{\tau_2\leq s, (X_{\tau_2-1},X_{\tau_2})\neq(x,y)}{G_n}{x}.
\end{align*}
As $(x,y)\in U\subset U_1$ we have that the second probability is at most $\sqrt{\alpha s}$. It thus remains to control the first probability appearing above.

As $(x,y)\in U_2$ we know that the $K=C_2\log \log n$ neighbourhood of $x$ which can be reached without crossing to community $2$ is a tree. From Lemma~\ref{lemma:ReturnProb}, we have that the probability to cross $(x,y)$ before otherwise leaving $V_1$ and before reaching the boundary~$\partial\widehat{\mathcal{B}}_{K}(x)$ is $\le 1-c<1$. 

From Lemma~\ref{lemma:ReturnProb} we have that the probability that a simple random walk on a tree which has degrees at least $3$ backtracks for $C_2\log\log n$ levels is at most $e^{-c'C_2\log\log n}\le \left(\frac{1}{\log n}\right)^{c'C_2},$ where $c'>0.$ Writing $\tau^+_{\partial\widehat{\mathcal{B}}_{K}(x)}$ for the first return time to the set $\partial\widehat{\mathcal{B}}_{K}(x)$ and $\tau_x$ for the first hitting time of $x$, we therefore get that for~$z\in \partial\widehat{\mathcal{B}}_{K}(x)$   
\[
\prcond{\tau_{x}< \tau^+_{\partial\widehat{\mathcal{B}}_{K}(x)},\tau_{x}< \tau_{2}}{G_n}{z}\le \left(\frac{1}{\log n}\right)^{c'C_2},\] 
and therefore for any time $t$ 
\begin{align*}
&\prcond{\tau_{x}< (t+1)\wedge\tau_{2}}{G_n}{z}\le\prcond{\tau_{x}< \tau^+_{\partial\widehat{\mathcal{B}}_{K}(x)}\wedge\tau_{2}}{G_n}{z}\\ &\quad \quad \quad \quad \quad \quad \quad \quad \quad\quad \quad \quad \quad +\prcond{\tau_{x}> \tau^+_{\partial\widehat{\mathcal{B}}_{K}(x)}, \tau_{x}<( \tau^+_{\partial\widehat{\mathcal{B}}_{K}(x)}+t)\wedge\tau_{2}}{G_n}{z}\\&=\prcond{\tau_{x }< \tau^+_{\partial\widehat{\mathcal{B}}_{K}(x)}\wedge\tau_{2}}{G_n}{z}+\sum_{\widetilde{z}}\prcond{X_{\tau^+_{\partial{\widehat{\mathcal{B}}_{K}(x)}}}=\widetilde{z}}{G_n}{z}\prcond{ \tau_{x}<t\wedge\tau_{2}}{G_n}{\widetilde{z}}
\\&\le \left(t+1\right) \left(\frac{1}{\log n}\right)^{c'C_2},
\end{align*} 
where the last line follows by induction. Taking $t=s$ and $C_2$ in the definition of $K$ sufficiently large such that $c'C_2>1$ this gives that 
\begin{align*}
\prcond{\tau_2\leq s, (X_{\tau_2-1},X_{\tau_2})=(x,y)}{G_n}{x}&\le 1-c+ (s+1) \left(\frac{1}{\log n}\right)^{c'C_2}\le 1-c+c/4
\end{align*} 
with high probability over $G_n$ as $s=C \log n$.
\end{proof}

\begin{proof}[Proof of Lemma~\ref{lem:lemma2}]

We will bound \[\prcond{(X_{\tau},X_{\tau-1})\in U_1^c}{G_n}{\pi_2}\quad \text{and} \quad \prcond{(X_{\tau},X_{\tau-1})\in U_2^c}{G_n}{\pi_2}\] by a constant $\beta_1<1/2$ which then implies the statement of the Lemma by union bound. We let $\mathcal{H}_1, \mathcal{H}_2$ and $\mathcal{H}$ be the matchings of internal half edges of community $1$ and $2$ and the matching of the outgoing half edges, respectively.

Let $(x,y)\in V_1\times V_2$. First of all  by considering a walk starting from $x$ which is killed if it visits the undirected edge $(x,y)$ we have that
\[
\prcond{\mathcal{A}(0,x,y)}{x\sim y}{x} \lesssim s\alpha. 
\]
Indeed, this holds as the probability to generate a new outgoing edge at each step is $\lesssim \alpha$ and we can use the union bound up to time $s$. 
It is not hard to see by conditioning on the matchings on each community and since the event $\mathcal{A}(0,x,y)$ only depends on $\mathcal{H}_1$ and $(\tau,X_{\tau-1})$ only depends on~$\mathcal{H}_2$, that 
\begin{align*}
        \prstart{\mathcal{A}(\tau,X_{\tau},X_{\tau-1})}{\pi_2} = \sum_{x,y} \prcond{\mathcal{A}(0,x,y)}{x\sim y}{x} \prstart{X_{\tau}=x,X_{\tau-1}=y, x\sim y}{\pi_2}.
\end{align*}

This implies that 
\begin{align*}
&\alpha s\gtrsim \prstart{\mathcal{A}(\tau,X_{\tau},X_{\tau-1})}{\pi_2} = \E{\prcond{\mathcal{A}(\tau,X_{\tau},X_{\tau-1})}{G_n}{\pi_2}}\\&\ge\E{\sum_{(x,y)\in U_1^c}\prcond{\left(X_{\tau-1},X_{\tau}\right)=(y,x)}{G_n}{\pi_2}\prcond{\mathcal{A}(x,y,0)}{G_n}{x}}\\&\ge \sqrt{\alpha s}\ \E{\prcond{\left(X_{\tau},X_{\tau-1}\right)\in U_1^c}{G_n}{\pi_2}}.
\end{align*} 
Therefore by Markov's inequality \[\pr{\prcond{\left(X_{\tau},X_{\tau-1}\right)\in U_1^c}{G_n}{\pi_2}>(\alpha s)^{1/4}}\le (\alpha s)^{1/4}\to 0 \text{ as }n\to \infty,\] which gives that for any constant $\beta_1<1/2$ for large enough $n$ we have that with high probability $\prcond{(X_{\tau},X_{\tau-1})\in U_1^c}{G_n}{\pi_2}\le \beta_1.$

To control the probability of $\prcond{(X_{\tau},X_{\tau-1})\in U_2^c}{G_n}{\pi_2}$ we use that for any vertex in community $1$ if we start generating its $C_2\log\log n$ descendants in community $1$ the probability of creating a cycle with each new edge we reveal is at most $\Delta^{C_2\log\log n}/n$ (as $\Delta^{C_2\log\log n}$ is the maximal total number of vertices we reveal in community $1$). As we reveal at most $\Delta^{C_2\log\log n}$ edges by the union bound the probability that a cycle is created is $\lesssim \Delta^{2C_2\log\log n}/n.$  As we again have that $X_{\tau-1}$ does not depend on the matching in community $1$ and on the matching across the communities, this gives that \[\Delta^{2C_2\log\log n}/n\ge \E{\prcond{(X_{\tau},X_{\tau-1})\in U_2^c}{G_n}{\pi_2}}.\] The proof now follows by Markov's inequality. 
\end{proof}
\begin{proof}[Proof of Proposition~\ref{prop:TRel}] ($t_{\rm{rel}}$ for 2 communities model when $\alpha \ll \frac{1}{\log n}$) The lower bound follows in exactly the same way as for the larger values of $\alpha$. As we have shown in Theorem~\ref{thrm:Cutoff} that $t_{\text{mix}}\lesssim \frac{1}{\alpha}$ with high probability and as the mixing time is always an upper bound on the relaxation time, this completes the proof. 
\end{proof}
 \section{Back to simple random walk}\label{section:simplefromlazy}
 We now turn to proving cutoff for simple random walk. First recall that absolute relaxation time is defined to be the inverse of the absolute spectral gap. We say that a Markov chain has cutoff window $W$ if

  We will use the following lemma. 
 
 \begin{lemma}\label{lem:cutofflazygivessimple} Let $P$ be a transition matrix with $\varepsilon$-mixing time $t(\varepsilon)$. Let $t_{L}(\varepsilon)$ be the $\varepsilon$-mixing time of $\frac{I+P}{2}$, i.e.\ of the lazy version of $P$. If the lazy version exhibits cutoff and the absolute relaxation time of $P,$ denoted by $t_{\mathrm{rel}}^{*}$, satisfies $t_{\mathrm{rel}}^{*} \ll t_{L}(\frac{1}{4})=:t_L$, then there is cutoff for the matrix $P$. Moreover, if there exist positive constants $C(\varepsilon)$ and a function $\mathcal{W}$  satisfying
$|t_L(\varepsilon) - t_L(1-\varepsilon)|\leq C(\varepsilon) \mathcal{W}$ and $\mathcal{W}\gg t_{\rm rel}^*$, then we have $|t(\varepsilon)-\frac{1}{2}t_{{L}}(\varepsilon)|\le C'(\varepsilon)\mathcal{W}$ for a positive constant $C'(\varepsilon)$. 
 On the other hand, if the lazy chain does not exhibit cutoff and $t_L\asymp  t^*_{\rm rel},$ this implies that for all~$\varepsilon$ we have $t(\varepsilon)\asymp t^*_{\rm rel}$, and hence there is no cutoff for $P$.  
 \end{lemma}
 \begin{proof} Let $X$ be a Markov chain with transition matrix $P$.
Define $ p_x(\alpha,t )=\max_{A\subset V,\pi(A)>\alpha}\prstart{\tau_A>t}{x}$, where $\tau_A=\inf\{t:X_t \in A\}$ is the first hitting time of the set $A$.  We define
\[ 
{\rm{hit}}_{\alpha}(\varepsilon) = \min \{t: \max_{x}p_x(\alpha,t) \le\varepsilon\}.
\]
Let ${\rm hit}^L_\alpha(\varepsilon)$ and $p^L$ be defined as above with respect to the lazy chain. From~\cite[Proposition 1.8]{CharacterizationCutoff} and~\cite[Remark 1.9]{CharacterizationCutoff} we have that for $\varepsilon \in (0,1/4)$
\begin{align}\label{eq:cutoffhitmix} 
\begin{split}&{\rm hit}_{\frac{1}{2}}(3\varepsilon/2)-\lceil2t^*_{\rm rel}|\log \varepsilon|\rceil\le t(\varepsilon )\le {\rm hit}_{\frac{1}{2}}(\varepsilon/2)+\lceil t^*_{\rm rel}|\log(\varepsilon/4)|\rceil,\text{  and   }
\\&{\rm hit}_{\frac{1}{2}}(1-\varepsilon/2)-\lceil2t^*_{\rm rel}|\log \varepsilon|\rceil\le t(1-\varepsilon )\le {\rm hit}_{\frac{1}{2}}(1-2\varepsilon)+\lceil\frac{1}{2} t^*_{\rm rel}\log(8)\rceil.\end{split}
\end{align}
The same statements also hold for $t_L(\varepsilon)$ and ${\rm hit}_{1/2}^L(\varepsilon)$. First in the case when there is cutoff, as $t_{\rm rel}^*\ll\mathcal{W}\ll t_L$, the inequalities above for the lazy chain give us that for all $\varepsilon \in (0,1/8)$  
\begin{align}\label{eq:hitmixtl14}
{\rm hit}^L_{\frac{1}{2}}(\varepsilon)=t_L+O(\mathcal{W}) \quad \text{ and } \quad {\rm hit}^L_{\frac{1}{2}}(1-\varepsilon)=t_L+O(\mathcal{W}) .
\end{align}
 If $\tau_A^S$ and $\tau_A^L$ are the hitting times of $A$ by $X$ and the lazy version of $X$ respectively, we see that
\[
\prstart{\tau_A^L>t}{x}=\sum_{u\le t}\left(\frac{1}{2}\right)^t{t \choose u}\prstart{\tau_A^S>u}{x}.
\]
Using large deviations for the binomial distribution by taking $C$ a sufficiently large constant, we get
\[
\pr{{\rm Bin}\left(t,\frac{1}{2}\right)\notin \left(\frac{t}{2}-C\sqrt{t}, \frac{t}{2}+C\sqrt{t}\right)}\le  \varepsilon^2.
\]
Hence, we also have 
\begin{align}\label{eq:pforsimplelazy}
p_x\left(\frac{1}{2},\frac{t}{2}+C\sqrt{t}\right)\le \frac{p_x^L\left(\frac{1}{2},t\right)}{1-\varepsilon^2} \quad  \text{ and } \quad p_x\left(\frac{1}{2},\frac{t}{2}-C\sqrt{t}\right)\ge p_x^L\left(\frac{1}{2},t\right)-\varepsilon^2.
\end{align}
Using this for $t={\rm hit}_{\frac{1}{2}}^L(\varepsilon)$ and $t={\rm hit}_{\frac{1}{2}}^L(1-\varepsilon)$,~\eqref{eq:hitmixtl14} and the fact that $\mathcal{W}^2\gtrsim t_L$ (see for instance~\cite[Theorem~C]{HermonLacoinPeres} or \cite[Theorem~3.4]{ChenSaloffCoste}) we get for $\varepsilon<1/16$
\[
{\rm hit}_{\frac{1}{2}}(\varepsilon)=\frac{1}{2}t_L+O(\mathcal{W}) \quad \text{ and } \quad {\rm hit}_{\frac{1}{2}}(1-\varepsilon)=\frac{1}{2}t_L+O(\mathcal{W}) .
\]
The proof is now complete in the case when there is cutoff, using~\eqref{eq:cutoffhitmix}.
Notice that when there is no cutoff for the lazy walk we have by~\eqref{eq:cutoffhitmix} that for all  $\varepsilon\in \left(0,\frac{1}{4}\right)$
\[{\rm hit}^L_{\frac{1}{2}}(3\varepsilon/2)\lesssim  t_L\quad \text{ and }\quad  {\rm hit}^L_{\frac{1}{2}}(1-\varepsilon/2)\lesssim  t_L,
\]
since $t_{\rm rel}\lesssim t_L$.
Now this together with~\eqref{eq:pforsimplelazy} gives that ${\rm hit}_{\frac{1}{2}}(\varepsilon)\lesssim t_L$ and ${\rm hit}_{\frac{1}{2}}(1-\varepsilon)\lesssim t_L$ and the proof of an upper bound on $t(\varepsilon)$ by $t_{\rm rel}^*$ follows from~\eqref{eq:cutoffhitmix} and the assumption $t_L\asymp t_{\rm rel}^*$. Since  $t_{\rm rel}^*\lesssim t(\varepsilon)$ (see for instance~\cite[Theorem~12.4]{MixingBook}), we get $t_{\rm rel}^*\asymp t(\varepsilon)$, and hence there is no cutoff for the chain with matrix $P$ either (see~\cite[Proposition~18.4]{MixingBook}).  
  \end{proof}

 The lemma above together with Theorem~\ref{thrm:mCutoff} and~\ref{thrm:Cutoff} in the case of a lazy walk gives that in order to prove Theorem~\ref{thrm:cutoffSimplem} and~\ref{thrm:cutoffSimple2} for the simple random walk it is enough to bound the absolute spectral gap by $\frac{1}{\alpha}$. 
 To do this we use the following result from~\cite{lazysimple}.
\begin{theorem}[\cite{lazysimple}]\label{thrm:JonathansPaperWithStudents}Let $P$ be a reversible Markov chain on a state space $\Omega$ of size $n$ with stationary distribution $\pi$.
Let $1=\lambda_1\geq\lambda_2\geq\ldots\geq \lambda_n$ be eigenvalues of $P$ and $f_i$ the corresponding unit eigenfunctions, such that $P f_i=\lambda_i f_i$ and $\estart{f_i f_j}{\pi}=\mathds{1}(i=j)$.
Then if $1+\lambda_n \leq c\left(1-\lambda_2\right)$ for some absolute constant $c \in (0,1) $ then
 $\text{Var}_{\pi}\left|f_n\right| \leq \frac{1+\lambda_n}{1-\lambda_2}$. Moreover, if we let
\[F_{+}:=\left\{x: f_n(x) \geq 0\right\} \text { and } F_{-}:\left\{x: f_n(x)<0\right\}\]
then $\left|\pi\left(F_{+}\right)-1 / 2\right|=\left|\pi\left(F_{-}\right)-1 / 2 \right|\lesssim\frac{1+\lambda_n}{1-\lambda_2}$ and the parity breaking time defined by
 \[S:=\inf\left\{i:\left(X_{i-1}, X_i\right) \in F_{+}^2 \cup F_{-}^2\right\}\] satisfies for some $\beta>\left|\lambda_n\right|$ that for all $k$
\[\left(1-\frac{2\left(1+\lambda_n\right)}{1-\lambda_2}\right)(\beta^{2k}+\beta^{2k+1}) \leq \prstart{S>2k}{\pi}+\prstart{S>2k+1}{\pi}\leq (\beta^{2k}+\beta^{2k+1}).\]
\end{theorem}
  
 We also use following two lemmas which we prove in Appendix~\ref{appendixSmallSets} as they contain similar calculations as the proof that small sets in the two communities model have good expansion.

 \begin{lemma}\label{lemma:ConfigurationModelNotBipartite}
Let $G=(V,E)$ be a configuration model on $n$ vertices, with minimal degree $3$ and maximal degree $\Delta$. Then there exists a constant $\delta>0$ such that the graph $G$ with high probability satisfies the following: for all sets $A\subset V$ and $B=A^c$ we have that \[\frac{|\{x\sim y, (x,y)\in A^2\cup B^2\}|}{\sum_{x\in V}\deg(x)}\ge \delta.\]
\end{lemma}

 \begin{lemma}\label{rmk:2ComNotBipartite} Let $G=(V,E)$ be the two communities model on $n$ vertices with $\alpha=\alpha_1+\alpha_2\leq {1}$. Then there exists a constant $\delta>0$ depending on $\alpha_1/\alpha_2$, such that the graph $G$ with high probability satisfies the following: for all sets $A\subset V$ and $B=A^c$ we have that \[\frac{|\{x\sim y, (x,y)\in A^2\cup B^2\}|}{\sum_{x\in V}\deg(x)}\ge \delta.\]
\end{lemma}
 
\begin{proof}[Proof of Proposition~\ref{prop:TRel}]($t_{\rm rel}^*$ for both models) The lower bound on $t_{\rm rel}^*$ follows as we always have that $t_{\rm rel}^*\ge t_{\rm rel}$ and we established Proposition~\ref{prop:TRel} for $t_{\rm rel}$. We now upper bound $t_{\rm rel}^*$ for the $m$-communities model. We prove the bound by contradiction. We assume that $1+\lambda_n\le \varepsilon\alpha$ for a small~$\varepsilon$ to be determined.  From the bound on $t_{\rm rel}$ from Proposition~\ref{prop:TRel} we get that $1-\lambda_2\ge c \alpha$. This gives that the time $S$ from Theorem~\ref{thrm:JonathansPaperWithStudents} satisfies that 
\begin{equation}\label{eq:P(S>2)}\prcond{S>2}{G_n}{\pi}\geq \left(1-\frac{2\varepsilon}{c}\right)\beta^3\ge \left(1-\frac{2\varepsilon}{c}\right)\left(1-\varepsilon\right)^3. 
\end{equation}
Recalling that $N_i$ is the number of half edges emanating from community $i$, we claim that using Lemma~\ref{lemma:ConfigurationModelNotBipartite} we get that with high probability
\begin{equation}\label{eq:Shappenedat1}\prcond{S=1}{G_n}{\pi}\geq \frac{3}{\Delta}\cdot \delta\cdot \min_{i\le m}\frac{N_i}{N}.\end{equation}
Indeed,  we know that the graph $G_n$ restricted to community $i$ satisfies with high probability that any partition of its vertices into two sets, has at least $\delta$ proportion of edges fully inside one of these sets. Recalling that the internal degree is at least $3$, we get that there are at least $\frac{3}{\Delta}N_i$ half edges which are fully inside community $i$. Therefore for the partition $F_-$ and $F_+$ as in the statement of Theorem~\ref{thrm:JonathansPaperWithStudents}  we get that with high probability there at least $\delta\frac{3}{\Delta}N_i$ half-edges fully inside $F_-$ or~$F_+$. As the walk started from $\pi$ at the first step goes through an edge chosen uniformly at random, we get that with probability at least $\frac{3}{\Delta}\delta\min_{i\leq m}\frac{N_i}{N}$ this walk crossed an edge fully in~$F_+$ or~$F_-$, so \eqref{eq:Shappenedat1} holds. This gives that $\prcond{S>2}{G_n}{\pi}$ is bounded away from $1$ and therefore we obtain a contradiction with \eqref{eq:P(S>2)} by choosing $\varepsilon$ small enough and this concludes the proof.

The proof in the case of the two communities model follows in exactly the same way in the case of the $m$ communities model using Lemma~\ref{rmk:2ComNotBipartite} instead of Lemma~\ref{lemma:ConfigurationModelNotBipartite}.
\end{proof} 
 
\begin{proof}[Proof of Theorems~\ref{thrm:cutoffSimplem} and~~\ref{thrm:cutoffSimple2}] (simple random walk case)
The proof follows directly from Lemma~\ref{lem:cutofflazygivessimple}, as we have proved Proposition~\ref{prop:TRel} and  Theorems~\ref{thrm:cutoffSimplem} and~\ref{thrm:cutoffSimple2} for lazy simple random walk.
\end{proof}

We notice that  Lemma~\ref{lem:cutofflazygivessimple} also gives the following.

\begin{remark}\label{rmk:lazytwotimessimple} In the setups of Theorems~\ref{thrm:cutoffSimple2} and~\ref{thrm:cutoffSimplem}, let $t$ and $t_L$ be the $\frac{1}{4}$ mixing times of the simple and lazy simple random walk, respectively. Then there exists a constant $C$ depending on $\Delta$, such that  with high probability we have $|t_L-2t|\leq C\sqrt{\frac{ \log|V_k|}{\alpha_k}}.$
\end{remark}

\subsection{Comparison of the cutoff time for simple and non-backtracking random walk}
 \label{sec:comparison}
 
 For an undirected graph $G$ with edges $E$ the non-backtracking random walk is the walk on directed edges of $G$ which evolves according to the transition matrix $P$ defined as follows. For $\{x,y\}\in E$   \[ P((x,y),(z,t))=\frac{1}{\deg(y)-1}\mathds{1}(z=y, \{z,t\}\in E),
 \] where for $x,y\in G$ we write $(x,y)$ for a directed edges and $\{x,y\}$ for undirected and $\deg{x}$ for the degree of $x$. 
 In words, this is the simple random walk conditioned on not moving back to the vertex it just came from. 
 
 In~\cite{AnnasPaper} it was shown that under certain additional assumptions of the degree sequences the $\varepsilon$-mixing time of the non-backtracking random walk denoted $t_{\rm{NB}}(\varepsilon)$ satisfies 
 \[
 t_{\rm NB}(\varepsilon)=\frac{\log n}{\frac{1}{N}\sum_{x\in V_k}{\text{deg}(x)\log (\text{deg}(x)-1)}}+o(\log n).
 \]
  We will show that under certain conditions $t(\varepsilon)>\log n/\mathfrak{h}_X$ for $\mathfrak{h}_X$ which satisfies $$\mathfrak{h}_X<\frac{1}{N}\sum_{x\in V_k}{\text{deg}(x)\log (\text{deg}(x)-1)}$$ which gives the following proposition. 
 
 \begin{prop}\label{prop:nbrwmixesfaster} Consider the setting of Theorem~\ref{thrm:cutoffSimple2} for $\alpha_k\gg 1/\log|V_k|$. There is a small constant $c(\Delta)$ depending on the maximal degree $\Delta$ and $N_1/N_2$ such that if  $\alpha_k\le c(\Delta)$ for all large enough $k$ or if the average degree of vertices in the first community equals to the average degree of vertices in the second community then for all $\varepsilon\in(0,1) $ we have
\[
t_{\rm{mix}}(G_k,\varepsilon)>t_{\rm NB}(G_k,\varepsilon),
\] 
where $t_{\mathrm{NB}}$ stands for the mixing time of the non-backtracking random walk. 
 \end{prop}

To establish the above proposition, we first need the following lemma. 
 
 \begin{lemma}\label{lemma:EntropyInTermsOfWalk} Let  $X$ and $\widetilde{X}$ be two independent lazy simple random walks on $T$ started from the root. Let $\mathfrak{h}$ and $\nu$ be constants from Proposition \ref{prop:LERWLimit} and Lemma~\ref{lemma:Speed}, respectively. We have that almost surely 
\[
                -\frac{1}{k} \log \prcond{X_k=\widetilde{X}_k}{X,T}{} \to \nu\mathfrak{h}, \text{ as } k \to \infty.
\]
\end{lemma}
We prove this lemma in Appendix~\ref{AppendixEntropy}. 

\begin{proof}[Proof of Proposition~\ref{prop:nbrwmixesfaster}] From the proof of Theorem~\ref{thrm:Cutoff} and Remark~\ref{rmk:lazytwotimessimple} we know that the mixing time of the simple random walk on the two community model is up to smaller order terms equal to $\frac{\log n}{2 \nu \mathfrak{h}},$ where~$\mathfrak{h}$ and~$\nu$ are the constants from Proposition \ref{prop:LERWLimit} and Lemma~\ref{lemma:Speed}, respectively. As it is easy to check that the entropy of the simple random walk is  twice the one of the lazy walk, using Lemma~\ref{lemma:EntropyInTermsOfWalk} we have that almost surely 
 \[\frac{\log n}{2 \nu \mathfrak{h}}=\frac{\log n}{-\lim_{k\to \infty}\frac{1}{k} \log \prcond{X_k=\widetilde{X}_k}{X,T}{}}, \]where $X$ and $\widetilde{X}$ are two independent simple random walks on T.  Therefore, in order to prove the statement, it is enough to compare the entropy of the simple random walk on $T$ with 
 \[
 \sum_{x\in V}\frac{1}{N}\text{deg}(x)\log(\text{deg}(x)-1),
 \] which is easily seen to be equal to the entropy of the non-backtracking random walk on~$T$, denoted by $\mathfrak{h}_Y$. We present the calculations comparing $\mathfrak{h}_X$ and $\mathfrak{h}_Y$ in Appendix~\ref{appendixNBRW}. 
 \end{proof}

{\bf Acknowledgements} We are grateful to Anna Ben-Hamou for useful discussions. We are very grateful to the referees for a careful reading and many insightful comments.
Jonathan Hermon's research was supported by an NSERC grant. An\dj ela \v{S}arkovi\'c was supported by DPMMS EPSRC International Doctoral Scholarship and by the Trinity College Internal Graduate Studentship.

\begin{appendix}

\section{Ergodic theory on multi-type trees} \label{appendixA}
{
\proof[Proof of Theorem~\ref{claim:StationaryMGW}]  We follow the proof from \cite{GWTSpeedHarmMeasure}. 
For Borel sets $A$ and $B$ of trees we let $$\widehat{p}_{\text{SRW}}(A,B)=\int_Ap_{\text{SRW}}(T,B)d\text{MGW}_{\pi_Q}(T).$$ We need to show that $\widehat{p}_{\text{SRW}}(A,B)=\widehat{p}_{\text{SRW}}(B,A).$ For two disjoint rooted  multi type  trees we define $[T_1,T_2]$ to be the tree rooted at  the root of $T_1$ obtained by joining the roots of $T_1$ and $T_2$ by an edge. We extend this operation to sets $C,D$ of rooted multi type trees by letting $$[C,D]=\{[T_1,T_2]: T_1\in C, T_2\in D\}.$$ 
It is enough to show that $\widehat{p}_{\text{SRW}}(A,B)=\widehat{p}_{\text{SRW}}(B,A),$ for sets of form $A=[C,D],$ $B=[D,C]$ where $C,D$ are disjoint sets of trees with fixed types of root and first level, as such sets generate the $\sigma$-algebra up to sets of measure $0.$  

For trees $T_1,\ldots, T_{d}$ we let $\bigvee_{i=1}^{d}T_i$ be the tree rooted at some new vertex $v$ obtained by joining the roots of the trees $T_1,\ldots, T_{d}$ to $v$. We extend these definitions to the sets of trees $C_1,\ldots, C_{d}$ by $\bigvee_{i=1}^{d}C_i=\{\bigvee_{i=1}^{d}T_i: T_i\in C_i\}.$

We let the type of  the root of all of the trees in $C$ be $\theta_1$ and the type of the root of all of the trees in $D$ be $\theta_2.$ We also let $d_1^{\text{out}}$ and $d_1^{\text{int}}$ be the number of outgoing and internal offspring of the root of all of the trees in $C$ and $d_2^{\text{out}},$ $d_2^{\text{int}}$ be the number of outgoing and internal offspring of the root of all of the trees in $D$.  Set $d_1=d_1^{\text{out}}+d_1^{\text{int}}$ and $d_2=d_2^{\text{out}}+d_2^{\text{int}}$. Then we can further assume that $$C=\bigvee_{s=1}^{d_1} C_s, \,\,\, D=\bigvee_{s=1}^{d_2}D_s,$$ and that set $D$ is also disjoint from all $C_1,\ldots C_{d_1}$ and that $C$ is disjoint from $D_1,\ldots, D_{d_2}$, and that for all $i$ the types of roots of all trees in $C_i$ and $D_i$ are fixed, as these sets will also generate the $\sigma$-algebra up to sets of measure zero. 

We define the conditional multi type \random measure $\text{CMGW}_{i,j}$ on $\mathcal{T},$ to be the law of the tree obtained by taking an offspring $v$ of the root of the tree generated according to $\text{MGW}(i)$ and all of its descendants and conditioning on the type of $v$ being $j.$ In other words, if we have a vertex of type $i$ in the multi type \random tree, with an offspring of type $j$, then the tree obtained by taking this offspring and all of its descendants 
has the law of $\text{CMGW}_{i,j}$. 

Label by $\theta_1$ and $\theta_2$ the types of the roots of trees in $A$ and $B$ respectively, and we also let $\Theta(C_j)$ and $\Theta(D_j)$ be the types of the roots of all trees in $C_j$ and $D_j$ respectively. 
Set $\widehat{V}$ to be the subset of vertices  of $V$ of type $\theta_1$ with the same number of outgoing and internal edges as the roots of trees in $A$. Setting $p_{i,j}=\frac{E_{i,j}}{\sum_{k\ne i}E_{i,k}}$ and $p_{i,i}=1$, we get
  \begin{align*}\text{MGW}_{\pi_Q}(A)=&\text{MGW}_{\pi_Q}([C,D])\\=&\pi_Q(\theta_1) \frac{\sum_{u\in \widehat{V}}|\deg(u)|}{\sum_{u\in V:\Theta(u)=\theta_1}|\deg(u)|}(d_1^{\text{out}}+\mathds{1}(\theta_1\neq \theta_2))!(d_1^{\text{int}}+\mathds{1}(\theta_1=\theta_2))!\\&\cdot p_{\theta_1,\theta_2}\text{CMGW}_{\theta_1,\theta_2}(D) \prod_{s=1}^{d_1} p_{\theta_1,\Theta(C_s)}\text{CMGW}_{\theta_1,\Theta(C_s)}(C_s)  \\=&\pi_{Q}(\theta_1)\frac{(d_1+1)|\widehat{V}|}{\sum_{j=1}^mE_{\theta_1,j}} (d_1^{\text{out}}+\mathds{1}(\theta_1\neq \theta_2))!(d_1^{\text{int}}+\mathds{1}(\theta_1=\theta_2))!\\&\cdot p_{\theta_1,\theta_2}\text{CMGW}_{\theta_1,\theta_2}(D) \prod_{s=1}^{d_1} p_{\theta_1,\Theta(C_s)}\text{CMGW}_{\theta_1,\Theta(C_s)}(C_s). 
 \end{align*}
 The above expression holds as the first two terms represent the probability to choose the root of the type we want and the first layer of the tree with the wanted number of internal and outgoing edges. The factorial terms represent the number of ways to reorder the trees from the sets $C_1,\ldots C_{d_1}, D$ (they are all disjoint) and then these trees as well as the exact type of their roots in the outgoing case are sampled according to the $\text{CMGW}$ law with the appropriate type parameters.
 
 We can now notice that if $\theta_1\ne \theta_2$ $(d_1^{\text{out}}+1)|\widehat{V}|=\left|\left\{u\in \widehat{V}, v\in V: u\sim v,\Theta(v)\neq \theta_1\right\}\right|,$ as for each $u\in \widehat{V}$ there are $d_1^{\text{out}}+1$ choices for a neighbour $v$ of type which is not $\theta_1.$ We can now notice that $\frac{|\left\{u\in \widehat{V}, v\in V: u\sim v,\Theta(v)\neq \theta_1\right\}|}{\sum_{j\ne \theta_1}E_{\theta_1,j}}$ gives the probability of the number of offspring in the first level in $\text{CMGW}_{\theta_2,\theta_1}$ which are outgoing and internal  being $d_1^{\text{out}}$ and $d_1^{\text{int}},$ respectively. This gives that \begin{align*}&\text{CMGW}_{\theta_2,\theta_1}(C)\\&=\frac{\left|\left\{u\in \widehat{V}, v\in V: u\sim v,\Theta(v)\neq\theta_1\right\}\right|}{\sum_{j\ne \theta_1}E_{\theta_1,j}}d_1^{\text{int}}!d_1^{\text{out}}!\prod_{s=1}^{d_1}p_{\theta_1,\Theta(C_s)}\text{CMGW}_{\theta_1,\Theta(C_s)}(C_s).\end{align*}
 Therefore when $\theta_1\ne \theta_2$, 
 \begin{align*}&\text{MGW}_{\pi_Q}(A)\\&=\pi_{Q}(\theta_1)\frac{(d_1+1)\sum_{j\ne \theta_1}E_{\theta_1,j}}{(d_1^{\text{out}}+1)\sum_{j=1}^m E_{\theta_1,j}}(d_1^{\text{out}}+1)p_{\theta_1,\theta_2}\text{CMGW}_{\theta_1,\theta_2}(D)\text{CMGW}_{\theta_2,\theta_1}(C)  
\\& = \frac{\sum_{j=1}^m E_{\theta_1,j}}{\sum_{u\in V}|\deg(u)|}\frac{(d_1+1)E_{\theta_1,\theta_2}}{\sum_{j=1}^m E_{\theta_1,j}}\text{CMGW}_{\theta_1,\theta_2}(D)\text{CMGW}_{\theta_2,\theta_1}(C) \\&= \frac{E_{\theta_1,\theta_2}}{\sum_{u\in V}|\deg(u)|}(d_1+1)\text{CMGW}_{\theta_1,\theta_2}(D)\text{CMGW}_{\theta_2,\theta_1}(C). \end{align*} 
Similarly, when $\theta_1=\theta_2$ we get that
 \begin{align*}\text{MGW}_{\pi_Q}(A)&=\pi_{Q}(\theta_1)\frac{(d_1+1)E_{\theta_1,\theta_1}}{(d_1^{\text{int}}+1)\sum_{j=1}^m E_{\theta_1,j}}(d_1^{\text{int}}+1)\text{CMGW}_{\theta_1,\theta_2}(D)\text{CMGW}_{\theta_2,\theta_1}(C)  \\&= \frac{E_{\theta_1,\theta_2}}{\sum_{u\in V}|\deg(u)|}(d_1+1)\text{CMGW}_{\theta_1,\theta_2}(D)\text{CMGW}_{\theta_2,\theta_1}(C). \end{align*}

Given that the sets of trees we considered above are disjoint the probability to move from a tree in $A$ to a tree in $B$ is only possible if the root moves to the root of a tree from the set $D$ rather than $C_1,\ldots C_{d_1}$ and this then has probability $\frac{1}{d_1+1}$ giving that 
 $$\widehat{p}_{\text{SRW}}(A,B)=\frac{E_{\theta_1,\theta_2}}{\sum_{u\in V}|\deg(u)|}\text{CMGW}_{\theta_1,\theta_2}(D)\text{CMGW}_{\theta_2,\theta_1}(C). $$ As this expression is symmetric the result follows. 
\qed}
\begin{remark} 
\rm{\label{rmk:stationarity2types} The proof of Theorem~\ref{claim:StationaryMGW} also works (and is simpler) when the tree is generated according to the two type random tree from Definition \ref{def:2TGWTree}.}
\end{remark}

\begin{proof}[Proof of Lemma~\ref{lemma:Speed}]The first two claims follow easily from Lemma \ref{lemma:ConcentrationPhi}. For the last claim, let $C$ be such that the first probability bound holds. Then 
\begin{align*} &\pr{\sup_{s:s\le t}d\left(\rho,X_s\right)>\nu t+2C\left(\sqrt{\frac{t}{\alpha}}+\frac{1}{\alpha}\right)}
\\ &\le \pr{d\left(\rho,X_t\right)<\nu t+C\left(\sqrt{\frac{t}{\alpha}}+\frac{1}{\alpha}\right),\sup_{s:s\le t}d\left(\rho,X_s\right)>\nu t+2C\left(\sqrt{\frac{t}{\alpha}}+\frac{1}{\alpha}\right)}+\varepsilon
\\ &\le \sum_{s:s\le t}\pr{d\left(\rho,X_t\right)<\nu t+C\left(\sqrt{\frac{t}{\alpha}}+\frac{1}{\alpha}\right),d\left(\rho,X_s\right)>\nu t+2C\left(\sqrt{\frac{t}{\alpha}}+\frac{1}{\alpha}\right)} +\varepsilon
\\ &\le t\cdot e^{-c\left(\sqrt{\frac{t}{\alpha}}+\frac{1}{\alpha}\right)} +\varepsilon,\end{align*} 
for a positive constant $c$, where for the last bound we used Lemma~\ref{lemma:ReturnProb}. 
\end{proof}

\begin{proof}[Proof of Lemma~\ref{lemma:EandVar}] The bounds~\eqref{eq:MomentSigma0} and~\eqref{eq:MomentY} follow in exactly the same was as the bounds on the moments of $(Y_i)$ in~\cite[Lemma~3.14]{PerlasPaper} with the only difference being that we need to condition on the type. More precisely, in our setting we have that for $\theta \in \left\{0,1\right\}$, $Y_k$ conditioned on $\Theta\left(X_{\sigma_{k-1}}\right)=\theta$  and $Y_\ell $ conditioned on $\Theta\left(X_{\sigma_{\ell  -1}}\right)=\theta,$ have the same distribution. Also $Y_k$ is independent of $\mathcal{B}_K\left(\rho\right)$ conditioned on the type of the vertex $X_{\sigma_1}$.

We only need to prove the variance bound~\eqref{eq:VarY}. We will show that there is a positive constant $c$ such that for $i<j<k$ 
\begin{equation}\label{eq:CovBigi+j}\econd{\left(Y_i-\estart{Y_2}{\pi_{\Sigma}}\right)\left(Y_j-\estart{Y_2}{\pi_{\Sigma}}\right)}{\mathcal{B}_{K}(\rho)=T_0}\lesssim 2^{-c\alpha(j-i)}+e^{-c(j-i)}\end{equation}
The result then follows, as given that all moments of $Y_i$ are bounded \begin{align*}&\econd{\left(\sum_{i=1}^k\left(Y_i-\estart{Y_2}{\pi_{\Sigma}}\right)\right)^2}{ \mathcal{B}_K\left(\rho\right)=T_0}{}&\\&\lesssim k+\sum_{i\neq j\le k}\econd{\left(Y_i-\estart{Y_2}{\pi_{\Sigma}}\right)\left(Y_j-\estart{Y_2}{\pi_{\Sigma}}\right)}{\mathcal{B}_K\left(\rho\right)=T_0}{} \\&\lesssim k +\sum_{i< j\le k}\left(2^{-c\alpha(j-i)}+e^{-c(j-i)}\right) \lesssim \frac{k}{\alpha}.\end{align*}

 We now turn to proving \eqref{eq:CovBigi+j}. We first recall some definitions from~\cite[Lemma~3.14]{PerlasPaper}. For each $i$ we let ${\xi}\left(i\right)$ be a loop erased random walk on $T\left(X_{\sigma_{i-1}}\right)$ started from the root $X_{\sigma_{i-1}}$ 
 and we take it to be independent from $X$.
 For each $i$ let $X^i$ be the walk that generates the loop erased path $\xi\left(i\right)$. Now for $i<j$, let $\xi\left(i,j\right)$ be the loop erased path obtained from the path $X^i$ when we run it until the first time that $X^i$ reaches the level of $X_{\sigma_{\lfloor\frac{i+j}{2}\rfloor}}$. Set 
 \begin{align*} & Z_i=\prcond{X_{\sigma_i}\in \xi\left(i\right)}{T\left(X_{\sigma_{i-1}}\right), X}{}, \,\, \\& Z_{i,j}=\prcond{X_{\sigma_i}\in\xi\left(i,j\right)}{T\left(X_{\sigma_{i-1}}\right),X}{}\text{  and   } Y_{i,j}=-\log Z_{i,j}. 
 \end{align*}
 Let $A\left(i,j\right)$ be the event that $X^i$ returns to $X_{\sigma_i}$ after reaching the level of $X_{\sigma_{\lfloor\frac{i+j}{2}\rfloor}}$ for the first time. Then as in~\cite{PerlasPaper}
\[|Z_i-Z_{i,j}|\leq \prcond{A\left(i,j\right)}{ T\left(X_{\sigma_{i-1}}\right),X}{}. 
\] 
Using Lemma \ref{lemma:ReturnProb} gives that there exists a constant $c$ so that 
\[\prcond{A\left(i,j\right)}{T\left(X_{\sigma_{i-1}}\right),X}{}\leq e^{-c\left(j-i\right)}. \] 
Using that $|\log x-\log y|\leq \frac{|x-y|}{x\wedge y}$ gives 
\[|Y_{i,j}-Y_{i}|=|\log Z_{i,j}-\log Z_{i}|\leq \frac{|Z_{i,j}-Z_i|}{|Z_{i,j}\wedge Z_i|}. 
\]
 Let $B\left(i,j\right)=\left\{d\left(X_{\sigma_i},X_{\sigma_{i-1}}\right)\leq \lfloor\frac{j-i}{C} \rfloor \right\}$ for some large positive constant $C$. On $B\left(i,j\right)$ lower bounding the probability that $X^i$ visits $X_{\sigma_i}$ for the first time without backtracking and then escapes we get  $$Z_{i,j}\wedge Z_i \geq c\left(2\Delta\right)^{-\frac{\left(j-i\right)}{C}},$$ where $c$ is the positive constant from Lemma \ref{lemma:ReturnProb} and $\Delta$ is the maximum degree. 
So for $C$ sufficiently large and some positive constant $c''$ 
\[ |Y_{i,j}-Y_i|\mathds{1}\left(B\left(i,j\right)\right)\leq e^{-c''\left(j-i\right)}. 
\]
Using the exponential tails of $\phi_i-\phi_{i-1}$ from Lemma \ref{lemma:RegIndep}, we get that for a positive constant~$c_1$ 
\[\prstart{B\left(i,j\right)^c}{T_0}\leq e^{-c_1\lfloor\frac{j-i}{C}\rfloor}.
\]  
Using Cauchy-Schwartz, \eqref{eq:MomentY} and the bound above we get that for $j>i$ 
\begin{align*}
&\estart{\left(Y_i-\estart{Y_2}{\pi_{\Sigma}}\right)\left(Y_j-\estart{Y_2}{\pi_{\Sigma}}\right)}{T_0} \\&\lesssim \estart{\left(Y_i-\estart{Y_2}{\pi_{\Sigma}}\right)\left(Y_j-\estart{Y_2}{\pi_{\Sigma}}\right)\mathds{1}\left(B\left(i,j\right)\right)}{T_0}+e^{-c_1\cdot \frac{j-i}{2C}}.  
\end{align*}  
We can write the last expectation appearing above as 
\begin{align*}\estart{\left(Y_i-Y_{i,j}\right)Y_j\mathds{1}\left(B\left(i,j\right)\right)}{T_0}&-\estart{\left(Y_i-Y_{i,j}\right)\mathds{1}\left(B\left(i,j\right)\right)}{T_0}\estart{Y_2}{\pi_{\Sigma}}
\\&+\estart{\left(Y_{i,j}-\estart{Y_2}{\pi_{\Sigma}}\right)\left(Y_j-\estart{Y_2}{\pi_{\Sigma}}\right)\mathds{1}\left(B\left(i,j\right)\right)}{T_0}.\end{align*} 
The first and second terms are bounded by $e^{-c''\left(j-i\right)}$, as $\estart{Y_2}{\pi_{\Sigma}}$ is bounded by a constant. We note that given the type of $X_{\sigma_{\lfloor\frac{i+j}{2}\rfloor}}$, $Y_{i,j}$ is independent of $Y_j$. Using this independence and the tower property we get that the last term appearing above is equal to
 \begin{align*}&\estart{\escond{\left(Y_{i,j}-\estart{Y_2}{\pi_{\Sigma}}\right)\left(Y_j-\estart{Y_2}{\pi_{\Sigma}}\right)\mathds{1}\left(B\left(i,j\right)\right)}{ \Theta\left(X_{\sigma_{\lfloor\frac{i+j}{2}\rfloor}}\right), \Theta\left(X_{\sigma_{j-1}}\right)}{T_0}}{T_0} 
 \\&= \estart{\escond{\left(Y_{i,j}-\estart{Y_2}{\pi_{\Sigma}}\right)\mathds{1}\left(B\left(i,j\right)\right)}{ \Theta\left(X_{\sigma_{\lfloor\frac{i+j}{2}\rfloor}}\right)}{T_0}\escond{Y_j-\mathbb{E}_{\pi_{\Sigma}}\left[Y_2\right]}{ \Theta\left(X_{\sigma_{j-1}}\right)}{T_0}}{T_0} 
 \end{align*}
For $\theta\in S$ define $f,g: \{1,\ldots, m\}\to \mathbb{R}$ by 
\begin{align*}
 &f\left(\theta\right) =\escond{\left(Y_{i,j}-\estart{Y_2}{\pi_{\Sigma}}\right)\mathds{1}\left(B\left(i,j\right)\right)}{\Theta\left(X_{\sigma_{\lfloor\frac{i+j}{2}\rfloor}}\right)=\theta}{T_0} \quad \mathrm{and}
\\ &g\left(\theta\right) =\escond{Y_j-\estart{Y_2}{\pi_{\Sigma}}}{\Theta\left(X_{\sigma_{j-1}}\right)=\theta}{T_0}.
 \end{align*}
It is enough to bound for $\theta'\in S$ 
\[\escond{f\left(\Theta\left(X_{\sigma_{\left\lfloor\frac{i+j}{2}\right\rfloor}}\right)\right)g\left(\Theta\left(X_{\sigma_{j-1}}\right)\right)}{\Theta(X_{\sigma_1})=\theta'}{T_0}. 
\]
 We need to show that the conditions from Lemma \ref{lemma:Decorrelation} for $f$ and $g$ are satisfied and then as direct consequence we get that this is bounded by $\lesssim 2^{-c\alpha(j-i)},$ for some constant $c$, where we also use Lemma~\ref{lemma:2comMixTypes} and Proposition~\ref{thrm:MixingOfW} to bound the mixing time. This will then complete the proof.
 \newline
 Since $Y_j$ has bounded moments, it follows that $g$ is a bounded function and we also have that its mean under $\pi_\Sigma$ is $0$. It remains to prove that $f$ is also a bounded function. 
By Cauchy-Schwartz, it is enough to bound the second moment of $Y_{i,j}$. 
This can be done in the same way as for $Y_i$ in~\cite{PerlasPaper}. In fact, the same approach as for the $Y_i$'s gives that $Y_{i,j}$ conditioned on $\Theta(X_{\sigma_{i-1}})=\theta$ and $Y_{2,j-i+2}$ conditioned on $\Theta\left(X_{\sigma_{1}}\right)=\theta$ have the same distribution. Therefore, it is enough to bound the moments of $Y_{2,j}$ and this follows as the proof of bounded moments of the $Y_i$'s in~\cite{PerlasPaper}.
 \end{proof}

\begin{proof}[Proof of Proposition~\ref{prop:LERWLimit}] Let $\sigma_i$ and $\phi_i$, $i\ge 0$ be as in Definition \ref{def:RegTimes} and $Y_i$, $i\ge 1$, as defined in Lemma \ref{lemma:EandVar}. Let $X$ be a simple random walk on $T$ generating the loop erasure $\xi$. We have that 
\[-\log\prcond{\left(X_{\sigma_k-1},X_{\sigma_k}\right)\in\widetilde{\xi}}{ X,T}{}=\sum_{i=1}^kY_i -\log\prcond{\left(X_{\sigma_0-1},X_{\sigma_0}\right)\in \widetilde{\xi}}{X,T}{}.
\]
By Lemma~\ref{lemma:EandVar} we have that the variables $Y_i$ have all moments bounded by constants and in~\cite[Lemma~3.14]{PerlasPaper} it was also shown that $Y_k$ is a measurable function of $\left(T\left(X_{\sigma_{k-1}}\right), \left(X_t\right)_{t \ge \sigma_{k-1}}\right)$. If $\Theta\left(X_{\sigma_k}\right)\sim {\pi_{\Sigma}}$ then the sequence $\left(T\left(X_{\sigma_{k-1}}\right), \left(X_t\right)_{t \ge \sigma_{k-1}}\right)$ is stationary, and so is $Y_k$ as well. Letting $\theta\sim {\pi_{\Sigma}}$ and $\tau_{\theta}=\inf\left\{t:\Theta(X_{\sigma_t})=\theta\right\},$ then Birkhoff's  ergodic theorem gives that for a constant $\gamma=\mathbb{E}_{\pi_{\Sigma}}\left[Y_2\right]$, almost surely as $k\to \infty$
\[\frac{\sum_{i=\tau_\theta}^{k}Y_i}{k-\tau_{\theta}}\to \gamma.
\] 
Since $\tau_\theta$ is finite almost surely, we obtain that as $k \to \infty$ almost surely 
\[ \frac{-\log\prcond{\left(X_{\sigma_k-1},X_{\sigma_k}\right)\in \widetilde{\xi}}{X,T}{}}{k}\to \gamma.
\]
Notice that $\xi_{\phi_k}=\left(X_{\sigma_{k-1}},X_{\sigma_k}\right),$ so almost surely as $k\to \infty$
\[ -\frac{\log\prcond{\xi_{\phi_k}\in \widetilde{\xi}}{X,T}{}}{k}\to \gamma.
\]
Lemma \ref{lemma:RegIndep} and the ergodic theorem give that $\frac{\phi_k}{k}\to \estart{\phi_2-\phi_1}{\pi_{\Sigma}}$, as $k\to \infty$ and $\estart{\phi_2-\phi_1}{\pi_{\Sigma}}<\infty$. Therefore, 
\[-\frac{\log\prcond{\xi_k\in \widetilde{\xi}}{\xi,T}{}}{k}\to \frac{\gamma}{\estart{\phi_2-\phi_1}{\pi_{\Sigma}}}=: \mathfrak{h} 
\]
In order to prove the bound on the fluctuations, we use Lemma \ref{lemma:EandVar}. Using the bound from \eqref{eq:VarY} and the bound on the second moment of $-\log\prcond{\left(X_{\sigma_0-1},X_{\sigma_0}\in \widetilde{\xi}\right)}{X,T}{}$ from~\eqref{eq:MomentSigma0} with Markov's inequality give that for all $\varepsilon>0$ there is a constant $C$ such that for all $k\ge K^2$
\begin{align*}&\prstart{\left|\sum_{i=1}^kY_i-\gamma k\right|\ge \frac{C}{2}\left(\sqrt{\frac{k}{\alpha}}+\frac{1}{\alpha}\right)}{T_0}\le \frac{4C'\frac{k}{\alpha}}{C^2\left(\sqrt{\frac{k}{\alpha}}+\frac{1}{\alpha}\right)^2}\le  \frac{\varepsilon}{2} \quad \text{and}
\\&\prstart{\left|-\log\prcond{\left(X_{\sigma_0-1},X_{\sigma_0}\in \widetilde{\xi}\right)}{X,T}{}\right| \ge \frac{C}{2}\sqrt{k}}{T_0}\le \frac{4C_2K^2}{C^2k} \le \frac{\varepsilon}{2}, \end{align*} where $C_2$ and $C'$ are the constants from Lemma~\ref{lemma:EandVar}.  
It now follows that 
\[\prstart{\left|-\log\prcond{\xi_{\phi_k}\in \widetilde{\xi}}{ \xi, T}{}-\gamma k\right|\ge C\left(\sqrt{\frac{k}{\alpha}}+\frac{1}{\alpha}\right)}{T_0}\le \varepsilon.
\]
As in Lemma \ref{lemma:ConcentrationPhi} we let for each $k\in \mathbb{N}$ 
\[N_k=\max\left\{i\ge 0: \phi_i\le k\right\}.
\] 
Then 
\begin{align*}& \prstart{\left|-\log\prcond{\xi_k\in \widetilde{\xi}}{\xi, T}{}-\mathfrak{h} k\right|\ge C\left(\sqrt{\frac{k}{\alpha}}+\frac{1}{\alpha}\right)}{T_0}
\\&\le \prstart{-\log\prcond{\xi_{\phi_{N_k+1}}\in \widetilde{\xi}}{\xi, T}{}>\mathfrak{h}k+ C\left(\sqrt{\frac{k}{\alpha}}+\frac{1}{\alpha}\right)}{T_0}
\\&+\prstart{-\log\prcond{\xi_{\phi_{N_k}}\in \widetilde{\xi}}{\xi, T}{}< \mathfrak{h}k-C\left(\sqrt{\frac{k}{\alpha}}+\frac{1}{\alpha}\right)}{T_0}. \end{align*} 
The concentration of $N_k$ from Lemma~\ref{lemma:ConcentrationPhi} finishes the proof. 
\end{proof}

\section{Proofs of some combinatorial lemmas}\label{appendixSmallSets}
We now present the proof of Lemma~\ref{lemma:SmallSetsBoundary},~\ref{lemma:ConfigurationModelNotBipartite} and~~\ref{rmk:2ComNotBipartite}:

\begin{proof}[Proof of Lemma~\ref{lemma:SmallSetsBoundary}] Consider a subset of vertices $D$ and recall that we labelled the two communities as $V_1$ and $V_2$. Suppose that there are $d_1+d_2$ half-edges in total corresponding to the vertices in $D$, where $d_1$ come from the vertices in the first community and $d_2$ from the second. 
Rewiring the outgoing edges of each community as explained in the proof of Lemma~\ref{lemma:DirEVCom1} gives a configuration model of minimum degree at least 3, and therefore it is an expander with high probability \cite{expander}. This gives that there is a constant $\delta'$ such that with high probability all sets with $x\le \frac{N_i}{2}$ half-edges, in community $i$, for $i \in\{1,2\}$, have boundary whose size is greater then $\delta'x$ in the rewired graph. Fix $c\le \frac{\delta'}{2}$. We first show that with high probability, the graph $G_n$ is such that all sets of vertices $D$ for which $d_1$ and $d_2$ are such that $d_1<cd_2$ or $d_2<cd_1$ have boundary of size at least $\frac{\delta'}{4}\left(d_1+d_2\right).$
Indeed, without loss of generality, if $d_1< cd_2$, the rewiring can increase the boundary in the second community by at most $d_1\le \frac{\delta'}{2}d_2$, meaning that the size of the boundary before rewiring was already at least $\frac{\delta'}{2}d_2\ge\frac{\delta'}{4}\left(d_1+d_2\right)$. 

Therefore from now on we can consider sets $D$ for which both $d_1\ge cd_2$ and $d_2\ge cd_1$ hold. Further, suppose that out of those $d_1$ half-edges  $d_1^o$ are outgoing (i.e.\ go to community 2), while $d^i_1=d_1-d^o_1$ are internal (i.e.\ go into community 1) and $d^i_2$ and $d^o_2$ are corresponding values for the second community. Suppose that $\ell$ is the number of edges with one end in $D\cap V_1$ and the other one in $D\cap V_2$. Suppose $k_i$, for $i\in\{1, 2\}$ is the number of edges with exactly one end in $D$ but also with both ends in the same community $i$. The edge boundary of $D$ then has size $k_1+k_2+\left(d^o_1-l\right)+\left(d^o_2-l\right)$  (therefore the vertex boundary has size at least $\frac{k_1+k_2+\left(d^0_1-l\right)+\left(d^0_2-l\right)}{\Delta}$). If we have a set $D$ with fixed $d_1 $ and $d_2 $ the probability to choose the types of edges and connect $G_n$ in such a way to get the above described values $d^o_1, d^i_1, d^o_2, d^i_2$, $\ell$, $k_1$ and $k_2$ is: 

 \begin{align}\label{djubre1}
 &\frac{{d_1 \choose d^o_1}{d_2\choose d^o_2}{N_1-d_1 \choose p-d^o_1}{N_2-d_2\choose p-d^o_2}}{{N_1\choose p}{N_2\choose p}}\times\frac{{d^o_1\choose \ell}{d^o_2\choose \ell}\ell!{p-d^o_2\choose d^o_1-\ell}{p-d^o_1 \choose d^o_2-\ell}\left(d^o_1-\ell\right)!\left(d^o_2-\ell\right)!\left(p-d^o_1-d^o_2+\ell\right)!}{p! } \\&\label{djubre2}\times \left({d^i_1\choose k_1}k_1!{N_1-p-d^i_1 \choose k_1}{d^i_2\choose k_2}k_2!{N_2-p-d^i_2\choose k_2}\right) \\&\label{djubre3}
\times  \frac{\left(d^i_1-k_1-1\right)!!\left(d^i_2-k_2-1\right)!!\left(N_1-p-d^i_1-k_1-1\right)!!\left(N_2-p-d^i_2-k_2-1\right)!!}{\left(N_1-p-1\right)!!\left(N_2-p-1\right)!!}.\end{align}

The first term in the product in~\eqref{djubre1} corresponds to the probability of choosing the internal and outgoing edges, such that $d^o_i$  is the number of outgoing edges of $D$ in community $i$. The second term in~\eqref{djubre1} is the probability of connecting outgoing edges so that exactly $\ell$ of them have both ends in $D.$ The product in~\eqref{djubre2} and~\eqref{djubre3} corresponds to the probability of connecting the internal edges such that the internal boundaries of $D$ are of size $k_1$ and $k_2$. 
 We will further show that the product of~\eqref{djubre2} and~\eqref{djubre3}  can be bounded above by 
$$\left(d^i_1\right)^{\frac{1}{4}}\left(\frac{{d^i_1\choose k_1}{N_1-p-d^i_1\choose k_1}}{{N_1-p\choose d^i_1}}\right)^{\frac{1}{2}}\left(d^i_2\right)^{\frac{1}{4}}\left(\frac{{d^i_2\choose k_2}{N_2-p-d^i_2\choose k_2}}{{N_2-p\choose d^i_2}}\right)^{\frac{1}{2}}.  $$ 
Indeed, by  Stirling's formula, $m!! \asymp\left(\sqrt{m!}\right)m^{\frac{1}{4}}$, hence $m!!\asymp m^{\frac{1}{2}}(m-1)!!$ and also using $\sqrt{m}\asymp \sqrt{m-1}$ we get
\begin{align*} &\frac{{d^i_1 \choose k_1}k_1!{N_1-p-d^i_1 \choose k_1}\left(d^i_1-k_1-1\right)!!\left(N_1-p-d^i_1-k_1-1\right)!!}{\left(N_1-p-1\right)!!}
\\&\lesssim \frac{{d^i_1 \choose k_1}k_1!{N_1-p-d^i_1 \choose k_1}\left(d^i_1-k_1\right)^{-\frac{1}{4}}\left(N_1-p-d^i_1-k_1\right)^{-\frac{1}{4}}\sqrt{\left(d^i_1-k_1\right)!\left(N_1-p-d^i_1-k_1\right)!}}{\left(N_1-p\right)^{-\frac{1}{4}}\sqrt{\left(N_1-p\right)!}} 
 \\&\lesssim \left(d^i_1\right)^{\frac{1}{4}} \left(\frac{{d^i_1\choose k_1}{N_1-p-d^i_1\choose k_1}}{{N_1-p\choose d^i_1}}\right)^{\frac{1}{2}},\end{align*}
where the last inequality follows as $k_1\le d^i_1$ implies $N_1-p \lesssim d^i_1\left(N_1-p-d^i_1-k_1\right).$
The product of~\eqref{djubre1} is
equal to
\[ \frac{{N_1-p \choose d^i_1}{N_2-p\choose d^i_2}p!}{{N_1\choose d_1}{N_2\choose d_2}\ell!\left(d^o_1-\ell\right)!\left(d^o_2-\ell\right)!\left(p-d^o_1-d^o_2+\ell\right)!}.\] 
Therefore, for small $\delta$, once we have fixed the vertices that lie in the set $D$, the probability that the boundary is smaller than $\delta \left(d_1+d_2\right)$ is \begin{align}\label{eq:BigSumExpansion}\sum_{\substack{d^o_1\le d_1,\, d^o_2\le d_2\\ \ell\le \min\left\{d^o_1,d^o_2\right\}\\ k_1\le d^i_1,  k_2\le d^i_2\\k_1+k_2+d^o_1-\ell+d^o_2-\ell<\delta \left(d_1+d_2\right)}}\frac{(d^i_1)^{\frac{1}{4}}\left({d^i_1\choose k_1}{N_1-p-d^i_1\choose k_1}{N_1-p\choose d^i_1}\right)^{\frac{1}{2}}(d^i_2)^{\frac{1}{4}}\left({d^i_2\choose k_2}{N_2-p-d^i_2\choose k_2}{N_2-p\choose d^i_2}\right)^{\frac{1}{2}}p!}{{N_1\choose d_1}{N_2\choose d_2}\ell!\left(d^o_1-\ell\right)!\left(d^o_2-\ell\right)!\left(p-d^o_1-d^o_2+\ell\right)!}.\end{align}
As the sizes of the communities are comparable and the degrees are bounded between $3$ and $\Delta$, we can choose the constant $\hat{c}$ to be small enough, such that all sets $D$ of at most $\hat{c}n$ vertices have at most $\frac{N_i}{2}$ half-edges in community $i$, for $i\in \left\{1,2\right\}$.  We now bound the number of ways to choose the set of vertices $D$ which have $d_1$ half-edges in the first and $d_2$ in the second community. First, as the minimum degree is 3, we choose at most $\frac{d_1}{3}$ vertices from at most $\frac{N_1}{3}$ vertices in total in the first community. We recall that an antichain is a collection of sets with the property that no two sets are contained in each other. 
Notice that if a certain set of vertices has $d_1$ half-edges then none of its subsets or sets containing it can have $d_1$ half-edges, so the family of sets of vertices in the first community, with exactly $d_1$ half-edges is an antichain. By LYM inequality \cite[Maximal Antichain]{probm}, if $\frac{d_1}{3}\le \frac{n_1}{2}$, the antichain can have size at most ${n_1\choose\lfloor\frac{d_1}{3}\rfloor}$, where $n_1$ is the number of vertices in the first community. If $\frac{d_1}{3}>\frac{n_1}{2}$, then the size is at most ${n_1\choose\lfloor \frac{n_1}{2}\rfloor}\le {\lfloor\frac{N_1}{3}\rfloor\choose\lfloor\frac{n_1}{2}\rfloor} \le {\lfloor\frac{N_1}{3}\rfloor\choose\lfloor\frac{d_1}{3}\rfloor}$, as $d_1\le \frac{N_1}{2}$. Therefore the number of ways to choose the initial set of vertices $D$ which has $d_1$ half-edges in the first and $d_2$ in the second community is bounded above by ${\lfloor \frac{N_1}{3}\rfloor\choose\lfloor \frac{d_1}{3}\rfloor}{\lfloor\frac{N_2}{3}\rfloor\choose \lfloor\frac{d_2}{3}\rfloor}$. So in order to find the bound on the probability that there is a small set $D$, with small boundary we need to multiply the above sum \eqref{eq:BigSumExpansion} by ${\lfloor \frac{N_1}{3}\rfloor\choose\lfloor \frac{d_1}{3}\rfloor}{\lfloor\frac{N_2}{3}\rfloor\choose \lfloor\frac{d_2}{3}\rfloor}$ (to get all sets of $d_1$ half-edges in community $1$, $d_2$ in community $2$) and also sum over all values for $d_1$ and $d_2$, which are suitably bounded from above and for which $d_1\ge cd_2$ and $d_2\ge cd_1$.

We now bound from above the following expression  \[{\lfloor \frac{N_1}{3}\rfloor\choose\lfloor \frac{d_1}{3}\rfloor }{\lfloor\frac{N_2}{3}\rfloor\choose \lfloor\frac{d_2}{3}\rfloor }\frac{(d^i_1)^{\frac{1}{4}}\left({d^i_1\choose k_1}{N_1-p-d^i_1\choose k_1}{N_1-p\choose d^i_1}\right)^{\frac{1}{2}}(d^i_2)^{\frac{1}{4}}\left({d^i_2\choose k_2}{N_2-p-d^i_2\choose k_2}{N_2-p\choose d^i_2}\right)^{\frac{1}{2}}p!}{{N_1\choose d_1}{N_2\choose d_2}\ell!\left(d^o_1-\ell\right)!\left(d^o_2-\ell\right)!\left(p-d^o_1-d^o_2+\ell\right)!}\] when $k_1+k_2+d^o_1-\ell+d^o_2-\ell<\delta \left(d_1+d_2\right)$ and by using the approximation ${n\choose k}\asymp \sqrt{\frac{n}{k\left(n-k\right)}}\exp\left(nH\left(\frac{k}{n}\right)\right)$ where $H\left(x\right)=-x\log\left(x\right)-\left(1-x\right)\log\left(1-x\right)$. Notice that this function is increasing on $\left(0,\frac{1}{2}\right)$ and also that it is concave on $\left(0,1\right)$. Also notice that 
\[\frac{p!}{\ell!\left(d^o_1-\ell\right)!\left(d^o_2-\ell\right)!\left(p-d^o_1-d^o_2+\ell\right)!}=\sqrt{{p\choose d^o_1}{p\choose d^o_2}{d^o_1 \choose \ell}{d^o_2\choose \ell}{p-d^o_1\choose d^o_2-\ell}{p-d^o_2\choose d^o_1-\ell}}.
\] 
We bound the exponential part which appears after these approximations, which means that for $k_1+k_2+d^o_1-\ell+d^o_2-\ell<\delta \left(d_1+d_2\right)$  we bound
\begin{align*}
& \exp\left(\frac{1}{2}\left(d^i_1H\left(\frac{k_1}{d^i_1}\right)+\left(N_1-p-d^i_1\right)H\left(\frac{k_1}{N_1-p-d^i_1}\right)+\left(N_1-p\right)H\left(\frac{d^i_1}{N_1-p}\right)\right. \right.\\&+d^i_2H\left(\frac{k_2}{d^i_2}\right)+\left(N_2-p-d^i_2\right)H\left(\frac{k_2}{N_2-p-d^i_2}\right)+\left(N_2-p\right)H\left(\frac{d^i_2}{N_2-p}\right)\\ &+pH\left(\frac{d^o_1}{p}\right)+pH\left(\frac{d^o_2}{p}\right)+d^o_1H\left(\frac{\ell}{d^o_1}\right)+d^o_2H\left(\frac{\ell}{d^o_2}\right)+\left(p-d^o_1\right)H\left(\frac{d^o_2-\ell}{p-d^o_1}\right)\\& \left.\left.+\left(p-d^o_2\right)H\left(\frac{d^o_1-\ell}{p-d^o_2}\right)\right)+\frac{N_1}{3}H\left(\frac{d_1}{N_1}\right)+\frac{N_2}{3}H\left(\frac{d_2}{N_2}\right)-N_1H\left(\frac{d_1}{N_1}\right)-N_2H\left(\frac{d_2}{N_2}\right)\right)\\&
\le \exp\left(\frac{1}{2}\left(\left(N_1-p\right)\left(H\left(\frac{2k_1}{N_1-p}\right)+H\left(\frac{d^i_1}{N_1-p}\right)\right)\right. \right.\\&+\left(N_2-p\right)\left(H\left(\frac{2k_2}{N_2-p}\right)+H\left(\frac{d^i_2}{N_2-p}\right)\right)  +pH\left(\frac{d^o_1}{p}\right)+pH\left(\frac{d^o_2}{p}\right) \\&\left.+d^o_1H\left(\frac{\ell}{d^o_1}\right)+d^o_2H\left(\frac{\ell}{d^o_2}\right)+\left(p-d^o_1\right)H\left(\frac{d^o_2-\ell}{p-d^o_1}\right)+\left(p-d^o_2\right)H\left(\frac{d^o_1-\ell}{p-d^o_2}\right)\right)\\&\left.-\frac{2N_1}{3}H\left(\frac{d_1}{N_1}\right)-\frac{2N_2}{3}H\left(\frac{d_2}{N_2}\right)\right), \end{align*} where the inequality holds as the function $H$ is concave. Again using the fact that $H$ is concave we get that the above is upper bounded by \begin{align*}& \exp\left(\frac{1}{2}\left(\left(N_1-p\right)H\left(\frac{2k_1}{N_1-p}\right)+N_1H\left(\frac{d_1}{N_1}\right)+\left(N_2-p\right)H\left(\frac{2k_2}{N_2-p}\right)\right. \right.\\&+N_2H\left(\frac{d_2}{N_2-p}\right)+d^o_1H\left(\frac{\ell}{d^o_1}\right)+d^o_2H\left(\frac{\ell}{d^o_2}\right)+\left(p-d^o_1\right)H\left(\frac{d^o_2-\ell}{p-d^o_1}\right)\\&\left.\left.+\left(p-d^o_2\right)H\left(\frac{d^o_1-\ell}{p-d^o_2}\right)\right)-\frac{2N_1}{3}H\left(\frac{d_1}{N_1}\right)-\frac{2N_2}{3}H\left(\frac{d_2}{N_2}\right)\right).
\end{align*}
Using that $H(1-x)=H(x)$ for all $x\in (0,1)$ gives
\begin{align*}
&=\exp\left(\frac{1}{2}\left(\left(N_1-p\right)H\left(\frac{2k_1}{N_1-p}\right)+\left(N_2-p\right)H\left(\frac{2k_2}{N_2-p}\right)\right.\right.\\&\left.+d^o_1H\left(\frac{d^o_1-\ell}{d^o_1}\right)+d^o_2H\left(\frac{d^o_2-\ell}{d^o_2}\right)+\left(p-d^o_1\right)H\left(\frac{d^o_2-\ell}{p-d^o_1}\right)+\left(p-d^o_2\right)H\left(\frac{d^o_1-\ell}{p-d^o_2}\right)\right)\\& \left.-\frac{N_1}{6}H\left(\frac{d_1}{N_1}\right)-\frac{N_2}{6}H\left(\frac{d_2}{N_2}\right)\right).\end{align*}
Again using concavity this is bounded by
\begin{align*}& \exp\left(\frac{1}{2}\left(\left(N_1-p\right)H\left(\frac{2k_1}{N_1-p}\right)+\left(N_2-p\right)H\left(\frac{2k_2}{N_2-p}\right)+2pH\left(\frac{d^o_1+d^o_2-2\ell}{p}\right)\right)\right.\\& \left.-\frac{N_1}{6}H\left(\frac{d_1}{N_1}\right)-\frac{N_2}{6}H\left(\frac{d_2}{N_2}\right)\right) \\&\leq \exp\left(\frac{1}{2}\left(N_1H\left(\frac{2k_1+d^o_1-\ell+d^o_2-\ell}{N_1}\right)+N_2H\left(\frac{2k_2+d^o_1-\ell+d^o_2-\ell}{N_2}\right)\right)\right.\\&\left.-\frac{N_1}{6}H\left(\frac{d_1}{N_1}\right)-\frac{N_2}{6}H\left(\frac{d_2}{N_2}\right)\right) \end{align*} 
As $d_1\ge cd_2$ and $d_2\ge cd_1$ then we can choose $\delta<\frac{c}{2}$ (so $\le \frac{\delta'}{4}$) such that for $k_1+k_2+d^o_1+d^o_2-2\ell\le \delta\left(d_1+d_2\right)$ we have that $H\left(\frac{2k_1+d^o_1-\ell+d^o_2-\ell}{N_1}\right)\le \frac{1}{100}H\left(\frac{d_1}{N_1}\right)$ and $H\left(\frac{2k_2+d^o_1-\ell+d^o_2-\ell}{N_2}\right)\le \frac{1}{100}H\left(\frac{d_2}{N_2}\right)$ (this follows exactly as in the proof that the configuration model is an expander, as we have that $2k_1+d_1^o+d_2^0-2\ell\le 2\delta\left(1+\frac{1}{c}\right)d_1$ and $2k_2+d_1^o+d_2^0-2\ell\le 2\delta\left(1+\frac{1}{c}\right)d_2$ and $\delta$ can be as small as we wish) so that our expression is bounded above by $$\exp\left(-\frac{N_1}{12}H\left(\frac{d_1}{N_1}\right)-\frac{N_2}{12}H\left(\frac{d_2}{N_2}\right)\right)\le \exp\left(-\frac{1}{12}\left(d_1\log\left(\frac{d_1}{N_1}\right)+d_2\log\left(\frac{d_2}{N_2}\right)\right)\right).$$ Adding the non-exponential parts which we omitted at the start and summing over all suitable $d^o_1,d^o_2,d^i_1,d^i_2, k_1,k_2$ and $\ell$, the obtained bound gives that the sum in~\eqref{eq:BigSumExpansion} multiplied by ${\lfloor \frac{N_1}{3}\rfloor\choose\lfloor \frac{d_1}{3}\rfloor}{\lfloor\frac{N_2}{3}\rfloor\choose \lfloor\frac{d_2}{3}\rfloor}$ is bounded above by $$d_1^{4+\frac{1}{4}}d_2^{4+\frac{1}{4}} \exp\left(-\frac{1}{12}\left(d_1\log\left(\frac{d_1}{N_1}\right)+d_2\log\left(\frac{d_2}{N_2}\right)\right)\right).$$ Summing over all $d_1$ and $d_2$ up to size of $\frac{1}{2}N_1$ and $\frac{1}{2}N_2$ gives the sum which converges to $0$ as $N_1$ and $N_2$ converge to infinity. This completes the proof that with high probability $G_n$ is such that all sets $D$ for which $d_1\ge cd_2$ and $d_2\ge cd_1$ have boundary of size at least $\delta\left(d_1+d_2\right),$ which finishes the proof of the lemma. 
\end{proof}

\begin{proof}[Proof of Lemma~\ref{lemma:ConfigurationModelNotBipartite}]
We need to calculate the probability that the graph $G$ is such that the vertex set can be split into two sets $A$ and $B$ for which the proportion of edges between them is at least~$1-\delta$.
Let $N=\sum_{x\in V}\text{deg}(x)$ be the total degree of $G$. First consider a set $A$ such that the total degree of all vertices in $A$ is $d=\sum_{x\in A}\text{deg}(x)\leq N/2$. We have \begin{align*} &\pr{\frac{|\{x\sim y, (x,y)\in (A\times B)\cup (B\times A)\}|}{N}>1-\delta}\\&=\sum_{s>(1-\delta)\frac{N}{2}}\pr{|\{x\sim y, (x,y)\in (A\times B)\cup (B\times A)\}|=2s}.\end{align*}
The terms with $s\leq d$ in the sum above are the only non-zero terms. For $s<d$ we have that
\begin{align*}&\pr{|\{x\sim y, (x,y)\in A\times B\cup B\times A\}|=2s}= \frac{{d \choose s} {N-d \choose s} s!(d-s-1)!!(N-d-s-1)!!}{(N-1)!!}
\\&\lesssim \frac{d!(N-d)!\sqrt{(d-s)!(N-d-s)!}(d-s)^{-1/4}(N-d-s)^{-1/4}}{s!(d-s)!(N-d-s)!\sqrt{N!} N^{-1/4}}
\\&=\left(\frac{N}{(d-s)(N-d-s)}\right)^{1/4}\sqrt{\frac{{d\choose s}{N-d\choose s}}{{N \choose d}}}
\end{align*}

Indeed, this holds as by  Stirling's formula, $m!! \asymp\left(\sqrt{m!}\right)m^{\frac{1}{4}}$, and  $m!!\asymp m^{\frac{1}{2}}(m-1)!!.$
For $s=d<N/2$ the probability above is equal to $\frac{{N-d \choose d}d!(N-2d-1)!!}{(N-1)!!}\lesssim \left(\frac{N}{N-2d}\right)^{1/4}\sqrt{\frac{{N-d\choose d}}{{N \choose d}}}$ and for $s=d=N/2$ this is $\lesssim N^{1/4}\sqrt{\frac{1}{{N \choose N/2}}}.$ So for fixed $d<N/2$, summing over $s$ the desired bound is 

 \begin{align*} &\left(\frac{N}{N-2d}\right)^{1/4}\sqrt{\frac{{N-d\choose d}}{{N \choose d}}}+\sum_{d> s>(1-\delta)\frac{N}{2}}\left(\frac{N}{(d-s)(N-d-s)}\right)^{1/4}\sqrt{\frac{{d\choose s}{N-d\choose s}}{{N \choose d}}}
 \\& \lesssim \left(\frac{N}{N-2d}\right)^{1/4} \delta \frac{N}{2} \sqrt{\frac{{d\choose (1-\delta)\frac{N}{2}}{N-d\choose (1-\delta)\frac{N}{2}}}{{N \choose d}}},
 \end{align*}
as it is easy to check that for small $\delta$ the binomial coefficients are maximised when $s$ takes its minimal allowed value. We can now sum over $d$ and all sets $A$ with $d$ edges (there are at most ${N/3\choose d/3}$ by LYM's inequality and the fact that sets of $d$ edges make an anti-chain) we get that we just need to bound 
\[ \sum_{d=(1-\delta)\frac{N}{2}}^{\frac{N}{2}}{\frac{N}{3}\choose \frac{d}{3}}\left({N}\right)^{1/4} \delta \frac{N}{2} \sqrt{\frac{{d\choose (1-\delta)\frac{N}{2}}{N-d\choose (1-\delta)\frac{N}{2}}}{{N \choose d}}}\leq  {\frac{N}{3}\choose \frac{N}{6}}\left({N}\right)^{1/4} (\delta \frac{N}{2})^2 \sqrt{\frac{{\frac{N}{2}\choose (1-\delta)\frac{N}{2}}{(1+\delta)\frac{N}{2}\choose (1-\delta)\frac{N}{2}}}{{N \choose \frac{N}{2}}}}.
\]
Using the approximation ${n\choose k}\asymp \sqrt{\frac{n}{k\left(n-k\right)}}\exp\left(nH\left(\frac{k}{n}\right)\right)$ where $H\left(x\right)=-x\log\left(x\right)-\left(1-x\right)\log\left(1-x\right)$ we get that the above is bounded by 
\begin{align*}  &N^4\exp\left(\frac{N}{3}H\left(\frac{1}{2}\right)+\frac{N}{4}H(1-\delta)+(1+\delta)\frac{N}{4}H\left(\frac{1-\delta}{1+\delta}\right)-\frac{N}{2}H\left(\frac{1}{2}\right)\right)
\\& =N^4\exp\left(\frac{N}{3}H\left(\frac{1}{2}\right)+\frac{N}{4}H(\delta)+(1+\delta)\frac{N}{4}H\left(\frac{2\delta}{1+\delta}\right)-\frac{N}{2}H\left(\frac{1}{2}\right)\right)
\\& \leq N^4\exp\left(-\frac{N}{6} \log(2)+\frac{N}{4}\cdot\frac{1}{100}\right)\leq N^4\exp\left(-\frac{N}{12}\right),
\end{align*}
since $H(\delta)$ and $H(\frac{2\delta}{1+\delta})$ can be made as small as we want by picking $\delta$ sufficiently small, as $H(x) \to 0$ as $x\to 0$. The above expression tends to $0$ when $N\to \infty$, and hence this completes the proof. 
 \end{proof}

\begin{proof}[Proof of Lemma~\ref{rmk:2ComNotBipartite}] 
Let $c$ be a small constant to be determined. Consider a set $A$ with $d_1$ and $d_2$ half edges in communities $1$ and $2$, respectively. We first deal with the sets for which $cd_1>d_2$ or $cd_2>d_1$.  Without loss of generality assume that $cd_1>d_2$. As $d_2/c<d_1\le d_1+d_2\le N/2\le CN_2/2$, we get that for every $c'$ there is a small enough $c$ such that $d_2<c'N_2/2$. This implies that on the rewired graph in community $2$ we have that the boundary of $A^c\cap V_2$ is at most $c'N_2/2$ and therefore there are at least $(1-c')N_2$ edges with both ends in $A$ or $A^c$. This implies that after rewiring back there are still at least $(1-c'-\alpha_2)N_2$ edges inside $A^c\cap V_2$ or $A\cap V_2$. If $1-c'-\alpha_2>0$ then we can choose $\delta$ to be small enough such that $(1-c'-\alpha_2)N_2\ge \delta N$ and the proof follows for such $\delta$. As~$c'$ is an arbitrary small constant, we only need to show that the condition $\alpha\le 1$ gives us that $\alpha_2$ is bounded away from $1$. Indeed, this follows since $1\ge \alpha=\alpha_1+\alpha_2=\frac{p}{N_1}+\frac{p}{N_2}\ge (1+\frac{1}{C})\frac{p}{N_2}$. Therefore for a small enough $\delta$ we know that there is a small constant $c$ such that with high probability $G_n$ is such that when $cd_1>d_2$ or $cd_2>d_1$ hold, then there are at least $\delta N$ edges with both ends in $A$ or $A^c$.

We now turn to the case when for some small constant $c$ we know that $cd_1\le d_2$ and $cd_2\le d_1$. We will use similar estimates as in the proof of Lemma~\ref{lemma:SmallSetsBoundary}. First, we see that here we need to bound just a slight modification of equation~\eqref{eq:BigSumExpansion} which gives the probability that the vertex set $A$ with $d_1$ half-edges in the first community and $d_2$ half-edges in the second one has a given size of the boundary (this is the last condition in the sum and only part of the sum which we need to change).  
As we want to bound the probability that the boundary is greater than $(1-\delta)N/2$, recalling the definition of $d^i_1, d_2^i, d_1^o, d_2^o, k_1, k_2$ and $\ell$ we see that we need to change the last condition in the sum to $k_1+k_2+d^o_1-\ell+d^o_2-\ell\ge (1-\delta)\frac{N}{2}$. In particular, we need to control
\begin{align}\label{eq:BigSumBipartite}\sum_{\substack{d^o_1\le d_1,\, d^o_2\le d_2\\ \ell\le \min\left\{d^o_1,d^o_2\right\}\\ k_1\le d^i_1,  k_2\le d^i_2\\k_1+k_2+d^o_1-\ell+d^o_2-\ell\ge (1-\delta)\frac{N}{2}}}\frac{(d^i_1)^{\frac{1}{4}}\left({d^i_1\choose k_1}{N_1-p-d^i_1\choose k_1}{N_1-p\choose d^i_1}\right)^{\frac{1}{2}}(d^i_2)^{\frac{1}{4}}\left({d^i_2\choose k_2}{N_2-p-d^i_2\choose k_2}{N_2-p\choose d^i_2}\right)^{\frac{1}{2}}p!}{{N_1\choose d_1}{N_2\choose d_2}\ell!\left(d^o_1-\ell\right)!\left(d^o_2-\ell\right)!\left(p-d^o_1-d^o_2+\ell\right)!}.\end{align}
Further, we need to take the sum of the above expression over all sets which have  $d_1$ half-edges in the first and $d_2$ in the second community satisfying $\frac{N}{2}\geq d_1+d_2\ge (1-\delta)\frac{N}{2}$. Notice that the sets of vertices with exactly $d_1$ half-edges in the first and $d_2$ in the second community can be chosen in at most $N^2{\lfloor \frac{N_1}{3}\rfloor\choose\lfloor \frac{d_1}{3}\rfloor }{\lfloor\frac{N_2}{3}\rfloor\choose \lfloor\frac{d_2}{3}\rfloor }$ ways as it makes an anti-chain and so we can apply LYM's inequality as before (the extra factor of $N^2$ comes because if $d_1>N_1/2$ we can count instead the number of ways of choosing $D^c$ in $V_1$ of size $N_1-d_1$ and then use that ${a\choose b}={a\choose a-b}$ and the extra $N$ factor comes from taking a ceiling instead of a floor). Therefore we need to bound the sum of
  \[N^2{\lfloor \frac{N_1}{3}\rfloor\choose\lfloor \frac{d_1}{3}\rfloor }{\lfloor\frac{N_2}{3}\rfloor\choose \lfloor\frac{d_2}{3}\rfloor }\frac{(d^i_1)^{\frac{1}{4}}\left({d^i_1\choose k_1}{N_1-p-d^i_1\choose k_1}{N_1-p\choose d^i_1}\right)^{\frac{1}{2}}(d^i_2)^{\frac{1}{4}}\left({d^i_2\choose k_2}{N_2-p-d^i_2\choose k_2}{N_2-p\choose d^i_2}\right)^{\frac{1}{2}}p!}{{N_1\choose d_1}{N_2\choose d_2}\ell!\left(d^o_1-\ell\right)!\left(d^o_2-\ell\right)!\left(p-d^o_1-d^o_2+\ell\right)!}
  \] 
  over $k_1+k_2+d^o_1-\ell+d^o_2-\ell>(1-\delta) \frac{N}{2}$. In exactly the same way as in the proof of Lemma~\ref{lemma:SmallSetsBoundary}, the sum above can be bounded by a constant power of $N$ and an exponential part which is bounded by 
\begin{align*}&\leq \exp\left(\frac{1}{2}\left(N_1H\left(\frac{2k_1+d^o_1-\ell+d^o_2-\ell}{N_1}\right)+N_2H\left(\frac{2k_2+d^o_1-\ell+d^o_2-\ell}{N_2}\right)\right)\right.\\&\left.-\frac{N_1}{6}H\left(\frac{d_1}{N_1}\right)-\frac{N_2}{6}H\left(\frac{d_2}{N_2}\right)\right).\end{align*} 
Since we have that for some constant $c$, $d_1\ge cd_2$ and $d_2\ge cd_1$, we claim that we can choose~$\delta$ such that for $\frac{N}{2}-k_1-k_2-d^o_1-d^o_2+2\ell \le \delta\frac{N}{2}$ we have that $H\left(\frac{N_1-(2k_1+d^o_1-\ell+d^o_2-\ell)}{N_1}\right) \leq \frac{1}{100}H\left(\frac{d_1}{N_1}\right)$ and $H\left(\frac{N_2-(2k_2+d^o_1-\ell+d^o_2-\ell)}{N_2}\right) \leq \frac{1}{100}H\left(\frac{d_2}{N_2}\right).$ Indeed, this would hold if ${N_1-(2k_1+d^o_1-\ell+d^o_2-\ell)}\le \delta'd_1$ for some small enough $\delta'$ and similarly for the expression involving $k_2$.
It is easy to see that $N/2-k_1-k_2=(N_1-2k_1+N_2-2k_2)/2$ and therefore ${N_1-(2k_1+d^o_1-\ell+d^o_2-\ell)}\le \delta N$. Using also $d_1(1+\frac{1}{c})\ge d_1+d_2>(1-\delta)N/2$ gives that we can choose small enough $\delta$ in terms of $c$ and $\delta'$ such that the required bound holds. Using the symmetry of the problem, the corresponding  bound by $\delta' d_2$ also holds. 
 
Using the above and that $H(x)=H(1-x)$, our expression is bounded from above by 
\[\exp\left(-\frac{N_1}{12}H\left(\frac{d_1}{N_1}\right)-\frac{N_2}{12}H\left(\frac{d_2}{N_2}\right)\right)\le \exp\left(-\frac{1}{12}\left(d_1\log\left(\frac{d_1}{N_1}\right)+d_2\log\left(\frac{d_2}{N_2}\right)\right)\right).\]

Adding the non-exponential parts and summing over all suitable $d^o_1,d^o_2,d^i_1,d^i_2, k_1,k_2$ and $\ell$, the obtained bound gives that the sum in~\eqref{eq:BigSumBipartite} multiplied by $N^2 {\lfloor \frac{N_1}{3}\rfloor\choose\lfloor \frac{d_1}{3}\rfloor}{\lfloor\frac{N_2}{3}\rfloor\choose \lfloor\frac{d_2}{3}\rfloor}$ is bounded from above by $$N^{10+\frac{1}{2}} \exp\left(-\frac{1}{12}\left(d_1\log\left(\frac{d_1}{N_1}\right)+d_2\log\left(\frac{d_2}{N_2}\right)\right)\right).$$ Summing over all possible $d_1\in (c_1N_1,c_2N_1)$ and $d_2\in (c_1N_2,c_2N_2)$ for suitable $c_1$ and $c_2$ depending on $c$ and $\delta$ gives that the sum converges to $0$ as $N_1$ and $N_2$ converge to infinity. This completes the proof that with high probability $G_n$ is such that all sets $A$ have boundary of size at most~$(1-\delta)\frac{N}{2}$.
\end{proof}

\section{Entropy in terms of the simple random walk on the tree}\label{AppendixEntropy}
We prove here the following lemma. 

In this proof we use ideas from~\cite{GWTSpeedHarmMeasure}.
\begin{proof}[Proof of Lemma~\ref{lemma:EntropyInTermsOfWalk}]
		
		Writing $\widetilde{\xi}$ for the loop erasure of $\widetilde{X}$, $\widetilde{\tau}_e$ for the first hitting time of the endpoint of $e$ furthest from the root by $\widetilde{X}$, $\widetilde{\tau}_e^{(2)}$ for the first return time to $e$ after time $\widetilde{\tau}_e$	
	and using the uniform drift from Lemma~\ref{lemma:ReturnProb}
	\[ 
	\prcond{e\in \widetilde{X}}{T}{} = 	\prcond{\widetilde{\tau}_e<\infty}{T}{} \asymp	\prcond{\widetilde{\tau}_e<\infty, \widetilde{\tau}_e^{(2)}=\infty}{T}{} = 	\prcond{e\in \widetilde{\xi}}{T}{}.
	\]
	
For $k\in \mathbb{N}$ let $\tau_k$ be the first time walk visits level $k$ and $\tau^+_k$ be the last time walk visits level $k$. We know that $X_{\tau_k^+}=\xi_k$ and above then implies using Proposition~\ref{prop:LERWLimit} that
\[
-\frac{1}{k} \log \prcond{X_{\tau_k^+}\in \widetilde{X}}{X, T}{} \to \mathfrak{h}.
\]
As by Lemma~\ref{lemma:Speed} we have that the speed of the walk converges to $\nu$ almost surely, which gives that almost surely $\frac{k}{\tau_k^+}\to \nu$ and therefore almost surely 
\[ -\frac{1}{\tau_k^+} \log \prcond{X_{\tau_k^+}\in \widetilde{X}}{X, T}{} \to\nu  \mathfrak{h}.
\]
Let $\ell(k)=\max\{\tau_i^+: \tau_i^+\le k\}$.  Notice that $k-\ell(k)\le \tau_{i+1}^+-\tau_{i}^+$ for some $i$ and that we can bound this difference by a random variable which is the bound on the time it takes on any given tree with degrees at least $3$ for the walk to visit level $1$ for the last time, conditional on never returning to the root. As this random variable has bounded expectation and variance, Chebyshev's inequality and the Borel Cantelli lemma give  that $k-\ell(k)=o(k)$ as $k\to \infty$ almost surely.

 Therefore, we get that
\[ -\frac{1}{k} \log \prcond{X_{\ell(k)}\in \widetilde{X}}{X, T}{} \to\nu  \mathfrak{h}.
\]
As we can extend a path which goes through $X_{\ell(k)}$ to the path visiting $X_k$, and vice versa, by just adding the path $X_{\ell(k)},X_{\ell(k)+1},\ldots X_k$ or $X_k, X_{k-1},\ldots, X_{\ell(k)}$, respectively, we have that
\begin{align*}&\prcond{X_{k}\in \widetilde{X}}{X, T}{}\ge \prcond{X_{\ell(k)}\in \widetilde{X}}{X, T}{}\prod_{j=\ell(k)}^k\frac{1}{\text{deg}\left(X_j\right)} \,\,\,\,\,\, {\rm       and}
\\&\prcond{X_{\ell(k)}\in \widetilde{X}}{X, T}{}\ge \prcond{X_{k}\in \widetilde{X}}{X, T}{}\prod_{j=\ell(k)}^k\frac{1}{\text{deg}\left(X_j\right)}.
\end{align*} 
As $k-\ell(k)=o(k)$ and  $\prod_{j=\ell(k)}^k{\text{deg}\left(X_j\right)}\le \Delta^{k-\ell(k)}$ we have that $\frac{1}{k}\log\left(\prod_{j=\ell(k)}^k{\text{deg}\left(X_j\right)}\right)\to 0$, hence the above gives that 
\begin{align} \label{eq:convxi} 
-\frac{1}{k} \log \prcond{X_{k}\in \widetilde{X}}{X, T}{} \to\nu  \mathfrak{h}.
\end{align}

We have that
\[
 \prcond{\widetilde{X}_k=X_k}{X, T}{} \leq \prcond{X_{k}\in \widetilde{X}}{X, T}{},
\]
hence by~\eqref{eq:convxi} we get almost surely as $k\to\infty$
\begin{align*}
-\frac{1}{k} \log  \prcond{\widetilde{X}_k=X_k}{X, T}{} \geq \nu\mathfrak{h}.
\end{align*}
We now notice that  \begin{align}\label{eq:XinXiLimit}\frac{-\log\prcond{X_k\in \widetilde{\xi}}{T,X}{}}{k}\to\nu \mathfrak{h}.\end{align}
Indeed, this follows from~\eqref{eq:convxi}, as \[c\cdot \prcond{X_k\in \widetilde{X}}{T,X}{}\le \prcond{X_k\in \widetilde{\xi}}{T,X}{}\le \prcond{X_k\in \widetilde{X}}{T,X}{},\] where $c$ is the constant from Lemma~\ref{lemma:ReturnProb} and where the lower bound holds as $c$ is the lower bound for the probability that~$\widetilde{X}$ never returns to the parent of $X_k$ after it's first visits to vertex $X_k$, which would imply that $X_k\in\widetilde{\xi}$. 

Following further the idea of~\cite{GWTSpeedHarmMeasure} we let $\gamma> \nu \mathfrak h$ and $\varepsilon>0$ such that $\varepsilon<(\gamma - \nu\mathfrak h)/2$ and we define a subset of undirected edges of the tree via
\[
B_k=\left\{ e: \|{e}\|\leq k, \prcond{\widetilde{X}_{k}=e}{ T}{}<e^{-(\gamma -\varepsilon)k} \ \text{ and } \ \prcond{e\in \widetilde{\xi}}{T}{}>e^{-(\nu\mathfrak h +\varepsilon)k} \right\},
\] where $\|e\|$ stands for the distance of this edge from the root. 
Then we have
\begin{align*}
k+1 \geq \sum_{\|e\|\le k} \prcond{e\in {\xi}}{T}{}\geq \sum_{e\in B_k} \prcond{{X}_{k}=e}{T}{}e^{(\gamma - \nu\mathfrak{ h} -2\varepsilon) k}  = \prcond{{X}_{k}\in B_k}{T}{} e^{(\gamma - \nu\mathfrak h -2\varepsilon) k}.
\end{align*}
Therefore, using the assumption on $\varepsilon$ we deduce that
\[
\sum_k \prcond{{X}_{k}\in B_k}{T}{}<\infty,
\]
and hence by the Borel-Cantelli Lemma, almost surely
\[
{X}_{k} \notin B_k \quad \text{eventually in } k.  
\]
So as~\eqref{eq:XinXiLimit} gives that for all $k$ large enough, $\prcond{{X}_{k}\in \widetilde{\xi}}{T,X}{}>e^{-(\nu\mathfrak h +\varepsilon) k}$,  which implies that for all sufficiently large $k$
\[
\prcond{\widetilde{X}_{k}= X_{k}}{X,T}{} \geq e^{-(\gamma - \varepsilon)k}.
\]
By taking logarithms of both sides and using the assumption on  $\gamma$ and $\varepsilon$ we finally conclude that almost surely as $k\to\infty$
\begin{align*}
-\frac{1}{k} \log \prcond{\widetilde{X}_k=X_k}{T, X}{} \leq \nu\mathfrak{h}
\end{align*}
and this finishes the proof.
\end{proof}

\section{Mixing time comparison for simple  and non-backtracking random walk on two communities model}\label{appendixNBRW}

\begin{proof}[Proof of Proposition~\ref{prop:nbrwmixesfaster}](comparison of $\mathfrak{h}_Y$ and $\mathfrak{h}_X$)
Recall that we only need to compare \[\mathfrak{h}_Y=\frac{1}{N}\sum_{x\in G_n}\log\left(\text{deg}(x)-1\right)\] and \[\mathfrak{h}_X=\lim_{t\to \infty}-\frac{1}{t}\log\left(\prcond{X_t=\widetilde{X}_t}{T,X}{}\right),\] where $\widetilde{X}$ and $X$ are independent simple random walks on $T$. Note that the limit above is an a.s.\ limit jointly over $X$ and $T$. Using the assumption on the maximum degree being bounded and the dominated convergence theorem, taking expectation over $X$ gives \[\mathfrak{h}_X=\lim_{t\to \infty}-\frac{1}{t}\sum_{x\in T}\prcond{\widetilde{X}_t=x}{T}{\rho}\log\left(\prcond{\widetilde{X}_t=x}{T}{\rho}\right).
\]
 Following the work of \cite{NBRWvsSRW} we show that  under the assumptions of Proposition~\ref{prop:nbrwmixesfaster} we have $\mathfrak{h}_X<\mathfrak{h}_Y$. 

For $t\in \mathbb{N}$ and  a simple random walk $X$ on $T$ we let \[h_t=\E{-\sum_{x\in T}\prcond{{X}_t=x}{T}{\rho}\log\left(\prcond{{X}_t=x}{T}{\rho}\right)}.\] We see that $\mathfrak{h}_X=\lim_{t\to \infty}\frac{h_t}{t}.$
We now notice that  \cite[Claim 3.1]{NBRWvsSRW} holds in our setting as  the stationarity of the environment was established in Remark~\ref{rmk:stationarity2types}. Therefore, in exactly the same way as in~\cite{NBRWvsSRW}, we only need to establish that $h_3-h_1<2\mathfrak{h}_Y$.

 For $x,y$ vertices in our tree notation we write $y< x$ to mean that $y$ is a child of $x$. With a slight abuse of notation we will write $d(x)$ for the degree of~$x$ and $d(x,y)$ for the graph distance between~$x$ and $y$. We now calculate $h_3$ by first looking at the sum over the vertices in the third level of the tree. Let \begin{align*}R_3&=-\sum_{x\in T, d(x,\rho)=3}\prcond{X_3= x}{T}{\rho}\log(\prcond{X_3= x}{T}{\rho})\\&=\sum_{y<x<\rho}\frac{(d(y)-1)\left(\log(d(y))+\log(d(x))+\log(d(\rho))\right)}{d(\rho)d(x)d(y)}.\end{align*} 
 Define for $i\in \{1,2\}$ \[\beta_i=\estart{\frac{d(\rho)-1}{d(\rho)}\log(d(\rho))}{\pi_i}\text{  and  } \gamma_i=\estart{\frac{d(\rho)-1}{d(\rho)}}{\pi_i}\] where $\pi_i$ means that the root is chosen from $\pi_1$, i.e.\ it corresponds to a vertex of type $1$ with degree~$d$ with probability $d\sum_{v \in V_1}\mathds{1}(d(v)=d)/N_1.$
  We will also abuse notation and write $x$ in the index to represent the type of $x$ and ${3-x}$ for the opposite type. We recall that $\alpha_i=p/N_i$ is the ratio of outgoing edges from community $i$.
 We have that
\begin{align*} &\econd{R_3}{\B_{2}(\rho)}\\&= \sum_{x<\rho}\frac{1}{d(\rho)d(x)}\econd{\sum_{y<x}\frac{(d(y)-1)\left(\log(d(y))+\log(d(x))+\log(d(\rho))\right)}{d(y)}}{\B_{2}(\rho)}.
\end{align*}
As a vertex $y<x$ has the same type as $x$ with probability $1-\alpha_x$ and the opposite type with probability $\alpha_x$, the above conditional expectation is further equal to \begin{align*} &\sum_{x<\rho}\frac{d(x)-1}{d(\rho)d(x)}[\alpha_x\left(\beta_{3-x}+(\log d(x)+\log d(\rho))\gamma_{3-x}\right)\\& +(1-\alpha_x)\left(\beta_x+(\log d(x)+\log d(\rho))\gamma_x\right)].
\end{align*} Therefore using the definition of $\beta_x$ and $\gamma_x$ it can easily be seen that
\begin{align*} &\econd{R_3}{\theta(\rho), d(\rho)}=\econd{\econd{R_3}{\B_2(\rho)}}{\Theta(\rho),d(\rho)}
\\&=(1-\alpha_\rho)\left(\left(\alpha_\rho\beta_{3-\rho}+(1-\alpha_\rho)\beta_\rho\right)\gamma_\rho+(\alpha_\rho\gamma_{3-\rho}+(1-\alpha_\rho)\gamma_\rho)\beta_\rho
\right)
\\&+\alpha_\rho\left(\left(\alpha_{3-\rho}\beta_{\rho}+(1-\alpha_{3-\rho})\beta_{3-\rho}\right)\gamma_{3-\rho}+\left(\alpha_{3-\rho}\gamma_{\rho}+(1-\alpha_{3-\rho})\gamma_{3-\rho}\right)\beta_{3-\rho}\right)
\\&+\log(d(\rho))(1-\alpha_\rho)\gamma_\rho\left(\alpha_\rho\gamma_{3-\rho}+(1-\alpha_\rho)\gamma_\rho\right)\\&+\log(d(\rho))\alpha_\rho\gamma_{3-\rho}\left(\alpha_{3-\rho}\gamma_{\rho}+(1-\alpha_{3-\rho})\gamma_{3-\rho}\right).
\end{align*}
We let $c_1=\pi(V_1)$ and $c_2=\pi(V_2)$ and we obtain after plugging in the above, rearranging and  using $c_1\alpha_1=\frac{c_1p}{N_1}=\frac{p}{N}=\alpha=c_2\alpha_2$ and $c_1(1-\alpha_1)=c_1-\alpha$ (abusing notation and writing $\alpha$ for $p/N$ instead of the previous definition as they are up to constants the same) that 
\begin{align*} \E{R_3}&=2\alpha\beta_{2}\gamma_1+2(c_1-\alpha)\beta_1\gamma_1+2\alpha\gamma_{2}\beta_1+2(c_2-\alpha)\beta_{2}\gamma_{2}
\\&+c_1\estart{\log(d(\rho))}{\pi_1}((1-\alpha_1)\gamma_1\left(\alpha_1\gamma_{2}+(1-\alpha_1)\gamma_1\right)+\alpha_1\gamma_{2}\left(\alpha_{2}\gamma_{1}+(1-\alpha_{2})\gamma_{2}\right))
\\&+c_2\estart{\log(d(\rho))}{\pi_2}\left((1-\alpha_2)\gamma_2\left(\alpha_2\gamma_{1}+(1-\alpha_2)\gamma_2\right)+\alpha_2\gamma_{1}\left(\alpha_{1}\gamma_{2}+(1-\alpha_{1})\gamma_{1}\right)\right).
\end{align*}

We now calculate $\E{R_1}$  where \begin{align*}R_1&=-\sum_{x\in T, d(x,\rho)=1}\prcond{X_3= x}{T}{\rho}\log(\prcond{X_3= x}{T}{\rho})\\&=\sum_{x<\rho}\frac{1}{d(\rho)}\left(\sum_{y<x}\frac{1}{d(x)d(y)}+\sum_{x'<\rho}\frac{1}{d(x')d(\rho)}\right)\\&\cdot\left(\log{d(\rho)}-\log\left(\sum_{y<x}\frac{1}{d(x)d(y)}+\sum_{x'<\rho}\frac{1}{d(x')d(\rho)}\right)\right).\end{align*} 
By convexity of $x\to x\log x$ and Jensen's inequality we get \begin{align*}&\econd{R_1}{d(\rho),\theta(\rho)}\\&\le -\sum_{x<\rho}\econd{\prcond{X_3=x}{T}{\rho}}{d(\rho),\theta(\rho)}\log\left(\econd{\prcond{X_3=x}{T}{\rho}}{d(\rho),\theta(\rho)}\right).\end{align*}
We have that \begin{align*}&\econd{\prcond{X_3=x}{T}{\rho}}{\B_2(\rho)}\\&=\frac{1}{d(\rho)}\left(\frac{d(x)-1}{d(x)}((1-\alpha_x)(1-\gamma_x)+\alpha_x(1-\gamma_{3-x}))+\sum_{x'<\rho}\frac{1}{d(\rho)d(x')}\right). \end{align*} Therefore, by the tower property of conditional expectation 
\begin{align*}\econd{\prcond{X_3=x}{T}{\rho}}{d(\rho),\theta(\rho)}&=\frac{1}{d(\rho)}((1-\alpha_\rho)(1-\gamma_\rho)+\alpha_\rho (1-\gamma_{3-\rho}))
\\&+\frac{1}{d(\rho)}(1-\alpha_\rho)\gamma_\rho((1-\alpha_\rho)(1-\gamma_\rho)+\alpha_\rho(1-\gamma_{3-\rho}))
\\&+\frac{1}{d(\rho)}\alpha_\rho\gamma_{3-\rho}((1-\alpha_{3-\rho})(1-\gamma_{3-\rho})+\alpha_{3-\rho}(1-\gamma_{\rho})).
\end{align*} 
We label $\phi_\rho=d(\rho)\econd{\prcond{X_3=x}{T}{\rho}}{d(\rho),\theta(\rho)}$ and notice that the expression only depends on the type of $\rho$. As the sum over all $x<\rho$ cancels with $1/d(\rho)$ we get \[\E{R_1}\le c_1\phi_1(\estart{\log(d(\rho))}{\pi_1}-\log(\phi_1))+c_2\phi_2(\estart{\log(d(\rho))}{\pi_2}-\log(\phi_2)).\] We first notice that the terms in $\E{R_1}$ and $\E{R_3}$ being multiplied by $c_i\estart{\log(d(\rho))}{\pi_i}$ for $i \in \{1,2\}$ add up to $1$. Indeed, for $i\in \{1,2\}$
\begin{align*}&\phi_i+\left((1-\alpha_i)\gamma_i\left(\alpha_i\gamma_{3-i}+(1-\alpha_i)\gamma_i\right)+\alpha_i\gamma_{3-i}\left(\alpha_{3-i}\gamma_{i}+(1-\alpha_{3-i})\gamma_{3-i}\right)\right)
\\&=((1-\alpha_i)(1-\gamma_i)+\alpha_i (1-\gamma_{3-i}))
+(1-\alpha_i)\gamma_i((1-\alpha_i)(1-\gamma_i)+\alpha_i(1-\gamma_{3-i}))
\\&+\alpha_i\gamma_{3-i}((1-\alpha_{3-i})(1-\gamma_{3-i})+\alpha_{3-i}(1-\gamma_{i}))\\&+\left((1-\alpha_i)\gamma_i\left(\alpha_i\gamma_{3-i}+(1-\alpha_i)\gamma_i\right)+\alpha_i\gamma_{3-i}\left(\alpha_{3-i}\gamma_{i}+(1-\alpha_{3-i})\gamma_{3-i}\right)\right)
\\&=((1-\alpha_i)(1-\gamma_i)+\alpha_i (1-\gamma_{3-i}))
\\&+(1-\alpha_i)\gamma_i((1-\alpha_i)(1-\gamma_i)+\alpha_i(1-\gamma_{3-i})+\alpha_i\gamma_{3-i}+(1-\alpha_{i})\gamma_{i})
\\&+\alpha_i\gamma_{3-i}((1-\alpha_{3-i})(1-\gamma_{3-i})+\alpha_{3-i}(1-\gamma_{i})+\alpha_{3-i}\gamma_{i}+(1-\alpha_{3-i})\gamma_{3-i})
\\&= ((1-\alpha_i)(1-\gamma_i)+\alpha_i (1-\gamma_{3-i}))+(1-\alpha_i)\gamma_i+\alpha_i\gamma_{3-i}=1.\end{align*}

It is easy to check that \[h_1=c_1\estart{\log(d(\rho))}{\pi_1}+c_2\estart{\log(d(\rho))}{\pi_2},\] 
and so we get that 
\begin{align*}&h_3-h_1\\&\le2\alpha\beta_{2}\gamma_1+2(c_1-\alpha)\beta_1\gamma_1+2\alpha\gamma_{2}\beta_1+2(c_2-\alpha)\beta_{2}\gamma_{2}-c_1\phi_1\log(\phi_1)-c_2\phi_2\log(\phi_2). 
 \end{align*}
 We first consider the case when $\alpha<c$ where $c$ will be taken to be a sufficiently small constant. 
 In this case, we rewrite the last expression as
 \[ 2\alpha(\beta_{2}\gamma_1-\beta_1\gamma_1+\gamma_2\beta_1 -\beta_2\gamma_2)+2c_1\beta_1\gamma_1+2c_2\beta_{2}\gamma_{2}-c_1\phi_1\log(\phi_1)-c_2\phi_2\log(\phi_2). 
\]
Notice that $\beta_i\leq \log \Delta$ and $\frac{2}{3}\le \gamma_i\le\frac{\Delta-1}{\Delta}\le 1$. Also $\phi_i=1-\gamma_i^2+\alpha_i\widetilde{\phi}_i\le 1$ where $\widetilde{\phi_i}$ is a function which is bounded by $\pm 2$. 
Indeed, we have \begin{align*}\phi_i&=1-\left[(1-\alpha_i)\gamma_i\left(\alpha_i\gamma_{3-i}+(1-\alpha_i)\gamma_i\right)+\alpha_i\gamma_{3-i}\left(\alpha_{3-i}\gamma_{i}+(1-\alpha_{3-i})\gamma_{3-i}\right)\right]
\\&=1-\gamma_i^2+2\gamma_i\alpha_i-\alpha_i^2\gamma_i^2-\gamma_i\gamma_{3-i}\alpha_i(1-\alpha_i)-\alpha_i\gamma_{3-i}(\alpha_{3-i}\gamma_i+(1-\alpha_{3-i})\gamma_{3-i})
\\&=1-\gamma_i^2+\gamma_i\alpha_i(2-[\alpha_i\gamma_i+\gamma_{3-i}(1-\alpha_i)])-\alpha_i\gamma_{3-i}[\alpha_{3-i}\gamma_i+(1-\alpha_{3-i})\gamma_{3-i}],
\end{align*} and using that $\gamma_i,\gamma_{3-i}\le 1$ we obtain that both  $[\alpha_i\gamma_i+\gamma_{3-i}(1-\alpha_i)]$ and $[\alpha_{3-i}\gamma_i+(1-\alpha_{3-i})\gamma_{3-i}]$ are positive and bounded above by $1$, and thus giving the desired expression for $\phi_i$. If $\widetilde{\phi}_i\ge 0$ then $-\log(\phi_i)\le -\log(1-\gamma_i^2)$ so $-\phi_i\log(\phi_i)\le -2\alpha_i\log(\phi_i)-(1-\gamma_i^2)\log(1-\gamma_i^2)$. If $\widetilde{\phi}_i<0$ we have from the mean value theorem applied to the $\log$ function that there exists  $x\in [1-\gamma_i^2-|\widetilde{\phi}_i|\alpha_i,1-\gamma_i^2]$ such that
\[\frac{-\log(1-\gamma_i^2-|\widetilde{\phi}_i|\alpha_i)+\log(1-\gamma_i^2)}{|\widetilde{\phi}_i|\alpha_i}=\frac{1}{x}\implies -\log({\phi}_i)\le -\log(1-\gamma_i^2)+\frac{|\widetilde{\phi}_i|\alpha_i}{\phi_i}.\]
This gives that $-\phi_i\log(\phi_i)\le -(1-\gamma_i^2)\log(1-\gamma_i^2)+2\alpha_i$ Therefore, in both cases we have $-\phi_i\log(\phi_i)\le -2\alpha_i(\log(\phi_i)-1)-(1-\gamma_i^2)\log(1-\gamma_i^2).$ We can also bound $-\log(\phi_i)$ by a constant which only depends on $\Delta$.
As  $\mathfrak{h}_Y=c_1\estart{\log(d(\rho)-1)}{\pi_1}+c_2\estart{\log(d(\rho)-1)}{\pi_2},$ it would be enough to show that  for $i\in \{1,2\}$ we have \[2\beta_i\gamma_i-(1-\gamma_i^2)\log(1- \gamma_i^2)-2\estart{\log(d(\rho)-1)}{\pi_i}<c'\] where $c'<0$ depends on $\Delta$. This is because, as $c_1+c_2=1$ and $c_i\alpha_i=\alpha$, we would have
\[h_3-h_1-2\mathfrak{h}_Y\le c'+\alpha(2\beta_{2}\gamma_1-2\beta_1\gamma_1+2\gamma_2\beta_1 -2\beta_2\gamma_2-\log(\phi_1)-\log(\phi_2)+4)\le c'+\alpha C'\] where $C'$ is some large constant also depending only on $\Delta$. Therefore, there is a small constant $c$ such that for $\alpha\le c$ the above bound is always negative. The existence of the constant  $c'$ follows directly from the proof of \cite{NBRWvsSRW} as their proof gives that $h_3-h_1-2\mathfrak{h}_Y$ can be bounded by a negative constant when all the degrees of the graph are between $3$ and $\Delta$.

We now get back to the case where there is no bound on $\alpha$ but we have an assumption that the average degree is the same in both communities. 
By applying Jensen's inequality again to the function $x\to x\log x$ and to the random variable which has probability $c_1$ to be $\phi_1$ and $c_2=1-c_1$ to be $\phi_2$ we have that  \[-c_1\phi_1\log(\phi_1)-c_2\phi_2\log(\phi_2)\le -(c_1\phi_1+c_2\phi_2)\log(c_1\phi_1+c_2\phi_2).\]
Therefore, we have that \begin{align*}&h_3-h_1\\&\le2\alpha\beta_{2}\gamma_1+2(c_1-\alpha)\beta_1\gamma_1+2\alpha\gamma_{2}\beta_1+2(c_2-\alpha)\beta_{2}\gamma_{2}-(c_1\phi_1+c_2\phi_2)\log(c_1\phi_1+c_2\phi_2). 
 \end{align*}
 Using that $c_i\alpha_i=\alpha$ we have that  \begin{align*}c_1\phi_1+c_2\phi_2&=1-c_1\gamma_1(\alpha_1\gamma_2+(1-\alpha_1)\gamma_1)-c_2\gamma_2(\alpha_2\gamma_1+(1-\alpha_2)\gamma_2)
 \\&=1-\estart{\frac{(d(x)-1)(d(\rho)-1)}{d(x)d(\rho)}}{\pi}
\end{align*}
 where $\rho$ is a root (chosen according to $\pi$) and $x$ is a neighbour of it and the last equality holds, since if the type of $\rho$ is $i$, which happens with probability $c_i$, then the type of $x$ is $i$ with probability $1-\alpha_i$ and $3-i$ with probability $\alpha_i$, and once the types have been decided the degrees are independent. Similarly, we have that
 \begin{align*}&\alpha(\beta_1\gamma_2+\beta_2\gamma_1)+(c_1-\alpha)\beta_1\gamma_1+(c_2-\alpha)\beta_2\gamma_2 \\&=\estart{\frac{(d(\rho)-1)\log(d(\rho))}{d(\rho)}}{\pi}-\estart{\frac{(d(\rho)-1)\log(d(\rho))}{d(\rho)d(x)}}{\pi}
 \end{align*}
 This gives that \begin{align*}h_3-h_1&\le2\estart{\frac{(d(\rho)-1)\log(d(\rho))}{d(\rho)}}{\pi}-2\estart{\frac{(d(\rho)-1)\log(d(\rho))}{d(\rho)d(x)}}{\pi}\\&-\left(1-\estart{\frac{(d(x)-1)(d(\rho)-1)}{d(x)d(\rho)}}{\pi}\right)\log\left(1-\estart{\frac{(d(x)-1)(d(\rho)-1)}{d(x)d(\rho)}}{\pi}\right). \end{align*}
We also notice that \[\estart{\frac{(d(x)-1)(d(\rho)-1)}{d(x)d(\rho)}}{\pi}\le \estart{\frac{(d(\rho)-1)^2}{d(\rho)^2}}{\pi},\] since this is equivalent to 
\begin{align*}&c_1(1-\alpha_1)\gamma_1^2+(c_1\alpha_1+c_2\alpha_2)\gamma_1\gamma_2+c_2(1-\alpha_2)\gamma_2^2\\&\le c_1 \estart{\frac{(d(\rho)-1)^2}{d(\rho)^2}}{\pi_1}+c_2 \estart{\frac{(d(\rho)-1)^2}{d(\rho)^2}}{\pi_2}\end{align*}
which holds as $\gamma_i^2\le \estart{\frac{(d(\rho)-1)^2}{d(\rho)^2}}{\pi_i}$ and
$2\alpha \gamma_1\gamma_2\le \alpha \estart{\frac{(d(\rho)-1)^2}{d(\rho)^2}}{\pi_1}+\alpha  \estart{\frac{(d(\rho)-1)^2}{d(\rho)^2}}{\pi_2}.$ Indeed, this last inequality is true, because if $Z_i\sim \pi_i$ are two independent random variables, then 
\[  \estart{\frac{(d(\rho)-1)^2}{d(\rho)^2}}{\pi_1}-\gamma_1\gamma_2+  \estart{\frac{(d(\rho)-1)^2}{d(\rho)^2}}{\pi_2}= \E{\left(\frac{Z_1-1}{Z_1}-\frac{Z_2-1}{Z_2}\right)^2}>0.\]
This with Jensen's inequality applied to $\log(x)$ gives that 
 \begin{align*}h_3-h_1&\le2\estart{\frac{(d(\rho)-1)\log(d(\rho))}{d(\rho)}}{\pi}-2\estart{\frac{(d(\rho)-1)\log(d(\rho))}{d(\rho)d(x)}}{\pi}\\&-\left(1-\estart{\frac{(d(x)-1)(d(\rho)-1)}{d(x)d(\rho)}}{\pi}\right)\estart{\log\left(1-\frac{(d(\rho)-1)^2}{d(\rho)^2}\right)}{\pi}. \end{align*}
 Using that the distribution of $d(x)$ and $d(\rho)$ is the same when $\rho$ starts from $\pi$ we have  \begin{align*}\estart{\frac{(d(x)-1)(d(\rho)-1)}{d(x)d(\rho)}}{\pi}&=1-\estart{\frac{2d(\rho)-1}{d(x)d(\rho)}}{\pi}.\end{align*} 
 This implies that 
  \begin{align*}h_3-h_1&\le2\estart{\frac{(d(\rho)-1)\log(d(\rho))}{d(\rho)}}{\pi}-2\estart{\frac{(d(\rho)-1)\log(d(\rho))}{d(\rho)d(x)}}{\pi}\\&-\estart{\frac{2d(\rho)-1}{d(x)d(\rho)}}{\pi}\estart{\log\left(\frac{(2d(\rho)-1)}{d(\rho)^2}\right)}{\pi}. \end{align*} 
We notice that $\estart{\frac{1}{d(x)}}{x\sim\pi_1}=\sum_{x\in V_1}\frac{d(x)}{N_1}\frac{1}{d(x)}=\frac{n_1}{N_1}=\frac{n_1}{\sum_{x\in V_1}d(x)}=\frac{n_2}{\sum_{x\in V_2}d(x)}=\estart{\frac{1}{d(x)}}{x\sim\pi_2}$ by the assumption that the average degree is equal. As $d(x)$ and $d(\rho)$ are independent if the types of $x$ and $\rho$ are given, we now have
  \begin{align*}h_3-h_1&\le2\estart{\frac{(d(\rho)-1)\log(d(\rho))}{d(\rho)}}{\pi}-2\estart{\frac{1}{d(x)}}{\pi}\estart{\frac{(d(\rho)-1)\log(d(\rho))}{d(\rho)}}{\pi}\\&-\estart{\frac{1}{d(x)}}{\pi}\estart{\frac{2d(\rho)-1}{d(\rho)}}{\pi}\estart{\log\left(\frac{(2d(\rho)-1)}{d(\rho)^2}\right)}{\pi}  
  \\&=2\estart{\frac{(d(\rho)-1)\log(d(\rho))}{d(\rho)}}{\pi}-2\estart{\frac{1}{d(\rho)}}{\pi}\estart{\frac{(d(\rho)-1)\log(d(\rho))}{d(\rho)}}{\pi}\\&-\estart{\frac{1}{d(\rho)}}{\pi}\estart{\frac{2d(\rho)-1}{d(\rho)}}{\pi}\estart{\log\left(\frac{(2d(\rho)-1)}{d(\rho)^2}\right)}{\pi}, 
  \end{align*}
where the last equality holds, since $d(x)$ and $d(\rho)$ have the same distribution when starting from~$\pi$. The proof that the above is smaller than $\estart{\log(d(\rho)-1)}{\pi}$ is given in~\cite{NBRWvsSRW} and this completes the proof of the proposition.
\end{proof}
\end{appendix}

\bibliography{communities_arxiv}

\begin{thebibliography}{10}

\bibitem{aldous-fill-2014}
David Aldous and James~Allen Fill.
\newblock Reversible {M}arkov {C}hains and {R}andom {W}alks on {G}raphs, 2002.
\newblock \emph{Available at}
  http://www.stat.berkeley.edu/$\sim$aldous/RWG/book.html.

\bibitem{probm}
Noga Alon and Joel~H. Spencer.
\newblock {\em The probabilistic method}.
\newblock Wiley-Interscience Series in Discrete Mathematics and Optimization.
  John Wiley \& Sons, Inc., Hoboken, NJ, third edition, 2008.

\bibitem{CharacterizationCutoff}
Riddhipratim Basu, Jonathan Hermon, and Yuval Peres.
\newblock Characterization of cutoff for reversible {M}arkov chains.
\newblock {\em Ann. Probab.}, 45(3):1448--1487, 2017.

\bibitem{AnnasPaper}
Anna Ben-Hamou.
\newblock A threshold for cutoff in two-community random graphs.
\newblock {\em Ann. Appl. Probab.}, 30(4):1824--1846, 2020.

\bibitem{NBRWvsSRW}
Anna Ben-Hamou, Eyal Lubetzky, and Yuval Peres.
\newblock Comparing mixing times on sparse random graphs.
\newblock {\em Annales de l'Institut Henri Poincar{\'e}, Probabilit{\'e}s et
  Statistiques}, 55, 07 2017.

\bibitem{RWonRG}
Nathana\"{e}l Berestycki, Eyal Lubetzky, Yuval Peres, and Allan Sly.
\newblock Random walks on the random graph.
\newblock {\em Ann. Probab.}, 46(1):456--490, 2018.

\bibitem{CaputoSalzeBordenave}
Charles Bordenave, Pietro Caputo, and Justin Salez.
\newblock Cutoff at the ``entropic time'' for sparse {M}arkov chains.
\newblock {\em Probab. Theory Related Fields}, 173(1-2):261--292, 2019.

\bibitem{BordenaveLacoin}
Charles Bordenave and Hubert Lacoin.
\newblock Cutoff at the entropic time for random walks on covered expander
  graphs.
\newblock {\em J. Inst. Math. Jussieu}, 21(5):1571--1616, 2022.

\bibitem{DiaconisChatterjee}
Sourav Chatterjee and Persi Diaconis.
\newblock Correction to: {S}peeding up {M}arkov chains with deterministic
  jumps.
\newblock {\em Probab. Theory Related Fields}, 181(1-3):377--400, 2021.

\bibitem{ChenSaloffCoste}
Guan-Yu Chen and Laurent Saloff-Coste.
\newblock Comparison of cutoffs between lazy walks and {M}arkovian semigroups.
\newblock {\em J. Appl. Probab.}, 50(4):943--959, 2013.

\bibitem{randomlifts}
Guillaume Conchon-Kerjan.
\newblock Cutoff for random lifts of weighted graphs.
\newblock {\em Ann. Probab.}, 50(1):304--338, 2022.

\bibitem{expander}
Jian Ding, Jeong~Han Kim, Eyal Lubetzky, and Yuval Peres.
\newblock Anatomy of a young giant component in the random graph.
\newblock {\em Random Structures \& Algorithms}, 39(2):139--178, 2011.

\bibitem{Dingperessensitivity}
Jian Ding and Yuval Peres.
\newblock Sensitivity of mixing times.
\newblock {\em Electron. Commun. Probab.}, 18:no. 88, 6, 2013.

\bibitem{VarjuEberhard}
Sean Eberhard and P\'{e}ter~P. Varj\'{u}.
\newblock Mixing time of the {C}hung-{D}iaconis-{G}raham random process.
\newblock {\em Probab. Theory Related Fields}, 179(1-2):317--344, 2021.

\bibitem{SpectralProfilePaper}
Sharad Goel, Ravi Montenegro, and Prasad Tetali.
\newblock Mixing time bounds via the spectral profile.
\newblock {\em Electron. J. Probab.}, 11:no. 1, 1--26, 2006.

\bibitem{IneqHit}
Simon Griffiths, Ross~J. Kang, Roberto~Imbuzeiro Oliveira, and Viresh Patel.
\newblock Tight inequalities among set hitting times in {M}arkov chains.
\newblock {\em Proc. Amer. Math. Soc.}, 142(9):3285--3298, 2014.

\bibitem{Hermonsensitivityofmixing}
Jonathan Hermon.
\newblock On sensitivity of uniform mixing times.
\newblock {\em Ann. Inst. Henri Poincar\'{e} Probab. Stat.}, 54(1):234--248,
  2018.

\bibitem{HermonLacoinPeres}
Jonathan Hermon, Hubert Lacoin, and Yuval Peres.
\newblock Total variation and separation cutoffs are not equivalent and neither
  one implies the other.
\newblock {\em Electron. J. Probab.}, 21:Paper No. 44, 36, 2016.

\bibitem{lazysimple}
Jonathan Hermon, Ben Langer, and Jens Malmquist.
\newblock On quantitative near-bipartiteness: mixing, parity breaking, max-cut
  and eigenvalue gap.
\newblock \emph{In preparation}.

\bibitem{RWAbelian}
Jonathan Hermon and Sam Olesker-Taylor.
\newblock Cutoff for {A}lmost {A}ll {R}andom {W}alks on {A}belian {G}roups.
\newblock https://arxiv.org/abs/2102.02809.

\bibitem{PerlasPaper}
Jonathan Hermon, Allan Sly, and Perla Sousi.
\newblock Universality of cutoff for graphs with an added random matching.
\newblock {\em Ann. Probab.}, 50(1):203--240, 2022.

\bibitem{DynamicalPercolations}
Jonathan Hermon and Perla Sousi.
\newblock A comparison principle for random walk on dynamical percolation.
\newblock {\em The Annals of Probability}, 48(6):2952--2987, 2020.

\bibitem{MixingBook}
David~A. Levin, Yuval Peres, and Elizabeth~L. Wilmer.
\newblock {\em {Markov chains and mixing times}}.
\newblock American Mathematical Society, 2006.

\bibitem{Lezaud}
Pascal Lezaud.
\newblock Chernoff-type bound for finite {M}arkov chains.
\newblock {\em Ann. Appl. Probab.}, 8(3):849--867, 1998.

\bibitem{GWTSpeedHarmMeasure}
Russell Lyons, Robin Pemantle, and Yuval Peres.
\newblock Ergodic theory on {G}alton-{W}atson trees: speed of random walk and
  dimension of harmonic measure.
\newblock {\em Ergodic Theory Dynam. Systems}, 15(3):593--619, 1995.

\bibitem{madrasrandall}
Neal Madras and Dana Randall.
\newblock {Markov chain decomposition for convergence rate analysis}.
\newblock {\em The Annals of Applied Probability}, 12(2):581 -- 606, 2002.

\bibitem{MontenegroTetali}
R.R. Montenegro and P.~Tetali.
\newblock {\em Mathematical Aspects of Mixing Times in Markov Chains}.
\newblock Foundations and trends in theoretical computer science. Now
  Publishers, 2006.

\bibitem{MixHitPaper}
Yuval Peres and Perla Sousi.
\newblock Mixing times are hitting times of large sets.
\newblock {\em Journal of Theoretical Probability}, 28:488--519, 2011.

\end{thebibliography}
\bibliographystyle{plain}



\end{document}